%% file: Revision4.0.tex
\documentclass[12pt, amsfonts, epsfig, reqno, a4paper]{amsart}
\usepackage[utf8]{inputenc}
\usepackage{amsmath}
\usepackage{amssymb}
\usepackage{amsfonts}
\usepackage{tikz}
\usetikzlibrary{decorations.markings}
\usepackage{tikz-cd}
\usepackage{slashed}
\usepackage{braket}
\usepackage{extarrows}
\usepackage[margin=1.3in]{geometry}
\usepackage{textcomp}
\usepackage{mathtools}
\usepackage[title]{appendix}
\usetikzlibrary{backgrounds}
\allowdisplaybreaks

\newcommand{\Ad}{\mathrm{Ad}}
\newcommand{\Sp}{\mathrm{Sp}}

\newcommand{\bP}{\mathbb P}
\newcommand{\sdet}{\mathrm{sdet}}
\newcommand{\bC}{\mathbb C}
\newcommand{\bQ}{\mathbb Q}
\newcommand{\bR}{\mathbb R}
\newcommand{\bZ}{\mathbb Z}

\newcommand{\cO}{\mathcal O}

\theoremstyle{remark}
\newtheorem{remark}{Remark}[section]
\theoremstyle{definition}
\newtheorem{definition}[remark]{Definition}
\theoremstyle{plain}
\newtheorem{theorem}[remark]{Theorem}
\newtheorem{corollary}[remark]{Corollary}
\newtheorem{lemma}{Lemma}[section]
\newtheorem{conjecture}[remark]{Conjecture}
\newtheorem{proposition}[remark]{Proposition}
\newtheorem*{proposition21}{Proposition 2.10}
\newtheorem*{proposition22}{Proposition 2.13}
\newtheorem*{thm4.2}{Theorem 4.2}
\newtheorem*{theorem1.1}{Theorem 1.1}

\begin{document}

\title{Completion of a period map of Hodge type $(1,2,2,1)$}
\author{Chongyao Chen}
\address{Mathematics Department, Duke University,  Durham, NC 27708}
\email{cychen@math.duke.edu}
\date{\today}
\thanks{This material is based upon work supported by the National Science Foundation under Grant No. DMS-1906352.}
\keywords{variation of Hodge structure, mirror symmetry, Kato-Usui completion}
\subjclass[2020]{14J32,14J33,32G20}

\maketitle

\begin{abstract}
    We give a completion of the period map associated to a variation of polarized Hodge structure arising from a 2-dimensional geometric family that has Hodge type $(1,2,2,1)$. This is the second known example of a completion of a period map that is higher than one dimensional with non-hermitian symmetric Mumford-Tate domain. The completion technique we use is the theory of Kato-Usui spaces, while the calculation of monodromy matrices is done with the help of mirror symmetry.
\end{abstract}

\setcounter{section}{0}
\setcounter{tocdepth}{2}
\tableofcontents

\section{Introduction}

In this paper, we consider a family of smooth Calabi-Yau 3-folds with $h^{2,1}=2$. The family is constructed as the mirror to the generic (resolved) Calabi-Yau hypersurface in $\bP[1,1,2,2,2]$, and thus is sometimes called the \textit{mirror octic}. Let $\Phi:B\rightarrow \Gamma\backslash D$ be the associated period map. 

\begin{theorem}
There is a finite cover $B'\xrightarrow{}B$ with the property that the lifted period map $\Phi':B'\rightarrow \Gamma'\backslash D$ admits a Kato-Usui type completion in the sense of Theorem 1.3. In particular, points at infinity parametrize nilpotent orbits. Moreover, at a generic point of $B$ (or $B'$) the real special Mumford-Tate group is $Sp(6,\bR)$.
\end{theorem}

As far as we know, this is the second known example of the completion of a period map with the properties that $\dim \Phi(B)\geq 2$ and the Mumford-Tate domain is not hermitian symmetric. The first known example is given by Deng \cite{deng2021extension}. More recently, Deng gives a partial compactification of period maps of Calabi-Yau type along boundary strata with type I or type IV degeneracy \cite{deng2023space}.  It is work in progress to show that every two-parameter period map can be completed using nilpotent orbits \cite{deng2023space2}.

\subsection{Summary of previous results in compactification problem}

Let $D$ be a period domain parametrizing $Q$-polarized Hodge structures, and $\Gamma$ be a discrete subgroup of $ Aut_\bZ(D)=Aut_\bZ(Q)$. As proposed by Griffiths  \cite{griffiths1970periods}, an important problem in this setting is giving a horizontal completion of $\Gamma\backslash D$. For any pair $(D,\Gamma)$, this would provide, via an extended period map, Hodge theoretical information about the boundary of a moduli space. Since (or even before) the problem was proposed, a lot of progress has been made.

When $D$ is a hermitian symmetric domain, the Hodge numbers are $(g,g)$ and $(1,g,1)$, $g\geq 1$.  Geometrically, they correspond to the moduli spaces of principally polarized abelian
varieties, and of polarized K3 surfaces ($g = 19$), respectively. In this circumstance, one has the Bailey-Borel compactification $(\Gamma\setminus D)^*$ \cite{baily1966compactification}. The Bailey-Borel compactification has the advantage of being both canonical and projective, but the disadvantage of being singular. Mumford et al introduced the toroidal normalization $\overline{\Gamma\setminus D}^\Sigma$ of $(\Gamma\setminus D)^*$ \cite{ash1975smooth}, where the points at infinity can be characterized as $\Gamma$-equivalence classes of nilpotent orbits. By Schmid's nilpotent orbit theorem  \cite{schmid1973variation}, these nilpotent orbits asymptotically approximate period maps at infinity. This construction requires a choice of a fan $\Sigma$, and there is always a suitable choice that can make $\overline{\Gamma\setminus D}^\Sigma$ projective with at most finite quotient singularities.

When $D$ is not hermitian one does not expect to have a good compactification of $\Gamma\backslash D$. However, one may still try to complete the period map. A topological     Satake-Bailey-Borel type completion of a period map is given in \cite{green2014extremal,green2017period}. (The construction is conjectured to be algebraic.) Kato-Usui proposed a generalization of the toroidal compactification \cite{kato2008classifying}. 
Consider a period domain $D$ and an arithmetic subgroup $\Gamma\leq Aut_\bZ(D)$, and a choice of polyhedral fan $\Sigma\subset End_\bQ(D)$ consisting of nilpotent operators. Schmid's nilpotent theorem suggests a horizontal completion of the period domain by adding nilpotent orbits 
    $$
    D_\Sigma = \{(\sigma,Z_\sigma):\sigma\in\Sigma,Z_\sigma \text{is a nilpotent orbit}\}/\{\text{reparametrization}\}
    $$
 If $(D,\Gamma,\Sigma)$ satisfies certain properties (see below), then Kato-Usui showed $\Gamma\setminus D_{\Sigma}$ admits a Hausdorff logarithmic analytic structure. After that, they conjectured the existence of a complete fan $\Sigma$ and a complete weak fan.  If either of these fans exists, then the period map $\Phi:B\rightarrow\Gamma\setminus D$ can be completed. (For a more precise statement, see Theorem \ref{KTtheorem} below.) For hermitian symmetric period domains, the existence of a complete fan is shown in \cite{kato2008classifying}. However, for non-classical cases counter-examples have been constructed in \cite{watanabe2008counterexample, deng2021extension}. For more details about Kato-Usui's theory and definitions for the objects above, we refer to Chapter 0 of \cite{kato2008classifying}. For recent progress on using Kato-Usui's theory to complete period maps see \cite{deng2023space,deng2023space2}.

\subsection{Outline of the completion approach}\label{olapp}

While the complete fan conjecture fails in general, Kato-Usui's theory can still be applied to complete a single period map as Deng did in \cite{deng2021extension}. Recall
\begin{definition}
    We say $\Sigma$ is \emph{strongly $\Gamma$-compatible} if 
    \begin{enumerate}
        \item $\Sigma$ is closed under the  action of $Ad(\Gamma)$.
        \item For any $\sigma\in\Sigma$, $\sigma$ is generated by $\log(\Gamma(\sigma))$, where $\Gamma(\sigma): = \Gamma\cap\exp(\sigma)$.
    \end{enumerate}
    \end{definition}
    The main theorem we will be using is
    \begin{theorem}[Kato-Usui \cite{kato2008classifying}]\label{KTtheorem}
    For a pair $(\Sigma,\Gamma)$ with  $\Gamma$ neat, if $\Sigma$ is strongly $\Gamma$-compatible, then 
    \begin{enumerate}
        \item The space $\Gamma\setminus D_\Sigma$ admits a structure of logarithmic manifold.
        \item Given a polarized VHS over a smooth quasi-projective base, let $\mathcal P:B\rightarrow \Gamma\setminus D$ be the period map. Fix a smooth projective completion $B\subset\bar B$ with $\bar B\setminus B$ a normal crossing divisor of $\bar B$. If the nilpotent cones associated to each boundary stratum lie   in $\Sigma$, then the period map extends to $\overline{\mathcal P}: \bar B\rightarrow \Gamma\backslash D_\Sigma$
    \end{enumerate}
    \end{theorem}

For the VHS $(L,\mathcal F^\bullet,\nabla,\mathcal Q)$ associated to the family mirror quintic, we can conduct the following process to determine whether its associated period map can be completed via Theorem \ref{KTtheorem}. 

\begin{enumerate}
    \item Fix a completion  $B\xhookrightarrow{}\tilde{B}$ such that $\partial B:=\tilde B\setminus B$ consists of simple normal crossing divisors in $\tilde B$.
    \item Choose a multivalued frame of $\mathcal F^0 = L\otimes_\bQ \mathcal O_B$ and represent the Gauss-Manin connection $\nabla$ in this frame.
    \item Fix a base point $b\in B$. Then at each normal crossing intersection point of $\tilde B$, we calculate the two local monodromy matrices with respect to the local basis, then pull them back to $b$ by parallel transportation via $\nabla$. 
    \item Find an integral symplectic basis of $V:=L_b\otimes_\bQ\bC$ at $b$, and represent all pairs of monodromy operators as matrices with respect to this basis. Some finite power of the monodromy operator is unipotent.  We denote the fan formed by the nilpotent logarithms of these pairs as $\Sigma^{'}\subset End_\bQ(V)$.
    \item By passing to a finite {\'e}tale cover $U\rightarrow B$, and taking $ Y\rightarrow \bar B$ the normalization of $\tilde B$ in the function field $k(U)$, we can assume $\Gamma$ is neat. As discussed in Remark \ref{ptfe}, after passing to $U$, and resolving any singularity in $Y$, the monodromy at infinity is unipotent, and the associated nilpotent cones are finite subdivisions of $\Sigma'$.
    \item Check whether there exists a $m\in\bZ_+$, such that, $\Sigma:=Ad(\Gamma_m)\cdot\Sigma'$ form a fan, i.e., there are no infinite interior intersections between nilpotent cones in $\Sigma'$ under the adjoint action.\footnote{Instead of checking $\Sigma$ is closed under $Ad(\Gamma)$ action, with $\Gamma\leq Aut_\bZ(V)$ the monodromy group, it is normally easier to check for some finite index subgroups of $Aut_\bZ(V)$. Although this will force us to go to a finite cover.} Note $\Gamma_m: = \{g\in Aut_\bZ(V):g\equiv \mathrm{Id}\mod m\}$ is a finite index subgroup  of $Aut_\bZ(V)$.
    \item Then we can conclude, that the period image associated to a finite {\'e}tale cover of the initial VHS can be completed via Theorem \ref{KTtheorem}. 
\end{enumerate}
We note this process can be generalized to arbitrary $VHS$ without difficulties.

\subsection{Overview}

The mirror symmetry for octic 3-fold has been studied thoroughly in \cite{candelas1994mirror1} and \cite{cox1999mirror}, and will be an important ingredient for our calculation. More precisely, one way to conduct step (4) above is taking $b$ near the \textit{maximal unipotent monodromy (MUM)} point, then an integral symplectic basis can be found via mirror symmetry \cite{hosono2000local,hosono2004central,hosono2014determinantal}. We will recall the geometry and mirror symmetry for the family in Section \ref{secgeofam}. Meanwhile, we will complete the steps (1-2) and find the integral symplectic basis of step (4). We note that (genus 0) mirror symmetry has been established for toric complete intersections in \cite{lian1997mirror} and \cite{givental1998mirror}. However, we will not rely on this result; instead, we will prove the relevant consequences, namely Proposition \ref{lemhosin} and Proposition \ref{prop2.2sec2}, in the appendix.

In Section \ref{seccal} we calculate the monodromy operator numerically. The idea is similar to that of \cite{hosono2014determinantal}. However, we will use a different approximation algorithm, the linear approximation of parallel transportation. Next, we find the transformation matrix from the complex basis to an integral symplectic basis. The result is made precise by showing within the numerical accuracy the mentioned transformation matrix is the only integral matrix. Then passing to some covering space, one can make all the monodromy matrices to be unipotent.
After that, we glue the log-monodromy nilpotent cones together to form $\Sigma''$ and represent it in the integral symplectic basis, which is denoted as $\Sigma'$. These correspond to steps (3-5).

Finally, in Section \ref{Secinint} we show the $\Gamma_6$-adjoint orbit of $\Sigma'$ has no infinite interior intersection which completes step (6). To summarize, the first 6 steps result in the following theorem
\begin{thm4.2}
There is a finite cover $B'\xrightarrow{}B$ such that the nilpotent fan $\Sigma'$ associated to the lifted period map $\Phi':B'\rightarrow \Gamma'\backslash D$ is strongly compatible with the monodromy group.
\end{thm4.2}
Then combined with Theorem \ref{KTtheorem}, we arrive at the first statement of our main result in Theorem 1.1. The Mumford-Tate genericity follows from Corollary \ref{coMT}, whose proof is essentially the same as Theorem 5.4 of \cite{deng2021extension}.

The notation in this paper will mainly follow that in \cite{candelas1994mirror1}, \cite{cox1999mirror} and \cite{deng2021extension}. For a quick introduction to (variation of) Hodge and mixed Hodge structures, we refer to \cite{robles2016degenerations}. For the details about the toric geometry we use, we refer to \cite{cox1999mirror,cox2011toric}. For the classification of limiting mixed Hodge structure, we follow the notation in \cite{kerr2017polarized}.

\subsection*{Acknowledgement} The author thanks his advisor Colleen Robles for the suggestion of the problem, and numerous instructive conversations, and also thanks Paul Aspinwall and Haohua Deng for helpful conversations.

\section{The geometry of the families}\label{secgeofam}
In this section, we review the geometry of the families that we will study and fix the notation. For details, we refer to \cite{candelas1994mirror1} and \cite{cox1999mirror}.

We are interested in the smooth 2-parameter family $\pi:X\rightarrow B$ that mirrors\footnote{As we will see, the mirror construction includes a natural extension of the family to a toric compactification of $B$, so the smooth family actually refers to the smooth locus of this family with degeneration.} to the resolved degree 8 hypersurfaces in $P[1,1,2,2,2]$.  Fix a  resolved degree 8 hypersurface $\hat X_0$ in $P[1,1,2,2,2]$, while also fix a point $p\in B$, and let $X_p:=\pi^{-1}(p)$. Both $\hat X_0$ and $X_p$ are smooth Calabi-Yau 3-folds, with Hodge numbers $h^{2,1}(X_p) = h^{1,1}(\hat X_0) = 2$.

\begin{definition}[Calabi-Yau 3-fold]\label{defcy3f}
A complex projective 3-fold $\hat X_0$ with at most Gorenstein singularities is a \textit{Calabi-Yau} 3-fold, if it has trivial canonical line bundle and $$H^1(\hat X_0,\mathcal O_{\hat X_0}) = H^2(\hat X_0,\mathcal O_{\hat X_0})=0.$$ In particular, we have vanishing Hodge numbers $h^{1,0} = h^{2,0} = 0$. 
\end{definition}

Mirror symmetry relates the K{\"a}hler geometry and K-theory of $\hat X_0$ and the complex geometry for the family $( X, B,  \pi)$, which are called A,B-models respectively.

\subsection{A-model geometry}\label{secamod}

We briefly review the K{\"a}hler geometry and K-theory of $\hat X_0$.

\subsubsection{Toric geometry for $\hat X_0$}\label{sec111}
We start with $\tilde X_0 := \{f_0=0\}$, where $f_0 = x_1^8+x_2^8+x_3^4+x_4^4+x_5^4$. Then $\hat X_0$ is resolved from $\tilde X_0$ by blowing up along the singular locus $x_1 = x_2 = 0$, it is sometimes called \textit{the resolved octic}. Denote the resolution by $\phi: \hat X_0  \rightarrow \tilde X_0$.
The polyhedral fan $\Sigma$ for the toric variety $ X_{\Sigma} = P[1,1,2,2,2]$ has 1-dimensional cones generated by
\begin{equation}\label{sigma1A}
\Sigma(1) = \{(-1,-2,-2,-2),(1,0,0,0),(0,1,0,0),(0,0,1,0),(0,0,0,1)\}.
\end{equation}
We denote the divisors corresponding to each of these vectors as $D_i,i=1,\dots,5$. The aforementioned blow-up can be done on the level of ambient space by subdividing the fan, where one adds the 1-dimensional cone generated by $(0,-1,-1,-1)$. This gives an exceptional divisor $D_6$. We denote the enlarged set of generators as $\Sigma'(1)$. The higher dimensional cones are modified correspondingly to make the fan simplicial, which is denoted by $\Sigma'$. This is the maximal projective simplicial subdivision.

\subsubsection{Hodge, Picard, and intersection numbers}
For $\hat X_0$, the Hodge numbers $h^{1,1}$ and $h^{1,2}$ can be determined by Theorem 4.3.1 in \cite{batyrev1993dual} (or formula (4.7), (4.8) in \cite{cox1999mirror}) via toric data, while $h^{1,0} = h^{2,0}=0$ can be derived from the Lefschetz hyperplane theorem. The remaining Hodge numbers are determined by the symmetry $h^{p,q} = h^{q,p}$ and the Serre duality. The resulting Hodge diamond is
\[
\left.\begin{array}{ccccccc}  &  &  &  h^{3,3} &  &  &  \\  &  & h^{3,2} &  & h^{2,3} &  &  \\ & h^{3,1} &  & h^{2,2} &  & h^{1,3} &  \\ h^{3,0}&  & h^{2,1} &  & h^{1,2}  &  & h^{0,3} \\& h^{2,0} &  & h^{1,1} &  & h^{0,2} &  \\  &  & h^{1,0} &  & h^{0,1} &  &  \\  &  &  &  h^{0,0} &  &  &  \end{array}\right. = \left.\begin{array}{ccccccc}  &  &  &  1 &  &  &  \\  &  & 0 &  & 0 &  &  \\ & 0 &  & 2 &  & 0 &  \\ 1&  & 86 &  &86  &  & 1 \\& 0 &  & 2 &  & 0 &  \\  &  & 0 &  & 0 &  &  \\  &  &  &  1 &  &  &  \end{array}\right..
\]
Since $h^{0,1}(\hat X_0)=h^{0,2}(\hat X_0)=0$, from the long exact sequence associated to the exponential sequence
\begin{equation}\label{esles}
\begin{split}
0 = H^{0,1} \cong  H^1(\hat X_0,\mathcal O_{\hat X_0}) &\rightarrow H^1(\hat X_0,\mathcal O^*_{\hat X_0})\\
&\xrightarrow{c_1} H^2(\hat X_0,{\bZ}) \rightarrow H^2(\hat X_0,\mathcal O_{\hat X_0}) \cong H^{0,2} = 0,
\end{split}
\end{equation}
we have $\mathrm{Pic}(\hat X_0) :=H^1(\hat X_0,\mathcal O^*_{\hat X_0})\cong H^2(\hat X_0,{\bZ})$; that is, $c_1$ gives an  isomorphism of abelian groups.
According to Section 2 of \cite{candelas1994mirror1},  $\mathrm{Pic}(\hat X_0) $ is generated by $\mathcal L = \phi^*\mathcal{O}_{\tilde X_q}(1)$ and $ \mathcal H = \phi^*\mathcal{O}_{\tilde X_q}(2)$. Then
$H^{2}(\hat X_0,\bZ) = \bZ[L,H]$, where $L = c_1(\mathcal L) , H = c_1(\mathcal H)$. We can also write $L=i^*D_1$ and $H=i^*D_3$ as in \cite{cox1999mirror}, where $i$ is the inclusion of $\hat X_0$ into the ambient toric variety. Moreover, we have $L^2 = 0$, and the intersection numbers are 
\begin{equation}\label{intnum}
H^3 = 8,\quad H^2L = 4,\quad HL^2 = L^3 =0.
\end{equation}

Let $l,h$ be the Poincar{\'e} duals of $L,H$; that is we have
$$
L\cdot l =H\cdot h =1,\quad H\cdot l = L\cdot h =0.
$$
Then using the intersection numbers \eqref{intnum}, we have $ l = \frac{1}{4}H^2-\frac{1}{2}HL$, $h = \frac{1}{4}HL$. The second Chern class can be shown \cite{candelas1994mirror1} to be $c_2(\hat X_0) = 24l+56h$. The Todd class is then $td(\hat X_0) = 1+\frac{c_2}{12} = 1+2l+\frac{14}{3}h$. Finally, the K{\"a}hler cone of $X_0$ is $\bR_{\geq 0}L+\bR_{\geq 0}H$.

\subsubsection{$\Gamma$-integral structure and A-polarization}

Recall there is a ring isomomorphism $ch:K_0(\hat X_0)\otimes_\bZ \bQ \rightarrow H^{even}(\hat X_0,\bQ)$, where $K_0(\hat X_0)$ is the Grothendieck ring of vector bundles on $\hat X_0$ and $ch$ is the Chern character. 
In particular, $ch$ induces an integral structure on $H^{even}(\hat X_0,\bC)$. This integral structure is called the \textit{$ \hat{\Gamma}$-integral structure} in \cite{iritani2009integral}, because one can replace the Todd class in (\ref{releuler}) below, by a square of the Gamma class.

On $K_0(\hat X_0)\otimes_\bZ \bQ$, the \textit{relative Euler characteristic} is defined as $$\chi(E_1,E_2)= \sum_{i=0}^3(-1)^i \mathrm{dim}\mathrm{Ext}^i(E_1,E_2).$$ By the Hirzebruch–Riemann–Roch formula, we have
\begin{equation}\label{releuler}
\begin{split}
    \chi(E_1,E_2) & = \int_{\hat X_0}ch(E_1^*)\cdot ch(E_2)\cdot td(\hat X_0)\\
    & = \int_{\hat X_0}ch(E_1^*)\cdot ch(E_2)\cdot \left(1+\frac{c_2}{12}\right).
\end{split}
\end{equation}

The category of algebraic vector bundles naturally embeds into the category of coherent sheaves as a full subcategory.
By the Serre (coherent) duality (\cite{huybrechts2006fourier}) and the Calabi-Yau condition $\omega_{\hat X_0} \cong \cO_{\hat X_0}$, we have
\[
\begin{split}
\mathrm{Ext}^i(E_1,E_2) & \cong \mathrm{Hom}_{\mathcal D^b(\mathbf{Coh})}(E_1,E_2[i]) \cong \mathrm{Hom}_{\mathcal D^b(\mathbf{Coh})}(E_2,E_1\otimes \omega_{\hat X_0}[3-i])^*\\
& =\mathrm{Hom}_{\mathcal D^b(\mathbf{Coh})}(E_2,E_1[3-i])^* \cong \mathrm{Ext}^{3-i}(E_2,E_1)^*,
\end{split}
\]
where $\mathcal D^b(\mathbf{Coh})$ is the bounded derived category of coherent sheaves on $\hat X_0$. This shows that the Euler characteristic is anti-symmetric. Moreover, since $ch$ is a ring isomorphism, $1+\frac{c_2}{12}$ is invertible in $H^{even}(\hat X_0,\bC)$, and the intersection pairing on cohomology is non-degenerate, we have
$\chi$ is non-degenerate. Therefore, $\chi$ linearly extends to a symplectic form on  $K(\hat X_0)\otimes_{\bZ} \bC$.

By the ring isomorphism $ch$, $\chi$ induces a symplectic structure on $H^{even}(\hat X_0,\bC)$, which together with the $\Gamma$-integral structure gives a polarization of $H^{even}(\hat X_0,\bC)$, which we will call the \textit{A-polarization}.

\subsection{B-model geometry}

For B-model, we will study the polarized variation of Hodge structure (VHS) for the family $(\pi,X,B)$, where $B$ is the moduli space of complex structures\footnote{in general, $B$ is the so-called \textit{simplified moduli space}, yet, in this case, it is equals to the moduli space of complex structure, see Chapter 6 of \cite{cox1999mirror}.} for the smooth Calabi-Yau hypersurface in the toric variety polar dual to $\bP[1,1,2,2,2]$. We denote the total space of the universal family as $X$, and the projection as $\pi$.

A standard toric construction gives a natural compactification  $B\xhookrightarrow{\imath}\bar B$ and extends the family, i.e., we have the following Cartesian diagram

\begin{center}
     \begin{tikzcd}
         X\arrow[hookrightarrow]{r}\dar[swap,"\pi"] \drar[phantom, "\square"] & \bar X\dar["\bar \pi"] \\%
         B\arrow[hookrightarrow]{r}[swap]{\imath} & \bar B
    \end{tikzcd}.
\end{center}
However, $\partial B:=\bar B\setminus B$ is not a normal crossing divisor in $\bar B$. A sequence of blow-ups yield
\begin{center}
     \begin{tikzcd}
         X\arrow[hookrightarrow]{r}\dar[swap,"\pi"] \drar[phantom, "\square"] & \bar X\dar["\bar \pi"]\arrow[hookrightarrow]{r} \drar[phantom, "\square"]& \tilde X \dar["\tilde{\pi}"] \\%
         B\arrow[hookrightarrow]{r}[swap]{\imath} & \bar B  \arrow[hookrightarrow]{r}[swap]{\jmath} & \tilde B,
    \end{tikzcd},
\end{center}
where $\jmath$ is the strict transform left inverse to the blow-up, and $\tilde B\setminus \jmath\circ\imath(B)$ is a simple normal crossing divisor. We now delve into the details.

\subsubsection{Geometry of a generic fiber}
For any point $p\in B$, $X_p$ is a smooth Calabi-Yau hypersurface of the toric variety $\bP_{\Delta^\circ}$ associated to the reflexive polytope $\Delta^\circ$, where $\Delta^\circ$ is the convex hull of $\Sigma(1)$ in (\ref{sigma1A}). On the other hand, in terms of polyhedra fan, the 1-dimensional cones are generated by
\begin{equation*}
\begin{split}
\Sigma^\circ(1) = \{(-1,-1,-1,-1),&(7,-1,-1,-1),(-1,3,-1,-1),\\
&(-1,-1,3,1),(-1,-1,-1,3)\}.
\end{split}
\end{equation*}

Denote the convex hull of these points as $\Delta$, then $\bP[1,1,2,2,2] = X_\Sigma \cong \bP_{\Delta}$, and $\Delta^\circ$ is the polar dual of $\Delta$. For more details, we refer to \cite{cox1999mirror}. The Hodge diamond for $X_p$ is thus the mirror reflection for that of $\hat X_0$:

\[
\left.\begin{array}{ccccccc}  &  &  &  h^{3,3} &  &  &  \\  &  & h^{3,2} &  & h^{2,3} &  &  \\ & h^{3,1} &  & h^{2,2} &  & h^{1,3} &  \\ h^{3,0}&  & h^{2,1} &  & h^{1,2}  &  & h^{0,3} \\& h^{2,0} &  & h^{1,1} &  & h^{0,2} &  \\  &  & h^{1,0} &  & h^{0,1} &  &  \\  &  &  &  h^{0,0} &  &  &  \end{array}\right. = \left.\begin{array}{ccccccc}  &  &  &  1 &  &  &  \\  &  & 0 &  & 0 &  &  \\ & 0 &  & 86 &  & 0 &  \\ 1&  & 2 &  &2  &  & 1 \\& 0 &  & 86 &  & 0 &  \\  &  & 0 &  & 0 &  &  \\  &  &  &  1 &  &  &  \end{array}\right.
\]

\subsubsection{Toric geometry of the moduli}

The moduli space $B$ has a natural compactification into a toric variety $\bar B$. The polyhedra fan for $\bar B$ is called the secondary fan and can be deduced combinatorically from $\Sigma'(1)$. To be more precise, we first find the generator for the linear relations between the $6$ vectors in $\Sigma'(1)$ and put them in a $2\times 6$ matrix. Then the fan generated by the columns of this matrix is precisely the secondary fan. More details can be found in \cite{aspinwall1994calabi,Gelfand1994}.
In our case, the secondary fan is generated by the labeled blue arrows in Figure \ref{toricdiagforB}.

\begin{figure}
\begin{center}
\begin{tikzpicture}[gridded]
     \draw[->,blue] (0, 0) -- (0, 1) node[above]   {$(0,1)$};
     \draw[->,blue] (0, 0) -- (-1, 0) node[below] {$(-1,0)$};
    \draw[->,blue] (0, 0) -- (1, -2) node[below] {$(1,-2)$};
    \draw[->,blue] (0, 0) -- (1, 0) node[above] {$(1,0)$};
    \draw[->,red] (0, 0) -- (0,-1);
    \draw[->,red] (0, 0) -- (1, -1);
  \end{tikzpicture}
 \caption{\textbf{The secondary fan}. Blue vectors generate the polyhedra fan for $\bar B$, while adding the red vectors desingularizes $\bar B$ by toric blow ups.}
 \label{toricdiagforB}
\end{center}
\end{figure}

More concretely, denote the coordinates associated to the one-dimensional cones $(1,0),(0,1),(-1,0),(1,-2)$ as $w_1,w_2,w_3,w_4$. Then the toric variety is 
$$
\bar B = (\bC^4\setminus (Z(w_1,w_3)\cup Z(w_2,w_4))/(\bC^*)^2,
$$
where the $(\bC^*)^2$-action is given by
\begin{equation}\label{gitmoduli}
    \begin{split}
        (\bC^*)^2\times \bC^4&\rightarrow \bC^4\\
        (t_1,t_2)\cdot (w_1,w_2,w_3,w_4)&\rightarrow (t_1w_1,t_2^2w_2,t_1t_2w_3,t_2w_4).
    \end{split}
\end{equation}

\subsubsection{Geometry of the total space}
As mentioned at the beginning of this subsection, we have a natural embedding $X\xhookrightarrow{}\bar X$.
In the affine chart $\mathrm{Spec}\,\bC[z_1,z_2]$, defined via
\[
\begin{split}
   \mathrm{Spec}\,\bC[z_1,z_2] &\rightarrow \bar B\\
   (z_1,z_2)&\rightarrow [z_1:z_2:1:1],
\end{split}
\]
the total space $\bar X$ can be written as a quotient $\{f=0\}/(\bZ_4)^3$ \cite{cox1999mirror}, where
\begin{equation}\label{defeq}
f := 2^{-2}z_2x_1^8+x_2^8+2^{-8}z_1x_3^4+x_4^4+x_5^4+x_1x_2x_3x_4x_5+x_1^4x_2^4.
\end{equation}
We use $[w_1:w_2:w_3:w_4]$ to denote the multi-graded homogeneous coordinates.

\begin{remark}
We have scaled the parameters $z_1,z_2$ used in \cite{cox1999mirror} by constants $2^{8}$ and $2^{2}$, which can also be viewed as a rescaling of coordinates $x_i$, $i=1,\cdots,5$ via the torus action. This rescaling will simplify the expressions later on (e.g., defining equations in (\ref{disclocus})) and increase the numerical accuracy when we compute the monodromy operators. Moreover, the defining equation used in \cite{candelas1994mirror1} is
$$
f' =  x_1^8+x_2^8+x_3^4+x_4^4+x_5^4-8\psi x_1x_2x_3x_4x_5-2\phi x_1^4x_2^4,
$$
which differs from (\ref{defeq}) by a change of variables
$$
z_1 = -\frac{\phi}{8\psi^4}=-\frac{\zeta}{8\eta}= -\frac{\eta}{8\xi},\quad z_2 = \frac{1}{\phi^2}=\frac{\tau}{\zeta}.
$$
In particular the moduli space of (\ref{defeq}) becomes a blowup of those in \cite{candelas1994mirror1}.
\end{remark}

\subsubsection{Discriminant loci}\label{secdiscloci} Our moduli space $B$ will be the smooth locus in $\bar B$, i.e., the locus in which the Calabi-Yau 3-fold is smooth. The complement is the discriminant locus. One way to find the discriminant locus is by analyzing the Jacobian of (\ref{defeq}) \cite{candelas1994mirror1}.

Another approach is using the theory of GKZ systems \cite{Gelfand1994} as in \cite{cox1999mirror}. The discriminant locus of the family is the same as the singular locus for the GKZ system associated to $(\Sigma'(1)\cup \{0\}) \times 1$. By the standard theory of GKZ systems, the singular locus is contained inside the union of all the toric boundaries and the \textit{principal A-determinant}.

Either way, the discriminant loci is contained in the union following divisors\footnote{The defining equation of $C_{con}$ given in Example 6.1.4.1 of \cite{cox1999mirror} has a typo: the last term $(1-256z_1)^2$ should be $(1-256z_2)^2$. Also, not necessarily all of these divisors are part of the discriminant locus.}
\begin{equation}\label{disclocus}
    \begin{array}{ll}
        1.\, C_{\infty} := D_{(0,1)} = Z(z_2);& 2.\, D_{(1,0)} = Z(z_1) ;\\
        3.\, C_0:=D_{(-1,0)}; & 4.\, D_{(1,-2)};\\
        5.\, C_1 = Z(z_2-1); & 6.\, C_{con} = Z(z_1^2z_2-(1-z_1)^2);
    \end{array}
\end{equation}

The names for the divisors are taken from \cite{candelas1994mirror1}, with discrepancies due to the fact that our toric diagram is a rotation of theirs. The discriminant locus in the real subspace of the affine chart  $(z_1,z_2)$ is shown in Figure \ref{discrilociaff}. These affine coordinates play a distinguished role in the mirror symmetry for Calabi-Yau complete intersections in toric variety \cite{hosono1996gkz,givental1996equivariant,givental1998mirror}. Note the divisors $D_{(-1,0)} = Z(w_3)$, $D_{(1,-2)}=Z(w_4)$ cannot be seen in this affine chart.

\begin{figure}
\begin{center}
\begin{tikzpicture}
  \draw[->] (-3, 0) -- (4.5, 0) node[right] {$z_1 (D_{(0,1)}=C_\infty)$};
  \draw[->] (0, -1) -- (0, 4.8) node[above] {$z_2 (D_{(1,0)})$};
  \draw[scale=1, domain=-3:-0.2, smooth, variable=\x, blue] plot ({\x}, {1/10*(1-1/\x)^2+1});
  \draw[scale=1, domain= 0.33:4.2, smooth, variable=\x, blue] plot ({\x}, {11/10*(1-1/\x)^2}) node[right] {$C_{con}$};
  \draw[scale=1, domain=-3:4.5, smooth, variable=\x, red] plot ({\x}, {11/10}) node[right] {$C_{1}$};
   \draw [fill] (0,11/10) circle [radius=0.05];
 \node [below left] at (0,11/10) {\footnotesize{$(0,1)$}};
 \node [above] at (-3, 12/10) {\footnotesize{$\color{blue}{C_{con}}$}};
\draw [fill] (1,0) circle [radius=0.05];
 \node [below] at (1,0) {\footnotesize{$(1,0)$}};
 \draw [fill] (1/2,11/10) circle [radius=0.05];
 \node [above right] at (1/2,11/10) {\footnotesize{$(\frac{1}{2},1)$}};
 \draw [fill] (0,0) circle [radius=0.05];
 \node [below left] at (0,0) {\footnotesize{$MUM$}};
\end{tikzpicture}
\end{center}

\caption{\textbf{Discriminant locus.} All curves, including the axes, in the figure represent the discriminant locus in the real subspace of the affine chart  $(z_1,z_2)$. The origin is the maximal unipotent point.}
\label{discrilociaff}
\end{figure}
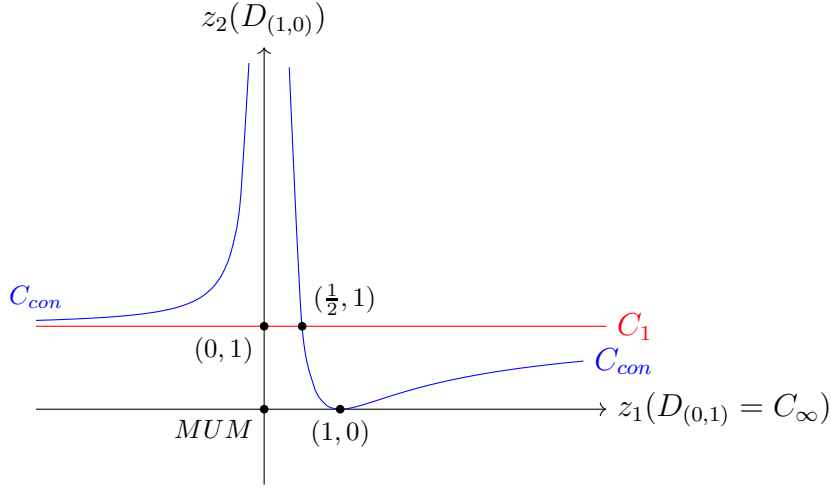

The toric variety $\bar B$ is singular at two points $D_{(-1,0)}\cap D_{(1,-2)}$ and $D_{(1,0)}\cap D_{(1,-2)}$. Adding the red vectors $(0,-1)$ and $(1,-1)$ in Figure \ref{toricdiagforB} corresponds to the resolution of the moduli space by toric blow-ups. Furthermore, if we want the discriminant locus to be a normal crossing divisor, we need to further blow up once for the triple intersection $D_{(1,0)}\cap C_1\cap C_{con}$, and twice for the 2-tangency at $C_{con}\cap D_{(0,1)}$. These blow-ups are not toric, and we denote the resulting divisors as $E_0$ and $E_1,E_2$. To sum up, in addition to those in (\ref{disclocus}), we have 5 more boundary divisors

\begin{equation}\label{disclocus2}
    \begin{array}{lll}
         7.\,D_{(1,-1)} ;& 8.\,D_{(0,-1)} ;&\\
         9.\, E_0;& 10.\, E_1;
         &11.\, E_2;
    \end{array}
\end{equation}
We denote the resulting smooth compact moduli space 
as $bl:\tilde B\rightarrow \bar B$. Then the family we mentioned at the beginning of this section is  $(\tilde\pi,\tilde X:=X\times_{\bar B}\tilde B, \tilde B)$. This completes the step (1) in Section \ref{olapp}.

\subsubsection{Variation of polarized Hodge structure}

Now we will discuss the Hodge theory for the smooth family $(\pi,X,B)$. Recall
\begin{definition}
A $\bZ$-\textit{variation of Hodge structure} (VHS) of weight $n$ on a local system $L\in \mathbf{Mod}(B,\bZ)$ over a complex manifold $B$ consists of a (finite, decreasing) Hodge filtration of holomorphic vector bundle
$$
0\subset \mathcal F^n\subset \mathcal F^{n-1} \subset \cdots \subset \mathcal F^1\subset \mathcal F^0:=L\otimes_\bZ\mathcal O_B,
$$
that induces a pure $\bZ$-Hodge structure of weight $n$ on each stalk of $L_\bC:=L\otimes_\bZ\bC$ and satisfies the \textit{Griffiths transversality (IPR)}
$$
\nabla\mathcal F^k \subset \mathcal F^{k-1}\otimes \Omega^1_B,
$$
where $(\mathcal F^0,\nabla)$ is the image of $L$ under Riemann-Hilbert correspondence. The holomorphic flat connection $\nabla$ is called the \textit{Gauss-Manin connection}.  Moreover, the VHS is \textit{$\mathcal Q$-polarized} if there is a $\nabla$-flat 2-form $\mathcal Q$ on $\mathcal F$, such that it induces a $Q$-polarization on each stalk.
\end{definition}
The \textit{Hodge numbers} is defined to be those on each stalk, i.e., it is a length $n+1$ vector $(h^{n,0},h^{n-1,1}\cdots,h^{0,n})$ with $h^{k,n-k}:=\mathrm{rank}\,\mathcal F^k- \mathrm{rank}\,\mathcal F^{k-1}$.
For details, we refer to \cite{robles2016degenerations,voisin2002hodge}.

In our example, we have a weight 3 VHS with Hodge numbers (1,2,2,1). The local system underlying the VHS is given by $L_\bZ:=R^3\pi_*\bZ_X$, and the polarization $\mathcal Q$ is given by the intersection pairing on each stalk. By the local Torelli theorem for Calabi-Yau manifolds (c.f., Chapter 9 of \cite{voisin2002hodge}), the Hodge filtration $\mathcal F^\bullet$ is determined by the Gauss-Manin connection.

To complete step (2) in Section \ref{olapp}, we need to find a multi-valued basis for $\mathcal F^0$ and represent the Gauss-Manin connection in this basis. This will occupy us for the rest of this subsection.

\subsubsection{The GKZ system and Picard-Fuchs equations}

To a toric variety, there is associated a holonomic $\mathcal D$-module \cite{Gelfand1994,hosono1996gkz}, called the GKZ system. In our case, we get a $\mathcal D_{\bar B}$-module $\mathcal M_{\bar B}$. Restricting to the smooth locus, the periods associated to a canonical section $\omega_1$ of $\pi_*\Omega^3_{X/B}$ are solutions of $\mathcal M_{\hat B}$. However, they are not the only solutions, i.e., the GKZ system is properly contained inside the Picard-Fuchs system \cite{cox1999mirror}. A larger system $\mathcal I$ that contains all the Picard-Fuchs equations can be calculated by the Griffiths-Dwork method or a factorization of the GKZ system \cite{hosono1995mirror}. For the definition and properties of $\mathcal D$-modules, we refer to \cite{hotta2007d}.

More precisely, the canonical section is $\omega_1:=\mathrm{Res}\frac{\Omega_0}{f}$, with

$$
\Omega_0 = x_1 \widehat{dx_1}+x_2\widehat{dx_2}+ 2x_3\widehat{dx_3}+2x_4\widehat{dx_4}+2x_5\widehat{dx_5},\quad \widehat{dx_k} := \bigwedge_{i\neq k} dx_k
$$
being a holomorphic 4-form on $\bP_{\Delta^\circ}$. As shown in Example 5.5.2.1 of \cite{cox1999mirror}, the Picard-Fuchs $\mathcal D_B$-module $\mathcal D_{B}/\mathcal I$ in the affine chart $\mathrm{Spec}\,\bC[z_1,z_2]$ of $B$ is determined by the left ideal $\mathcal I$ in $ A_2$ generated by
\begin{equation}\label{PFeqns}
\begin{split}
D_{PF_1}:= & \delta_{2}^2 - 2^{-2}z_2\left(\delta_{1}-2\delta_{2}\right)\left(\delta_{1}-2\delta_{2}-1\right),\\
D_{PF_2}:= & \delta_{1}^2(\delta_{1}-2\delta_2)-2^{-6}z_1(4\delta_{1}+1)(4\delta_{1}+2)(4\delta_{1}+3).
\end{split}
\end{equation}
Here $A_2$ is the 2-variable Weyl algebra and $\delta_{i} = z_i\partial_{z_i}$ are the log-differentials. Recall, we have an isomorphism between $\mathcal D_B$-modules $(\mathcal F^0,\nabla)\cong \mathcal D_B/\mathcal I$. Denote the multi-valued holomorphic solution space of $\mathcal I$ as $Sol(\mathcal D_B/\mathcal I)$, i.e.,
\[
Sol(\mathcal D_B/\mathcal I)  : = \mathrm{Hom}_{\mathcal D_{\tilde B}}(j_!(\mathcal D_{ B}/\mathcal I),\mathcal O_{\tilde B}(\log(\tilde B\setminus B))),
\]
where $B\xhookrightarrow{j} \tilde B$ is the inclusion. In other words, $Sol(\mathcal D_B/\mathcal I)$ consists of local holomorphic solutions that have logarithmic poles on the normal crossing $\tilde B\setminus B$. Most importantly, we have
\begin{equation}\label{threeisos}
Sol(\mathcal D_B/\mathcal I) \cong H_3(X_p,\bC)  \cong H^3(X_p,\bC),
\end{equation}
where the first isomorphism is induced by the period map and the second by Poincar{\'e} duality.
These isomorphisms will play an important role in our calculations.

\subsubsection{Gauss-Manin connection} From the Picard-Fuchs equations,\footnote{If using Griffiths-Dwork method one actually first gets the Gauss-Manin connection before the Picard-Fuchs equation.} the Gauss-Manin connection can be deduced once we fix a basis for the multi-valued sections in $\mathcal F^0$. 
We fix the basis $\omega^T = [\omega_1,\cdots,\omega_6]$ as
\begin{equation}\label{basisomega}
\left[\begin{array}{c}
      \omega_1 \\
      \omega_2\\
      \omega_3\\
      \omega_4\\
      \omega_5\\
      \omega_6
\end{array}
\right] := \mathrm{Res}_{f=0}
\left[\begin{array}{c}
\frac{\Omega_0}{f}\\
-2i\pi\delta_1\omega_1\\
-2i\pi\delta_2\omega_1\\
-\pi^2\delta_1\delta_2\omega_1\\
-\pi^2\delta_1^2\omega_1\\
-2i\pi^3\delta_1^2\delta_2\omega_1
\end{array}
\right] = \mathrm{Res}_{f=0}
\left[\begin{array}{c}
\frac{\Omega_0}{f}\\
\frac{i\pi2^{-7}z_1x_3^4\Omega_0}{f^2}\\
\frac{i\pi2^{-1}z_2x_1^8\Omega_0}{f^2}\\
\frac{-\pi^22^{-9}x_1^8x_3^4\Omega_0}{f^3}\\
-\pi^2\omega_2-\frac{\pi^22^{-3}z_1^2x_3^8\Omega_0}{f^3}\\
2i\pi\omega_4+\frac{3i\pi^3 2^{-16}x_1^8x_3^8\Omega_0}{f^4}
\end{array}
\right].
\end{equation}
The covariant derivatives with respect to the Gauss-Manin connection can be represented by the connection matrices $\Gamma^{(k)}$, defined via $\sum_{j=1}^6\Gamma_{ij}^{(k)}\omega_j: = \nabla_{z_k} \omega_i$, $i,j=1,\cdots,6,k=1,2$. We also denote the transpose of $\Gamma^{(k)}$ as $M^{(k)}$, which is more convenient for numerical calculation (and easier for displaying). From (\ref{PFeqns}), we have

\begin{equation}\label{GMcon1}
z_1M^{(1)} = \left[
\begin{array}{rrrrrr}
    0 & 0 & 0 & 0 & \frac{3\pi^2z_1}{32(z_1-1)} & -\frac{15i\pi^3z_1^2z_2}{64(z_1 - 1)\alpha}  \\
    \frac{i}{2\pi}& 0 & 0 & 0 & \frac{11i\pi z_1}{32(z_1-1)} & \frac{\pi^2z_1z_2(52z_1 + 3)}{64(z_1-1)\alpha} \\
    0 & 0 & 0 & 0 & 0 & -\frac{3\pi^2z_1(z_1 - 1)(z_2 - 1)}{32\alpha} \\
    0 & 0 &  \frac{2i}{\pi} & 0 & 0 & -\frac{11i\pi z_1(z_1 - 1)(z_2 - 1)}{8\alpha}\\
     0 & \frac{2i}{\pi} & 0 & 0 & -\frac{3z_1}{2(z_1-1)} & \frac{i\pi z_1z_2(49z_1+11)}{16(z_1-1)\alpha}\\
      0 & 0 & 0  & -\frac{i}{2\pi} & \frac{i}{\pi(z_1-1)} & \frac{z_1(8z_1z_2 - 6z_1 - 3z_1^2z_2 + 3z_1^2 + 3)}{2(z_1-1)\alpha} \\
\end{array}
\right]
\end{equation}
and
\begin{equation}\label{GMcon2}
z_2M^{(2)} = \left[
\begin{array}{rrrrrr}
0 & 0 & 0 & -\frac{3\pi^2z_1z_2}{128(z_1-1)(z_2-1)} & 0 & \frac{15i\pi^3 z_1^2}{128\alpha}\\
0 & 0 & \frac{z_2}{4(z_2-1)} & -\frac{11i\pi z_1z_2}{128(z_1-1)(z_2-1)} & 0 & -\frac{\pi^2z_1z_2(52z_1 + 3)}{128\alpha}\\
\frac{i}{2\pi} & 0 & -\frac{z_2}{2(z_2-1)} & 0 & 0 & -\frac{3\pi^2z_1z_2(2z_1 - 1)}{64\alpha}\\
0 & \frac{2i}{\pi} & \frac{2iz_2}{\pi(z_2-1)} & -\frac{z_2  }{2(z_2-1)} & 0 & -\frac{11i\pi z_1z_2(2z_1 - 1)}{16\alpha}\\
0 & 0 & -\frac{iz_2}{2\pi(z_2-1)} & \frac{z_2(5z_1-2)}{8(z_1-1)(z_2-1)} & 0 & -\frac{iz_1z_2\pi(49z_1 + 11)}{32\alpha}\\
0 & 0 & 0 & -\frac{iz_2(2z_1-1)}{4\pi(z_1-1)(z_2-1)} & -\frac{i}{2\pi} & -\frac{2z_1^2z_2}{\alpha},
\end{array}
\right]
\end{equation}
where $\alpha = z_1^2z_2 -(1-z_1)^2$ is the defining equation of $C_{con}$ in (\ref{disclocus}).

\subsubsection{Yukawa couplings and intersection matrix}
In this section, we calculate the intersection matrix for the basis $\omega$ in \eqref{basisomega}. This is closely related to the so-called Yukawa coupling and more generally, the \textit{$n$-point functions}. An $n$-point function is defined as the intersection pairing between $\omega_1$ and its $n$-th derivative, and the Yukawa couplings are the 3-point functions. More precisely, we define
$$
K^{i,j} := \int_{X_b}\omega_1\wedge \delta_{1}^{i}\delta_{2}^{j}\omega_1,\quad i,j\in\bZ_{\geq0}.
$$
Apparently, for $i+j\leq 2$, we have $K^{i,j}=0$ by Griffiths transversality. More generally, all the $n$-functions  can be related to each other via Griffiths transversality and Picard-Fuchs equations. In fact, they can be solved up to a multiplicative constant \cite{hosono1995mirror}. For our example, the 3,4-point functions are worked out in \cite{candelas1994mirror1,hosono1995mirror}, and Example 5.6.2.1 of \cite{cox1999mirror}. We have
\begin{equation}\label{yukawacoup}
\begin{split}
K^{3,0} & = \frac{2c}{D(z_1,z_2)},\\
K^{2,1} & = \frac{1-z_1}{2}K^{3,0} = \frac{c(1-z_1)}{D(z_1,z_2)},\\
K^{1,2} & = \frac{z_2(2z_1-1)}{1-z_2}K^{3,0} = \frac{cz_2(2z_1-1)}{(1-z_2)D(z_1,z_2)},\\
K^{0,3} & = \frac{z_2(1-z_1+z_2-3z_1z_2)}{2(1-z_2)^2}K^{3,0} = \frac{cz_2(1-z_1+z_2-3z_1z_2)}{(1-z_2)^2D(z_1,z_2)},
\end{split}
\end{equation}
where $D(z_1,z_2) = (1-z_1)^2-z_1^2z_2$ and $c$ is the multiplicative constant that cannot be determined by Picard-Fuchs equations and Griffiths transversality.

Now we want to express the entries of the intersection matrix $I : = (\int_{X_b}\omega_i\wedge \omega_j)_{6\times 6}$ in terms of the Yukawa couplings and their derivatives to the extent. By Griffiths transversality, we have $\int_{X_b}\omega_i\wedge\omega_j=0$ for $(i,j) = (1,k),(k,1),k=1,\cdots,5$ and $(i,j) = (2,3),(3,2)$. Then by the equality $\delta_{2}\int_{X_b}\omega_1\wedge \omega_5=0$, we have
$$
I_{3,5} = \int_{X_b}\omega_3\wedge \omega_5
 = \int_{X_b}\omega_1\wedge \omega_6 = I_{1,6} = -2i\pi^3 K^{2,1},
$$
Similarly, by $\delta_{1}I_{1,4} = \delta_1I_{2,3}= 0$, we have $I_{2,4} = I_{3,5} = I_{1,6} = -2i\pi^3 K^{2,1}$. To continue, we consider the equality $\delta_1I_{1,5}=0$, this gives
$$
\int_{X_b}\omega_2\wedge\omega_5+2i\pi^3\int_{X_b}\omega_1\wedge\delta_1^3\omega_1 =0,
$$
which means, we have $I_{2,5} = -2i\pi^3 K^{3,0}$. Similarly $\delta_2I_{1,4}=0$ yields $I_{3,4} = -2i\pi^3 K^{1,2}$.
Expanding $\delta_2I_{2,5},\delta_1I_{2,4}$, we have
$$
4\pi^4\delta_2K^{3,0} = -4I_{4,5}+I_{2,6},\quad 4\pi^4\delta_1K^{2,1} = 4I_{4,5}+I_{2,6};
$$
likewise $\delta_1 I_{3,4}$ gives $4\pi^4\delta_1K^{1,2} = I_{3,6}$. 

The intersection matrix now looks like
\begin{equation}\label{intermat}
I = -2i\pi^3\cdot
\left[
\begin{array}{rrrrrr}
     0 & 0 & 0 & 0 & 0 & K^{2,1}\\
      & 0 & 0 & K^{2,1} & K^{3,0} & i\pi(\delta_1K^{2,1}+\delta_2K^{3,0})\\
      &  & 0 & K^{1,2} & K^{2,1} & 2i\pi \delta_1K^{1,2} \\
      &  &  & 0 &  \frac{i\pi}{4}(\delta_1K^{2,1}-\delta_2K^{3,0}) & \frac{i}{2\pi^3}I_{4,6}\\
      &  &  &  & 0   &\frac{i}{2\pi^3}I_{5,6}\\
      &  &  &  &  &  0
\end{array}
\right].
\end{equation}
Since $I$ is anti-symmetric by definition, the lower triangular entries are determined by the upper triangular entries. Then by \eqref{yukawacoup}, $I$ is determined up to a constant $c$, except for $I_{4,6}$ and $I_{5,6}$.

To determine $I_{4,6}$ and $I_{5,6}$, we start by expanding $\delta_1I_{4,5}$ and $\delta_2 I_{4,5}$, which gives us
\begin{equation}\label{eqI456}
    \begin{split}
    2i\pi\delta_1 I_{4,5} & = -I_{5,6}+2i\pi^5\int_{
X_b}\delta_1\delta_2\omega_1\wedge\delta_1^3\omega_1,\\
    2i\pi\delta_2 I_{4,5} & = I_{4,6}+2i\pi^5\int_{
X_b}\delta_1\delta_2^2\omega_1\wedge\delta_1^2\omega_1.
    \end{split}
\end{equation}
Next, expand the Picard-Fuchs operators \eqref{PFeqns}, to get
\begin{align*}
D_{PF_1} & = (1-z_2)\delta_2^2-z_2\left(\frac{1}{4}\delta_1^2 - \delta_1\delta_2 -\frac 14 \delta_1+\frac 12\delta_2 \right),\\ D_{PF_2} & = (1-z_1)\delta_1^3 - 2\delta_1^2\delta_2-z_1\left( \frac{3}{2}\delta_1^2+\frac{11}{16}\delta_1+\frac{3}{32}\right).
\end{align*}
Then  the equation $ 0 = \int_{X_b}\delta_1\delta_2\omega_1\wedge D_{PF_2}\omega_1$ yields
\begin{align*}
0  = & (1-z_1)\int_{X_b}\delta_1\delta_2\omega_1\wedge\delta_1^3\omega_1-2\int_{X_b}\delta_1\delta_2\omega_1\wedge\delta_1^2\delta_2\omega_1 - \frac{3z_1}{2}\int_{X_b}\delta_1\delta_2\omega_1\wedge\delta_1^2\omega_1\\
& -\frac{11z_1}{16}\int_{X_b}\delta_1\delta_2\omega_1\wedge\delta_1\omega_1 - \frac{3z_1}{32}\int_{X_b}\delta_1\delta_2\omega_1\wedge\omega_1,
\end{align*}
where the last term is 0 by Griffiths transversality. Multiply both sides by $2i\pi^5$ and simplify to get
\begin{equation}\label{I456pic1}
0 = (1-z_1)2i\pi^5\int_{X_b}\delta_1\delta_2\omega_1\wedge\delta_1^3\omega_1 -2I_{4,6}-3i\pi z_1I_{4,5}+\frac{11\pi^2z_1}{16}I_{2,4}.
\end{equation}
Next, consider the operator
\begin{align*}
4(1-z_1)\delta_1D_{PF_1} + z_2D_{PF_2} = & 4(1-z_1)(1-z_2)\delta_1\delta^2_2+2z_2(1-2z_1)\delta_1^2\delta_2\\
  + z_2\left(1-\frac52z_1\right)\delta_1^2- & 2z_2(1-z_1)\delta_1\delta_2-\frac{11}{16}z_1z_2\delta_1-\frac{3}{32}z_1z_2.\\
\end{align*}
The equation $i\pi^5\int_{X_b}\delta_1^2\omega_1\wedge (4(1-z_1)\delta_1D_{PF_1} + z_2D_{PF_2})\omega_1=0 $ simplifies to
\begin{align*}
    0 = & 4i\pi^5(1-z_1)(1-z_2) \int_{X_b}\delta_1^2\omega_1\wedge \delta_1\delta^2_2\omega_1 + z_2(1-2z_1) I_{5,6}\\
    & +2i\pi z_2(1-z_1)I_{4,5} + \frac{11}{32}\pi^2z_1z_2 I_{2,5}.
\end{align*}
Then combining the above equation with \eqref{eqI456} and \eqref{I456pic1}, we get
\begin{align*}
  I_{4,6} & = \frac{3i\pi^5\cdot c\cdot z_1^2z_2(z_1-3)}{8D(z_1,z_2)^2},  \\
  I_{5,6} & = \frac{i\pi^5\cdot c\cdot z_1(3-10z_1+18z_1z_2+7z_1^2-7z_1^2z_2)}{8D(z_1,z_2)^2}.
\end{align*}
The resulting intersection matrix is holomorphic near the origin (MUM point). In particular, we have
\begin{equation}\label{Interori}
I(0,0) = 2i\pi^3c\cdot\left[\begin{array}{cccccc}0 & 0 & 0 & 0 & 0 & -1 \\0 & 0 & 0 & -1 & -2 & 0 \\0 & 0 & 0 & 0 & -1 & 0 \\0 & 1 & 0 & 0 & 0 & 0 \\0 & 2 & 1 & 0 & 0 & 0 \\1 & 0 & 0 & 0 & 0 & 0\end{array}\right].
\end{equation}

\subsection{Mirror symmetry}

Mirror symmetry is a phenomenon observed in string theory that asserts two different supersymmetric conformal field theories are in fact isomorphic to each other. When restricted to the topological field theory, it can be interpreted as the local identification of two theories near some boundary points. These two theories are called the A-model and B-model, respectively. The A-model is the Gromov-Witten theory for a smooth Calabi-Yau 3-fold $\hat X_0$, considered as varying over the complexified K{\"a}hler moduli 
and the boundary point in question is the \textit{large volume limit} point. The B-model is the Hodge theory of another smooth Calabi-Yau 3-fold $X_p$ considered as varying over the complex moduli, 
and the boundary point is the \textit{maximal unipotent monodromy} point. In particular, $\hat X_0$ and $X_p$ have mirror Hodge diamonds and are called a mirror pair.

This phenomenon was first discovered and made precise for the quintic 3-fold and its mirror \cite{candelas1991pair}. There is an ostensibly different formulation of mirror symmetry, i.e., Kontsevich's homological mirror symmetry \cite{kontsevich1995homological}, which predicts the derived equivalence between the Fukaya category of $X_p$ and category of coherent sheaves on $\hat X_0$.

Since in A-model, the complex structure is fixed, all the fibers are isomorphic as complex manifolds. This is why we use the notation $\hat X_0$ instead of something base point dependent like $\hat X_b$.

\subsubsection{Mirror map}
To be more precise, as formulated by Morrison \cite{morrison1997mathematical} mirror symmetry asserts that locally near the maximal unipotent monodromy point and large volume limit point there exists an isomorphism of $\bZ$-local systems, the mirror map,  that it induces an isomorphism of polarized  $\bZ$-VHSs. The mirror map is given as
$$
mir: R^{3}\pi_*(\underline{\bZ}_X)\big|_{(\Delta^*)^n} \xrightarrow{\simeq} \left((\Delta^*)^n\times H^{even}(\hat X_0,\bC),\nabla_A\right),
$$
where the $n$-polydiscs on both sides are centered at the maximal unipotent point and large volume limit, respectively, and $n = \dim(B) = h^{1,1}(\hat X_0) = h^{2,1}(X_p)$. The flat A-model connection $\nabla_A$ is defined via Gromov-Witten invariants. The polarization on the left-hand side is the intersection pairing while the integral structure comes from the integral homologies. Now, what about the polarization and integral structure on the right-hand side?

As pointed out by Hosono \cite{ hosono2000local,hosono2004central}, from the point of view of homological mirror symmetry the natural integral structure on $H^{even}(\hat X_0,\bC) \cong K_0(\hat X_0)\otimes_\bZ\bC $ should come from $K_0(\hat X_0)$, and the polarization on each stalk should be the relative Euler characteristic $\chi$ in (\ref{releuler}). Restrict to a fixed point $b\in (\Delta^*)^n \subset B$, we have

\begin{conjecture}[Implication of homological mirror symmetry]\label{conjihms}
$$
mir_b: \left(H^{3}(X_b,\bZ),F_B^\bullet,Q\right) \xrightarrow{\simeq} (K_0(\hat X_{0}),F_A^\bullet,\chi),
$$
where $F_B^\bullet$ and $F_A^\bullet$ are the Hodge filtrations restricted to the stalks at $b$ and $mir(b)$.
\end{conjecture}
We note this version of mirror symmetry can be extended to a statement about the ambient toric variety. For toric orbifolds, the mirror symmetry is then between the Gromov-Witten theory on a toric orbifold (A-model), and the mirror Landau-Ginzburg theory (B-model). In this situation, a similar conjecture is formulated and proved in \cite{iritani2009integral,iritani2011quantum}.

\subsubsection{Maximal unipotent monodromy point}\label{sec132}

The precise mathematical definition of maximal unipotent monodromy (MUM) point was first formulated by Morrison \cite{morrison1993compactifications}. 

\begin{definition}[Definition 5.2.2 in \cite{cox1999mirror}]\label{defmum}
For an VHS $(L,F^\bullet,B)$ whose base $B$ has dimension $r$ and normal crossing boundaries $D = \bar{B}\setminus B$, and $D = \bigcup_{i\in I}D_i$.
The point $p = D_1\cap D_2\cap\cdots\cap D_r$ is a \textit{MUM boundary point} if the following conditions hold.
\begin{enumerate}
    \item The monodromy $T_i$ around $D_i$ is unipotent $\forall i$.
    \item Denote $N_i:=\log(T_i)$, then the monodromy weight filtration $W^\bullet$ induced by $N := \sum_{i}a_i N_i$, $a_i>0$ satisfies $\dim W_0 = \dim W_1 = 1, \dim W_2 = 1+r$.
    \item Let $g_0,\cdots,g_r$ be a basis for $W_2$, with $g_0$ spanning $W_0$, then the matrix $(m_{ij})$ defined via $N_i(g_j) = m_{ij}g_0$ is invertible.
\end{enumerate}
\end{definition}
For the definition of monodromy weight filtration (which is also called the Jacobson-Morosov filtration), we refer to \cite{cox1999mirror,robles2016degenerations}. In particular, we have

\begin{lemma}\label{lemHT}
Assume $(L,F^\bullet,B)$ is a weight 3 VHS  of type $(1,r,r,1)$, $r\in\bZ_+$, e.g., a geometric VHS for a family of smooth Calabi-Yau 3-fold with $h^{2,1}=r$. We further assume $B$ satisfies the same assumption in Definition \ref{defmum} and $MUM\in\bar B
\setminus B$ is a MUM point. Then the limiting mixed Hodge structure at the MUM is Hodge-Tate and indecomposable.
\end{lemma}
\begin{proof}
For weight 3, the condition (2) ensures that the MUM exists only when $\dim B=r$. Furthermore, condition (2) also ensures the limiting mix Hodge structure has Hodge diamond
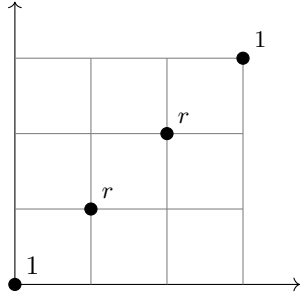
\begin{center}
\begin{tikzpicture}
 \draw [<->] (0,3.75) -- (0,0) -- (3.75,0);
 \draw [gray] (1,0) -- (1,3);
 \draw [gray] (2,0) -- (2,3);
 \draw [gray] (3,0) -- (3,3);
 \draw [gray] (0,1) -- (3,1);
 \draw [gray] (0,2) -- (3,2);
 \draw [gray] (0,3) -- (3,3);
 \draw [fill] (0,0) circle [radius=0.08];
 \node [above right] at (0,0) {\footnotesize{$1$}};
 \draw [fill] (1,1) circle [radius=0.08];
 \node [above right] at (1,1) {\scriptsize{$r$}};
 \draw [fill] (2,2) circle [radius=0.08];
 \node [above right] at (2,2) {\scriptsize{$r$}};
 \draw [fill] (3,3) circle [radius=0.08];
 \node [above right] at (3,3) {\scriptsize{$1$}};
\end{tikzpicture}
\end{center}
By the classification of mixed Hodge structure, it is Hodge-Tate, or of type $IV_r$ \cite{kerr2017polarized}. The indecomposability follows from the condition (3).
\end{proof}

\begin{corollary}\label{coMT}
For a weight 3 VHS of type $(1,2,2,1)$, if it has both a MUM point and a type $I$ degeneration boundary point, then the period map associated to it is Mumford-Tate generic.
\end{corollary}
\begin{proof}
The proof is very similar to the proof of Theorem 5.4 in \cite{deng2021extension}, for the definition of the Mumford-Tate group we refer to \cite{green2013hodge}. By the classification of Hodge representation of type $(1,2,2,1)$ \cite{han2020hodge}, the non-generic Mumford-Tate groups can appear in the following cases
\begin{enumerate}
    \item A compact maximal torus.
    \item $SL(2)\times SL(2)$ corresponding to Hodge types of $(1,1,1)\otimes (1,1)$.
    \item $U(2,1)$ corresponding to $W\oplus \overline W$, where $W$ has type $(1,2,0,0)$ or $(1,1,1,0)$.
    \item $SL(2)\times Sp(4)$ corresponding to Hodge types of $ (1,1)\oplus(1,1,1,1)$,
\end{enumerate}
By Lemma \ref{lemHT}, the MUM point corresponding to type $IV_2$ mixed Hodge structure, which can not appear in cases 1 and 3 above. For case 2, it cannot have type 1 degeneration. Case 4 can be ruled out by Lemma \ref{lemHT}, as the Hodge structure at MUM point is indecomposable. Thus the Mumford-Tate group has to be generic.
\end{proof}

In our case, as shown in \cite{candelas1994mirror1}, the MUM point is $D_{(1,0)}\cap C_\infty$, which is precisely the origin in the affine chart $\mathrm{Spec}\,\bC[z_1,z_2]$ as in Figure \ref{discrilociaff}. This can also be confirmed by our calculation in Section \ref{secnilcone}.

Finally, we note by the isomorphisms in (\ref{threeisos}), that there is a unique single valued holomorphic solution for the Picard-Fuchs system in a neighborhood of a MUM point.

\subsubsection{Givental's I-function and Hosono's $w$-function}

By the Frobenius method, the solutions of the GKZ and the Picard-Fuchs systems for the B-model can be naturally assembled into an A-model cohomology valued function. These are called the Givental's $I$-functions \cite{givental1996equivariant,givental1998mirror}, which is crucial in Givental's proof of mirror theorem for toric complete intersection. A slightly different form is also considered in \cite{lian1997mirror} and \cite{hosono2000local,hosono2004central}. One key difference is the following. In Givental's $I$-function the cohomological coefficients take value in the cohomology ring of the ambient toric variety, while Hosono's $w$-function takes value in the cohomology ring of the Calabi-Yau 3-fold.

In our example, still working in the affine chart $\mathrm{Spec}\,\bC[z_1,z_2]$. As shown in Example 11.2.5.1 \cite{cox1999mirror}, the unnormalized Givental's I-function is
\begin{equation*}
\begin{split}
\tilde I  =  & e^{t_0/\hbar} \sum_{d_1,d_2\geq0}\left((z_1/2^8)^{d_1+D_3}(z_2/2^2)^{d_2+D_1}\times\right.\\
& \left.\frac{\prod_{m=1}^{d_1}(4D_3+m\hbar)}{\prod_{m=1}^{d_2}(D_1+m\hbar)^2\prod_{m=1}^{d_1}(D_3+m\hbar)^3\prod_{m=1}^{d_1-2d_2}(D_3-2D_1+m\hbar)}\right).
\end{split}
\end{equation*}
This function satisfies the GKZ system. We can normalize it by setting $I_{\mathcal V}:=\mathrm{Euler(\mathcal V)} \cdot \tilde I$, where $\mathrm{Euler(\mathcal V)} = c_1(-\mathcal K_{X_{\Sigma'}}) = 4D_3$ is the anti-canonical class for  the ambient toric variety $X_{\Sigma'}$. Recall $X_{\Sigma'}$ is the resolved $\bP[1,1,2,2,2]$ as defined in Section \ref{sec111}. Then by \cite{givental1998mirror}, the coefficient of $I_{\mathcal V}$ generates the solution space of Picard-Fuchs system $Sol(\mathcal D_B/\mathcal I)$.

Now, we first set $t_0=0,\hbar = 1$ in $\tilde I$, then make the change $D_1\rightarrow \frac{i^*D_1}{2\pi i}=\frac{L}{2\pi i}$, $D_3\rightarrow\frac{i^*D_3}{2\pi i}= \frac{H}{2\pi i}$, and then multiply a suitable ratio of $\Gamma$-functions. This yields Hosono's $w$-function
\begin{equation}\label{wformourc}
\begin{split}
  w(z_1,z_2, \frac{H}{2\pi i}, & \frac{L}{2\pi i})  =   \sum_{d_1,d_2\geq 0}\left((z_1/2^8)^{d_1+\frac{H}{2\pi i}}(z_2/2^2)^{d_2+\frac{L}{2\pi i}}\times\right.\\
  & \frac{\Gamma(4d_1+\frac{4H}{2\pi i}+1)}{\Gamma(d_2+\frac{L}{2\pi i}+1)^2\Gamma(d_1+\frac{H}{2\pi i}+1)^3\Gamma(d_1-2d_2+\frac{H-2L}{2\pi i}+1)}.
  \end{split}
\end{equation}
After expanding this function, since the ring for the cohomology classes in $\tilde I$ and $w$ are different, the number of terms that survived is different. For example, there is a non-zero term in $\tilde I$ with value in $H^8(X_{\Sigma'},\bZ)$, which vanishes after pulling back to the Calabi-Yau hypersurface. One can easily see that the coefficients of $w$ still form a basis of $Sol(\mathcal D_B/\mathcal I)$.

As we mentioned at the end of Section \ref{sec132}, near the MUM point ($z_1=z_2=0$) there is a unique single-valued holomorphic solution. In terms of the $w$-function it is 
$$w_0(z_1,z_2) := w(z_1,z_2,0,0) =\sum_{d_1,d_2\geq0}\frac{\Gamma(4d_1+1)\cdot (z_1/2^8)^{d_1}(z_2/2^2)^{d_2}}{\Gamma(d_2+1)^2\Gamma(d_1+1)^3\Gamma(d_1-2d_2+1)}. $$

\subsubsection{Hosono's central charge formula}
One of the (conjectural) nice properties for the $w$-function is the following.
\begin{conjecture}[\cite{hosono2000local,hosono2004central}]\label{conjcc}
Consider a mirror pair of Calabi-Yau 3-folds $\hat X_0, X_p$ that are complete intersections inside toric varieties.
Fix a basis $\{E_i\}$ of $K_0(\hat X_0)\otimes_\bZ\bC$, with $\chi_{ij} = \chi(E_i,E_j)$. 
    Then the $w$-function for the B-model on $X_b$ can be written as
    $$
    w = \sum_{i,j}\int_{mir(E_i)}\Omega\cdot\chi^{ij}\cdot ch(\check E_j),
    $$
    where $\Omega$ is the canonical 3-form on $X_b$ and $\int_{mir(E_i)}\Omega$ is the period associated to the mirror of $E_i$.
\end{conjecture}

In particular, as proposed in \cite{hosono2000local}, $w$ has the following property.

\begin{corollary}[Symplectic invariance of $w$ \cite{hosono2000local}]\label{corinv}
Assume Conjectures \ref{conjihms} and \ref{conjcc} hold, then $w$ is $Aut(H^3(X_b,\bC),Q)$ invariant; in particular it is monodromy invariant.
\end{corollary}
\begin{proof}

By Conjecture \ref{conjihms}, we have $ Aut(K(\hat X_0)\otimes_\bZ\bC,\chi) \cong Aut(H^3(X_b,\bC),Q) $. 
Then the corollary follows immediately from Conjecture \ref{conjcc}. The second part is due to the fact that the monodromy group is a subgroup of $Aut(H^3(X_b,\bC),Q)$.
\end{proof}

By the mirror map and Chern character, we can extend the isomorphisms in (\ref{threeisos}) as
\begin{equation}\label{exthreeisos}
Sol(\mathcal D_B/\mathcal I) \cong H_3(X_b,\bC)  \cong H^3(X_b,\bC)\cong H^{even}(\hat X_0,\bC)\cong K_0(\hat X_0)\otimes_\bZ\bC,
\end{equation}
The central charge formula provides a direct isomorphism between the left and the right-hand sides:

\begin{corollary}[Central charge formula \cite{hosono2000local,hosono2004central}]\label{corccf}
Assume Conjecture \ref{conjcc}. Then for $E\in K_0(\hat X_0)\otimes_\bZ\bC$, we have
$$\int_{mir(E)}\Omega = \int_{\hat X_0} ch(E)\cdot w\cdot td(\hat X_0).$$
\end{corollary}
\begin{proof}
By Conjecture \ref{conjcc} we have the following map
\[
\begin{split}
    K_0(\hat X_0)\otimes _\bZ\bC&\rightarrow Sol(\mathcal D_B/\mathcal I)\\
    E& \rightarrow  \int_{\hat X_0} ch(E)\cdot w\cdot td(\hat X_0),
\end{split}
\]
Then the corollary follows from the Hirzebruch–Riemann–Roch formula (\ref{releuler}). 
\end{proof}

Now, if we fix an integral symplectic basis in $K_0(\hat X_0)\otimes _\bZ\bC$, then by the symplectic invariance of $w$, the central charge formula will give an integral symplectic basis on $Sol(\mathcal D_B/\mathcal I)$, which in turn gives an integral symplectic basis for $H^3(X_b,\bC)$. For the rest of this section, we fix our symplectic form as the standard symplectic form 
\begin{equation}\label{strdsympl}
Q^{st}:=\left[\begin{array}{cccccc}0 & 0 & 0 & 0 & 0 & -1 \\0 & 0 & 0 & 0 & -1 & 0 \\0 & 0 & 0 & -1 & 0 & 0 \\0 & 0 & 1 & 0 & 0 & 0 \\0 & 1 & 0 & 0 & 0 & 0 \\1 & 0 & 0 & 0 & 0 & 0\end{array}\right].
\end{equation}

Hosono stated  \cite{hosono2000local}, and gave a heuristic proof, of a variant of the following proposition, having noted that certain special cases (i.e., complete intersections in toric variety) are implicit in \cite{hosono1995mirror,hosono1995mirror2,hosono1996gkz}.

\begin{proposition}\label{lemhosin}

Let $\hat X_0$ be a smooth Calabi-Yau 3-fold. Given an arbitrary basis $\{T_i\}_{i=1}^r$ of $H^2(\hat X_0,\bZ)/torsion$, define $\{K_j\}_{j=1}^r\subset H^4(\hat X_0,\bZ)/torsion$ by $\int_{\hat X_0}T_kK_j = \delta_{kj}$, and define $V$ via $\int_{\hat X_0} V=-1$. Then there exists an integral symmetric matrix $C:=(C_{ij})_{r\times r}$, such that 
\begin{equation}\label{hosobaisis}
\{2,(2T_k-\sum_{l}C_{kl}K_l)td^{-1}_{\hat X_0},K_{r-j},V\},\quad j,k=1,\cdots,r
\end{equation}
gives an integral symplectic basis for $K_0(\hat X_0)\otimes_\bZ\bC$ under the inverse of the Chern character. Moreover, if one assumes homological mirror symmetry or that Conjecture \ref{conjihms} holds for $\hat X_0$, then 
\[
\{1,(T_k-\sum_{l}\frac 12 C_{kl}K_l)td^{-1}_{\hat X_0},K_{r-j},V\}    
\]
is a symplectic basis for $K_0(\hat X_0)/torsion$.
\end{proposition}
\begin{proof}
We give a detailed proof in Appendix \ref{appa}.
\end{proof}

\begin{remark}\label{rmkaftprop}
The integral matrix $C$ in the proposition is not unique. Let $AId_{r\times r}$ be the anti-diagonal matrix with all entries in $(k,r-k)$ equal to $1$. If an integral symplectic matrix $B$ satisfies $[B,AId_{r\times r}]=0$, then Proposition \ref{lemhosin} holds with $C': = C+B$ in place of $C$

\end{remark}

\begin{remark}
In Proposition 1 of \cite{hosono2000local}, the matrix $C$ is rational and the basis 
\[
\{1,(T_k-\sum_{l}C_{kl}K_l)td^{-1}_{\hat X_0},K_{r-j},V\},\quad j,k=1,\cdots,r.
\]
is asserted to be integral symplectic. However, without the assumption of the homological mirror symmetry, the Proposition \ref{lemhosin} only guarantees the above basis to be half-integral. Nonetheless, we have improved the conclusion that $C_{ij}$ is rational to half-integral, which is very useful in numerical calculations.
\end{remark}

we now apply Proposition \ref{lemhosin} to our case. By the intersection numbers in (\ref{intnum}), we have $r=2$, $T_1 = H,T_2 = L$, $K_1 = h := \frac{1}{4}HL$, $K_2 = l:=\frac{1}{4}H^2-\frac 12 HL$, $V = -\frac 18 H^3$, and $td^{-1}(\hat X_0) = 1-2l-\frac{14}{3}h$. Then we have the following lemma.
\begin{lemma}\label{lembaourc}
For the resolved octic studied in Section \ref{secamod}, if we assume Conjecture \ref{conjihms} holds, then there exists  integers $C_{11},C_{12}=C_{21},C_{22}$ such that
\begin{equation}\label{hosbaourca}
\{1,(H-\frac12C_{11}h-\frac12C_{12}l)td^{-1}(\hat X_0),(L-\frac12C_{12}h-\frac12C_{22}l)td^{-1}(\hat X_0),l,h,-\frac18 H^3\}
\end{equation}
gives an half-integral symplectic basis for $K_0(\hat X_0)\otimes \bC$. One may  take $C_{12}$ to be any integer, or modify $C_{11}$ and $C_{22}$ by adding a common integer.
\end{lemma}
\begin{proof}
Specializing Proposition \ref{lemhosin} into our case gives the first half of the statement. For the second half, apply the argument in Remark \ref{rmkaftprop} to the case $r=2$. This tells us, we are allowed to twist $C$ by adding an arbitrary integral matrix $B = \left[\begin{array}{cc}B_{11} & B_{12} \\B_{12} & B_{11}\end{array}\right]$, which proves the second half of the lemma.
\end{proof}

Now, assume Conjecture \ref{conjcc} holds, then by Corollary \ref{corccf} we can calculate the monodromies for the basis (\ref{hosbaourca}) from Hosono's $w$-function (\ref{wformourc}). 

From now on, in order to 
compare with the results in \cite{candelas1994mirror1} and 
simplify the calculation in the next section, we will now work in a different symplectic form
$$
Q:=\left[\begin{array}{cccccc}0 & 0 & 0 & 1 & 0 & 0 \\0 & 0 & 0 & 0 & 1 & 0 \\0 & 0 & 0 & 0 & 0 & 1 \\-1 & 0 & 0 & 0 & 0 & 0 \\0 & -1 & 0 & 0 & 0 & 0 \\0 & 0 & -1 & 0 & 0 & 0\end{array}\right].
$$
Since $Q$ and $-Q$ have the same automorphism group, we are allowed to work with a $-Q$-symplectic basis.
By Lemma \ref{lembaourc}, the following basis is integral symplectic with respect to $-Q$
\begin{equation}\label{newsympbasis}
\{1,(H-\frac12C_{11}h-\frac12C_{12}l)td^{-1}(\hat X_0),(L-\frac12C_{12}h-\frac12C_{22}l)td^{-1}(\hat X_0),-\frac18 H^3,h,l\}.
\end{equation}

\begin{proposition}\label{prop2.2sec2}
On the affine chart $\mathrm{Spec}\,\bC[z_1,z_2]$ of the moduli space of the mirror octic. Fix a base point $b$ near the origin (the MUM point). An integral basis of $H^3(X_b,\bC)$ at $b$ is given as the mirror of the basis (\ref{hosbaourca}) with $C_{12} = 4$, $C_{22}=0$.\footnote{This choice is made to compare with those in \cite{candelas1994mirror1}.}  Denote the monodromy matrices around the divisors $D_{(1,0)}$ and $D_{(0,1)}$ as $T_{m_1}$ and $T_{m_2}$, then we have
$$
T_{m_1} = 
\left[\begin{array}{rrrrrr}1 & -1 & 0 & 6 & 4-\frac{C_{11}}{2} & 0 \\0 & 1 & 0 & \frac{-C_{11}}{2}-4 & -8 & -4 \\0 & 0 & 1 & -4 & -4 & 0 \\0 & 0 & 0 & 1 & 0 & 0\\0 & 0 & 0 & 1 & 1 & 0\\ 0 & 0 & 0 & 0 & 0 & 1\end{array}\right],\,
T_{m_2} = 
\left[\begin{array}{rrrrrr}1 & 0 & -1 & 2 & -2 & 0 \\0 & 1 & 0 & -2 & -4 & 0 \\0 & 0 & 1 & 0 & 0 & 0 \\0 & 0 & 0 & 1 & 0 & 0\\0 & 0 & 0 & 0 & 1 & 0\\ 0 & 0 & 0 & 1 & 0 & 1\end{array}\right].
$$
\end{proposition}
\begin{proof}
In order to compare with the result in \cite{candelas1994mirror1}, we will give a detailed calculation in Appendix \ref{appb}.
\end{proof}

From the calculation in the next section,  we will see that all the monodromies will stay (half)-integral when we take $C_{11}$ to be any even number.
\begin{remark}
The monodromy matrices in Proposition \ref{prop2.2sec2} differ from those in \cite{candelas1994mirror1} both at entries $(1,4)$ by a minus sign. In particular, if we change our basis by an integral symplectic transformation
\[
\left[
\begin{array}{c}
     1 \\
     T_1\\
     T_2\\
     V\\
     K_1\\
     K_2
\end{array}
\right] \rightarrow
\left[
\begin{array}{cccccc}
1 & 0 & 0 & 0 & -12 & 0\\
0 & 1 & 0 & -12 & 0 & 0\\
0 & 0 & 1 & 0 & 0 & 0\\
0 & 0 & 0 & 1 & 0 & 0\\
0 & 0 & 0 & 0 & 1 & 0\\
0 & 0 & 0 & 0 & 0 & 1\\
\end{array}
\right]\cdot
\left[
\begin{array}{c}
     1 \\
     T_1\\
     T_2\\
     V\\
     K_1\\
     K_2
\end{array}
\right],
\]
 then the resulting monodromy matrix will have the same form as those in \cite{candelas1994mirror1}.
\end{remark}

\section{Calculation of the nilpotent cones} \label{seccal}

In this section, we will finish steps (3)-(5) in Section \ref{olapp}. Recall in Section \ref{secdiscloci} we have constructed a compact moduli space $\tilde B$, where the discriminant loci consist of normal crossing boundary divisors.

Now fix a base point $b$ in $B$ near the maximal unipotent point and an integral basis of $H^3(X_b,\bC)$ as in Proposition \ref{prop2.2sec2}. Then the (candidate) nilpotent fan $\Sigma$ can be calculated as follows. For each normal crossing intersection point $p$ of two divisors, say $D_1,D_2$, we take two small loops $l_1,l_2$ based at a point $p'\in B$ near the intersection, such that $l_i$ winds around $D_i$ once. More precisely, locally the smooth part of moduli space is isomorphic to $\Delta^*\times\Delta^*$, and the two loops above correspond to the two generators of the free abelian group $\pi_1(\Delta^*\times\Delta^*)$. Then we chose a path $\beta_{bp'}$ from $b$ to $p'$, and calculate the monodromy $T_i$ around the path $\beta_{bp'}\cdot l_i\cdot\beta_{p'b}$, for each $i=1,2$. Then the nilpotent cone associated to $p$ is $\sigma_p:=\{\sum_{i= 1}^2\log(T_i)t_i:t_i>0,t_i\in\bQ\}\subset End(H^3(X_p,\bQ),Q) = \mathfrak{sp}(6,\bQ)$. Now take the union of all these cones, denote it as $\Sigma'$, then (candidate) fan $\Sigma$ is the $\Gamma$-adjoint orbit of $\Sigma'$, where $\Gamma\subset Aut(H^3(X_b,\bZ),Q) = Sp(6,\bZ)$ is a finite subindex subgroup that will be chosen to make $\Sigma$ indeed a fan.

As pointed out in Section \ref{olapp}, we will first calculate the monodromy operators in a non-integral basis (\ref{basisomega}) and then transform it into the integral basis in Proposition \ref{prop2.2sec2}.

\subsection{Monodromy matrices in basis $\omega$}\label{secmmbo}
The monodromy action on the basis $\omega$ is calculated numerically via the parallel transportation by the Gauss-Manin connection given in (\ref{GMcon1}) and (\ref{GMcon2}). The details of this calculation are given as follows.

\subsubsection{Algorithm for parallel transportation}
Consider a loop $\gamma:[0,1]\rightarrow  B$ with $b =\gamma(0) = \gamma(1)$, and a local section $s(z_1,z_2)$ of $E$. The section is parallel along $\gamma$ if and only if
\begin{align*}
0 &= \nabla_{\dot{\gamma}(t)}s(z_1,z_2) 
=\sum_{i=1}^2 \dot{z}_i(t)\nabla_{\partial_{z_i}}\left(\sum_{k=1}^6s^k\omega_k\right)\\
&= \sum_{i=1}^2 \dot{z}_i(t)\left(\sum_{k=1}^6\partial_{z_i}s^k\cdot \omega_k  + \sum_{k=1}^6s^k\nabla_{\partial_{z_i}}\omega_k \right) \\
& = \sum_{i=1}^2 \dot{z}_i(t)\left(\sum_{k=1}^6\partial_{z_i}s^k\cdot \omega_k +\sum_{j,k=1}^6s^k\Gamma_{kj}^{(i)} \omega_j \right)\\
& = \sum_{k=1}^6\left(\sum_{i=1}^2 \dot{z}_i(t)\left(\partial_{z_i}s^k+\sum_{j=1}^6s^j\Gamma^{(i)}_{jk}\right)\right)\omega_k \\
& = \sum_{k=1}^6 \left(\dot{s}^k(t) + \sum_{i,j=1}^6\dot{z}_i\Gamma_{jk}^{(i)}s^j(t) \right)\omega_k \\
&= \sum_{k=1}^6 \left(\dot{s}^k(t) + \sum_{i,j=1}^6\dot{z}_iM_{kj}^{(i)}s^j(t)\right)\omega_k,
\end{align*}
which gives
$$
\dot{s}(t)  = -\sum_{i,k=1}^k\dot{z}_i\omega_kM_{kj}^{(i)}s^j(t).
$$
Then the parallel transportation of a vector $v_0\in H^3(X_b,\bC)$ around the loop $\gamma$ is
$$
v_1 = \int_0^1\dot{v}(t)dt =-\int_{0}^1\left( \sum_{i,k=1}^6\omega_kM_{kj}^{(i)}v^j(t)\right)dt.
$$

The numerical calculation is an iterated linear approximation
$$
\mathbf{v}_{\frac{k}{N}} := -\sum_{l,m=1}^6\omega_lM_{lm}^{(i)}\left(z_1\left(\frac{k-1}{N}\right),z_2\left(\frac{k-1}{N}\right)\right)\mathbf{v}^l_{\frac{k-1}{N}}, \quad k=1,\cdots, N,
$$
where we have divided the loop into $N$ equal pieces, and $\mathbf{v}_{\frac{k}{N}}=(\mathbf{v}_{\frac{k}{N}}^1,\dots,\mathbf{v}_{\frac{k}{N}}^6)$ is the result at step $k$. Apparently, we have $\mathbf{v}_0 = v_0$ and $\mathbf v_{1}\rightarrow v_1$ when $N\rightarrow \infty$. Then the monodromy matrix is $S = [\mathbf{col}_1,\cdots,\mathbf{col}_6]^T$, where $\omega^T\cdot\mathbf{col}_k$ is the parallel transport of the basis vector $\omega_k$ along $\gamma$.

\subsubsection{Numerical calculation of monodromy matrices}

We start by working in the affine chart $\mathrm{Spec}\,\bC[z_1,z_2]$ and pick a base point $b=(10^{-4},10^{-4})$ near the maximal unipotent point $(0,0)$. As shown in Figure \ref{discrilocalz1z2path}, to calculate the monodromy operators along $D_{(1,0)}$ and $C_{\infty}$, we choose the loop to be $l_1$ and $l_2$, respectively. We take the number of steps to be $N=10^6$. The monodromy matrix  $S_{m_i}$ for $l_i$ should approximately equal to $\exp(-2i\pi\mathrm{Res}_{z_i=0}\Gamma^{(i)})$ which can serve as a validity check.

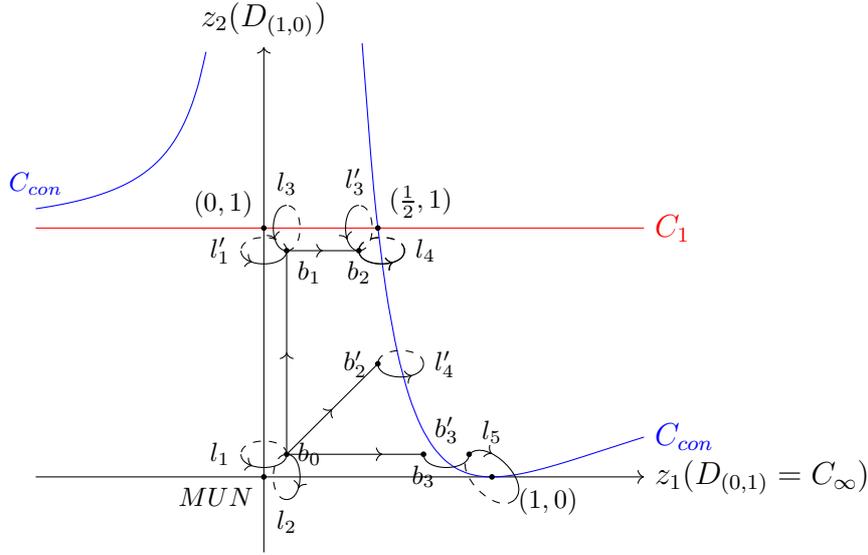
\begin{figure}
\begin{center}
\begin{tikzpicture}
  \draw[->] (-3, 0) -- (5, 0) node[right] {$z_1 (D_{(0,1)}=C_\infty)$};
  \draw[->] (0, -1) -- (0, 5.7) node[above] {$z_2 (D_{(1,0)})$};
  \draw[scale=2, domain=-1.5:-0.38, smooth, variable=\x, blue] plot ({\x}, {1/10*(1-1/\x)^2+3/2});
  \draw[scale=3, domain= 0.431:5/3, smooth, variable=\x, blue] plot ({\x}, {11/10*(1-1/\x)^2}) node[right] {$C_{con}$};
  \draw[scale=3, domain=-1:5/3, smooth, variable=\x, red] plot ({\x}, {11/10}) node[right] {$C_{1}$};
  \node [above] at (-3, 36/10) {\footnotesize{$\color{blue}{C_{con}}$}};
   \draw [fill] (0,33/10) circle [radius=0.03];
 \node [above left] at (0,33/10) {\footnotesize{$(0,1)$}};
\draw [fill] (3,0) circle [radius=0.03];
 \node [below right] at (3.2,0) {\footnotesize{$(1,0)$}};
 \draw [fill] (1.5,33/10) circle [radius=0.03];
 \node [above right] at (1.5,33/10) {\footnotesize{$(\frac{1}{2},1)$}};
 \draw [fill] (0,0) circle [radius=0.03];
 \node [below left] at (0,0) {\footnotesize{$MUN$}};
 \draw [fill] (0.3,0.3) circle [radius=0.03];
 \node [right] at (0.3,0.3) {\footnotesize{$b_0$}};
  \draw[dashed] (0.3,0.3) to [out=90,in=90] (-0.3,0.3);
  \draw[ decoration={markings, mark=at position 0.3 with {\arrow{>}}}, postaction={decorate}] (-0.3,0.3)  to [out=-90,in=-90]   (0.3,0.3);
  \draw[decoration={markings, mark=at position 0.7 with {\arrow{>}}}, postaction={decorate}] (0.3,0.3) to [out=0,in=0] (0.3,-0.3);
  \draw[dashed] (0.3,-0.3)  to [out=180,in=180]   (0.3,0.3);
  \draw[decoration={markings, mark=at position 0.5 with {\arrow{>}}}, postaction={decorate}] (0.3,0.3) -- (0.3,3);
  \draw[dashed] (0.3,3) to [out=90,in=90] (-0.3,3);
  \draw[ decoration={markings, mark=at position 0.3 with {\arrow{>}}}, postaction={decorate}] (-0.3,3)  to [out=-90,in=-90]   (0.3,3);
  \draw[dashed] (0.3,3) to [out=0,in=0] (0.3,3.6);
  \draw[decoration={markings, mark=at position 0.7 with {\arrow{>}}}, postaction={decorate}] (0.3,3.6)  to [out=180,in=180]   (0.3,3);
   \draw [fill] (0.3,3) circle [radius=0.03];
 \node [below right] at (0.3,3) {\footnotesize{$b_1$}};
  \draw [fill] (1.25,3) circle [radius=0.03];
 \node [below] at (1.25,3) {\footnotesize{$b_2$}};
  \draw[decoration={markings, mark=at position 0.5 with {\arrow{>}}}, postaction={decorate}] (0.3,3) -- (1.25,3);
  \draw[dashed] (1.25,3) to [out=0,in=0] (1.25,3.6);
  \draw[decoration={markings, mark=at position 0.7 with {\arrow{>}}}, postaction={decorate}] (1.25,3.6)  to [out=180,in=180]   (1.25,3);
  \draw[decoration={markings, mark=at position 0.7 with {\arrow{>}}}, postaction={decorate}] (1.25,3) to [out=-90,in=-90] (1.85,3);
  \draw[dashed] (1.85,3)  to [out=90,in=90]   (1.25,3);
    \draw[decoration={markings, mark=at position 0.7 with {\arrow{>}}}, postaction={decorate}] (1.25,3) to [out=-90,in=-90] (1.85,3);
  \draw[dashed] (1.85,3)  to [out=90,in=90]   (1.25,3);
  \draw[decoration={markings, mark=at position 0.5 with {\arrow{>}}}, postaction={decorate}] (0.3,0.3) -- (1.5,1.5);
    \draw[decoration={markings, mark=at position 0.7 with {\arrow{>}}}, postaction={decorate}] (1.5,1.5) to [out=-90,in=-90] (2.1,1.5);
  \draw[dashed] (2.1,1.5)  to [out=90,in=90]   (1.5,1.5);
    \draw [fill] (1.5,1.5) circle [radius=0.03];
 \node [left] at (1.5,1.5) {\footnotesize{$b_2'$}};
 \draw [fill] (2.1,0.3) circle [radius=0.03];
 \node [below] at (2.1,0.3) {\footnotesize{$b_3$}};
    \draw (2.1,0.3) to [out=-90,in=-90] (2.7,0.3);
  \draw[decoration={markings, mark=at position 0.7 with {\arrow{>}}}, postaction={decorate}] (0.3,0.3) -- (2.1,0.3);
    \draw [fill] (2.7,0.3) circle [radius=0.03];
 \node [above left] at (2.7,0.3) {\footnotesize{$b_3'$}};
 %
 %
   \draw[decoration={markings, mark=at position 0.3 with {\arrow{>}}}, postaction={decorate}] (2.7,0.3) to [out=45,in=45] (3.3,-0.3);
  \draw[dashed] (3.3,-0.3)  to [out=-135,in=-135]   (2.7,0.3);
    %
  %
 \node [below] at (0.3,-0.3) {\footnotesize{$l_2$}};
  \node [left] at (-0.3,0.3) {\footnotesize{$l_1$}};
  \node [left] at (-0.3,3.0) {\footnotesize{$l_1'$}};
  \node [above] at (0.3,3.6) {\footnotesize{$l_3$}};
  \node [above] at (1.2,3.6) {\footnotesize{$l_3'$}};
  \node [right] at (1.85,3) {\footnotesize{$l_4$}};
  \node [right] at (2.1,1.5) {\footnotesize{$l_4'$}};
  \node [above] at (3,0.3) {\footnotesize{$l_5$}};
\end{tikzpicture}
\caption{Discriminant locus and the loops used to calculate the monodromy operators in the affine chart $(z_1,z_2)$.}
\label{discrilocalz1z2path}
\end{center}
\end{figure}
With an error $\sim 10^{-5}$, the result is
\begin{equation}\label{mummonnum}
S_{m_1} = 
\left[\begin{array}{rrrrrr}1 & 1 & 0 & 0 & 2 & -4/3 \\0 & 1 & 0 & 0 & 4 & -4 \\0 & 0 & 1 & 4 & 0 & -2 \\0 & 0 & 0 & 1 & 0 & -1\\0 & 0 & 0 & 0 & 1 & -2\\ 0 & 0 & 0 & 0 & 0 & 1\end{array}\right],\,
S_{m_2} = 
\left[\begin{array}{rrrrrr}1 & 0 & 1 & 0 & 0 & 0 \\0 & 1 & 0 & 4 & 0 & 0 \\0 & 0 & 1 & 0 & 0 & 0 \\0 & 0 & 0 & 1 & 0 & 0\\0 & 0 & 0 & 0 & 1 & -1\\ 0 & 0 & 0 & 0 & 0 & 1\end{array}\right]
\end{equation}
Now, we continue to calculate the monodromy operators in this affine chart. For the other operators, we conduct parallel transportation along the following loops
\begin{equation}
\begin{array}{ll}
    (1) S_{C_1}: \beta_{b_0b_1}\cdot l_3\cdot \beta_{b_1b_0},& b_1 = (10^{-4},1-10^{-4}),r(l_1') = 10^{-4}.\\
    (2) S_{C_{con_1}}: \beta_{b_0b_2'} \cdot l_4'\cdot \beta_{b_2'b_0},& b_2 = (95/110,1/100),r(l_4') = 10^{-4}.\\
    (3) S_{E_1}: \beta_{b_0b_3}\cdot \beta_{b_3b_3'}\cdot l_5\cdot\beta_{b_3'b_3}\cdot\beta_{b_3b_0}, &b_3' = (1-10^{-4},10^{-4}), r(l_5) =  \sqrt{2}\cdot10^{-4}.\\
    (4) S_{E_2}: \beta_{b_0b_3}\cdot \beta_{b_3b_3'}\cdot l_5^2\cdot\beta_{b_3'b_3}\cdot\beta_{b_3b_0}, &b_3' = (1-10^{-4},10^{-4}), r(l_5) =  \sqrt{2}\cdot10^{-4}.
\end{array}
\end{equation}
where $\beta_{ab}$ denotes the path starting from $a$ ending at $b$, and the meaning for the other notations should be clear by consulting Figure \ref{discrilocalz1z2path}. 
We note, that the loops for the exceptional divisors are first chosen in $\tilde B$ and then pushed forward to $ B$.

The results are a lot messier in these cases, for example, if we approximate parallel transport along both $\beta_{b_0b_2'}$ and $l_4'$ with $10^{6}$ steps each, the result keeping four significant digits is
\begin{equation}\label{con1numma}
\begin{split}
S_{con_1} & =  
\left[\begin{array}{rrr}1.003 - 29.51i & 69.37+0.0064i & 49.72+0.004587i  \\35.55 + 0.004223i & 0.9901 + 83.56i & - 0.007114 + 59.9i \\10.03 + 0.001375i & - 0.003232 + 23.57i & 0.9977 + 16.89i  \\- 0.0004211 + 2.347i & - 5.516 - 0.0009899i & - 3.954 - 0.0007095i \\- 0.001059 + 6.378i & - 14.99 - 0.002489i & - 10.75 - 0.001784i \\ 0.9991 + 0.0002574i & - 0.000605 + 2.349i & - 0.0004336 + 1.683i \end{array}\right.\\
& \quad \quad \left.\begin{array}{rrr} - 0.04224 + 457.9i & - 0.02736 + 296.5i & 872.8 + 0.08052i\\  - 551.6 - 0.06551i & - 357.2 - 0.04243i & - 0.1249 + 1051.0i\\  - 155.6 - 0.02133i & - 100.7 - 0.01381i & - 0.04066 + 296.5i\\ 1.007 - 36.41i & 0.004231 - 23.58i & - 69.4 - 0.01245i\\ 0.01643 - 98.97i & 1.011 - 64.08i & - 188.6 - 0.03131i\\ - 15.5 - 0.003993i & - 10.04 - 0.002586i & 0.9924 + 29.55i \end{array}\right].
\end{split}
\end{equation}
So we will not list all the numerical results here. Now, there are divisors in the discriminant loci (\ref{disclocus}), (\ref{disclocus2}) that located outside the affine chart $\mathrm{Spec}\,\bC[z_1,z_2]$. These are the divisors $C_0 = D_{(-1,0)}$, $D_{(1,-2)}$, $D_{(0,-1)}$, $D_{(1,-1)}$, $D_{(-1,-1)}$, and $E_0$.
In order to calculate the monodromy around these divisors, we need to work in other four affine charts. The first one is
\[
\begin{split}
    \mathrm{Spec}\,\bC[z_1',z_2'] &\rightarrow \bar B\\
    (z_1',z_2')&\rightarrow [1:z_2':z_1':1].
\end{split}
\]
Then from (\ref{gitmoduli}), on the intersection  $\mathrm{Spec}\,\bC[z_1',z_2']\cap\mathrm{Spec}\,\bC[z_1,z_2]$, we have $(z_1',z_2') = (z_1^{-1},z_2)$. The discriminant in this chart is given in Figure \ref{figsecondchart}.

\begin{figure}
\begin{center}
\begin{tikzpicture}
  \draw[->] (-3, 0) -- (5, 0) node[right] {$z'_1 (D_{(0,1)}=C_\infty)$};
  \draw[->] (0, -1.5) -- (0, 4.5) node[above] {$z_2' (C_0=D_{(-1,0)})$};
  \draw[scale=1, domain= -0.9:4.9, smooth, variable=\x, blue] plot ({\x}, {2*(1-\x/2)^2}) node[right] {$C_{con}$};
  \draw[scale=1, domain=-3:5, smooth, variable=\x, red] plot ({\x}, {2}) node[right] {$C_{1}$};
   \draw [fill] (0,2) circle [radius=0.03];
 \node [above right] at (0,2) {\footnotesize{$(0,1)$}};
    \draw [fill] (4,2) circle [radius=0.03];
 \node [above left] at (4,2) {\footnotesize{$(2,1)$}};
\draw [fill] (2,0) circle [radius=0.03];
 \node [below] at (2,0) {\footnotesize{$(1,0)$}};
 \draw [fill] (0,0) circle [radius=0.03];
  %
  \draw [fill] (-1,-1) circle [radius=0.03];
  \node [below ] at (-1,-1) {\footnotesize{$c_0$}};
  \draw [fill] (-0.3,-0.3) circle [radius=0.03];
  \node [below ] at (-0.3,-0.3) {\footnotesize{$c_1$}};
    \draw [fill] (-0.3,0.3) circle [radius=0.03];
  \node [left] at (-0.3,0.3) {\footnotesize{$c_2$}};
    \draw [fill] (-0.3,1.7) circle [radius=0.03];
  \node [below left] at (-0.3,1.7) {\footnotesize{$c_3$}};
      \draw [fill] (-0.3,2.3) circle [radius=0.03];
  \node [left] at (-0.3,2.3) {\footnotesize{$c_3$}};
  \draw[decoration={markings, mark=at position 0.5 with {\arrow{>}}}, postaction={decorate}] (-1,-1) -- (-0.3,-0.3);
  \draw[decoration={markings, mark=at position 0.5 with {\arrow{>}}}, postaction={decorate}] (-0.3,0.3) -- (-0.3,1.7);

   \draw[decoration={markings, mark=at position 0.3 with {\arrow{>}}}, postaction={decorate}] (-0.3,1.7) to [out= -45,in=-45] (0.3,2.3);
  \draw[dashed] (0.3,2.3)  to [out=135,in=135]   (-0.3,1.7);
  \draw[dashed] (-0.3,-0.3) to [out=0,in=0] (-0.3,0.3);
  \draw[decoration={markings, mark=at position 0.7 with {\arrow{>}}}, postaction={decorate}] (-0.3,0.3)  to [out=180,in=180]   (-0.3,-0.3);
  \draw[decoration={markings, mark=at position 0.7 with {\arrow{>}}}, postaction={decorate}] (-0.3,-0.3) to [out=-90,in=-90] (0.3,-0.3);
  \draw[dashed] (0.3,-0.3)  to [out=90,in=90]   (-0.3,-0.3);
  \draw[dashed] (-0.3,1.7) to [out=0,in=0] (-0.3,2.3);
  \draw[decoration={markings, mark=at position 0.7 with {\arrow{>}}}, postaction={decorate}] (-0.3,2.3)  to [out=180,in=180]   (-0.3,1.7);
   \draw[dashed] (-0.3,2.3) to [out=0,in=0] (-0.3,2.9);
  \draw[decoration={markings, mark=at position 0.7 with {\arrow{>}}}, postaction={decorate}] (-0.3,2.9)  to [out=180,in=180]   (-0.3,2.3);
     \draw[decoration={markings, mark=at position 0.3 with {\arrow{>}}}, postaction={decorate}] (-0.3,-0.3) to [out= -45,in=-45] (0.3,0.3);
  \draw[dashed] (0.3,0.3)  to [out=135,in=135]   (-0.3,-0.3);
     \node [right] at (0.3,-0.3) {\footnotesize{$l_6$}};
     \node [left] at (-0.3,-0.3) {\footnotesize{$l_2'$}};
      \node [above] at (-0.3,2.9) {\footnotesize{$l_4''$}};
    \node [right] at (0.2,1.8) {\footnotesize{$l_7$}};
        \node [right] at (0.3,0.3) {\footnotesize{$l_8$}};
        \node [left] at (-0.4,1.8) {\footnotesize{$l_3''$}};
\end{tikzpicture}
\end{center}
\caption{The discriminant locus in the affine chart $\mathrm{Spec}\,\bC[z_1',z_2']$.}
\label{figsecondchart}
\end{figure}
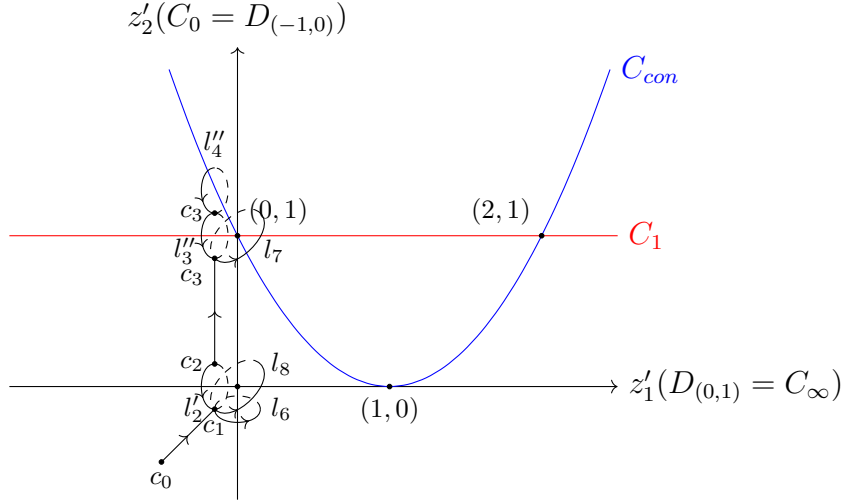

In this affine chart, we choose a base point $c_0$ at $(-1,-1)$. The nilpotent cones we want to calculate in this chart are those associated to the point $(0,0)$ and $(0,1)$, the rest cones on this chart are already calculated in the chart $\mathrm{Spec}\,\bC[z_1,z_2]$. As we discussed in Section \ref{secdiscloci}, for the point $(0,0)$ and $(0,1)$, we need to blow them up, this results in five cones.

The Gauss-Manin connection in coordinates $(z_1',z_2')$ can be calculated by a change of coordinates from  $(z_1,z_2)$. This is because the restriction of the transition map $g(z_1,z_2) = (z_1^{-1},z_2)$ to the smooth locus is a local isomorphism, thus $g^*(\mathcal D_{\mathrm{Spec}\,\bC[z_1,z_2]}/I) \cong \mathcal D_{\mathrm{Spec}\,\bC[z_1',z_2']}/g^*I$ (restricted to the smooth locus).

We now calculate the monodromy operators at the point $c_0$ using this connection. After that, since $g^{-1}(-1,-1) = (-1,-1)$, we take a path in the affine chart $\mathrm{Spec}\,\bC[z_1,z_2]$ between $c_0$ and $b_0$ to parallel transport the result back to $b_0$. More precisely, we take $\beta_{c_0b_0}$ to be the composition of the line from $(-1,-1)$ to $(-10^{-4},-10^{-4}) $ and the anti-clockwise loop from $(10^{-4},10^{-4})$ to $(10^{-4},10^{-4})$.

Notice that there are only four monodromy operators we need to calculate in this chart. This is because, by our choice of $\beta_{c_0b_0}$,  the loop $\beta_{b_0c_0}\cdot\beta_{c_0c_1}\cdot l_2'\cdot \beta_{c_1c_0}\cdot \beta_{c_0b_0}$ is homotopic to $l_2$, and $\beta_{b_0c_0}\cdot\beta_{c_0c_1}\cdot\beta_{c_1c_2}\cdot\beta_{c_2c_3}\cdot l_3''\cdot\beta_{c_3c_2}\cdot\beta_{c_2c_1}\cdot \beta_{c_1c_0}\cdot \beta_{c_0b_0}$ is homotopic to $l_3$. However, one should note, that the loop used to calculate the monodromy around $C_{con}$ in this chart is different than those in $\mathrm{Spec}\,\bC[z_1,z_2]$. The choice of loops is drawn in Figure \ref{figsecondchart}.  Passing to $\tilde B$, we need to blow up $(0,0)$ and $(0,1)$ once and $(1,0)$ twice. The coordinate of $c_1$ is $(-10^{-4},-10^{-4})$ and the rest of the numerical data can be read off it. We note for the exceptional divisors $D_{(-1,-1)}$ and $E_0$, the loop will be the composition of $l_8$ and $l_7$ with paths connecting them to $c_0$.
We denote the resulting monodromy matrices as $S_{C_0}$, $S_{D_{(-1,-1)}}$, $S_{con_2}$, and $S_{E_0}$.

For the second affine chart, we choose it to be
$\mathrm{Spec}\,\bC[z_1'',z_2'']$
\[
\begin{split}
    \mathrm{Spec}\,\bC[z_1'',z_2''] &\rightarrow \bar B\\
    (z_1'',z_2'')&\rightarrow [1:1:z_1'':z_2''].
\end{split}
\] 
For the transition map on the intersection  $\mathrm{Spec}\,\bC[z_1'',z_2'']\cap\mathrm{Spec}\,\bC[z_1,z_2]$, we have $(z_1'',z_2'') = \pm( z_1^{-1}z_2^{-1/2}, z_2^{-1/2})$. The sign $\pm$ shows up due to the fact that this map is not a local isomorphism, since $[1:1:z_1'':z_2''] = [1:1:-z_1''
:-z_2'']$ under the equivalence relation defined in (\ref{gitmoduli}). Instead, the map above is locally a double cover. 

Nonetheless, the Gauss-Manin connection can still be calculated by coordinate transformation. This is due to the fact that the transition map is unramified on the smooth locus, thus flat. The discriminant locus in this chart is given in Figure \ref{figthirdchart}. We only draw the upper half plane of this chart due to the fact that this chart is locally a double cover of $\bar B$. For example, $(1,0)$ and $(-1,0)$ represent the same point on $\bar B$. Moreover,  $C_{con}$ lifted to two irreducible components, and half arcs represent loops in $B$.

\begin{figure}
\begin{center}
\begin{tikzpicture}
  \draw[->] (-3, 0) -- (4.5, 0) node[right] {$z''_1 (D_{(1,-2)})$};
  \draw[->] (0, -1) -- (0, 4.8) node[above] {$z_2'' (C_0 = D_{(-1,0)})$};
  \draw[scale=1, domain= -2:2.6, smooth, variable=\x, blue] plot ({\x}, {2+\x}) node[below right] {$C_{con}$};
  \draw[scale=1, domain= 2:4.5, smooth, variable=\x, blue] plot ({\x}, {-2+\x})
  node[above] {$C_{con}$};
  \draw[scale=1, domain=-3:4.5, smooth, variable=\x, red] plot ({\x}, {2}) node[right] {$C_{1}$};
   \draw [fill] (0,2) circle [radius=0.03];
 \node [above left] at (0,2) {\footnotesize{$(0,1)$}};
 \draw [fill] (-2,0) circle [radius=0.03];
 \node [below left] at (-2,0) {\footnotesize{$(-1,0)$}};
    \draw [fill] (4,2) circle [radius=0.03];
 \node [above left] at (4,2) {\footnotesize{$(2,1)$}};
\draw [fill] (2,0) circle [radius=0.03];
 \node [below] at (2,0) {\footnotesize{$(1,0)$}};
 \draw [fill] (-0.3,-0.3) circle [radius=0.03];
 \node [below left] at (-0.3,-0.3) {\footnotesize{$d_0'$}};
 \draw [fill] (0.3,0.3) circle [radius=0.03];
 \node [above] at (0.3,0.3) {\footnotesize{$d_1$}};
 \draw [fill] (-3,-0.8) circle [radius=0.03];
  \node [below] at (-3,-0.8) {\footnotesize{$d_0$}};
 \draw[decoration={markings, mark=at position 0.7 with {\arrow{>}}}, postaction={decorate}] (-0.3,-0.3)  to [out=180,in=180]   (-0.3,0.3);
   \draw[dashed] (-0.3,0.3)  to [out=0,in=0]   (-0.3,-0.3);
  \draw[decoration={markings, mark=at position 0.7 with {\arrow{>}}}, postaction={decorate}] (-0.3,-0.3) to [out=-90,in=-90] (0.3,-0.3);
   \draw[dashed] (0.3,-0.3)  to [out=90,in=-90]   (-0.3,-0.3);
   \draw[decoration={markings, mark=at position 0.5 with {\arrow{>}}}, postaction={decorate}] (-0.3,-0.3) to [out= -45,in=-45] (0.3,0.3);
   \draw[decoration={markings, mark=at position 0.5 with {\arrow{>}}}, postaction={decorate}] (-3,-0.8) -- (-0.3,-0.3);
\end{tikzpicture}
\end{center}
\caption{This figure shows the discriminant locus in the affine chart $\mathrm{Spec}\,\bC[z_1'',z_2'']$. }
\label{figthirdchart}
\end{figure}
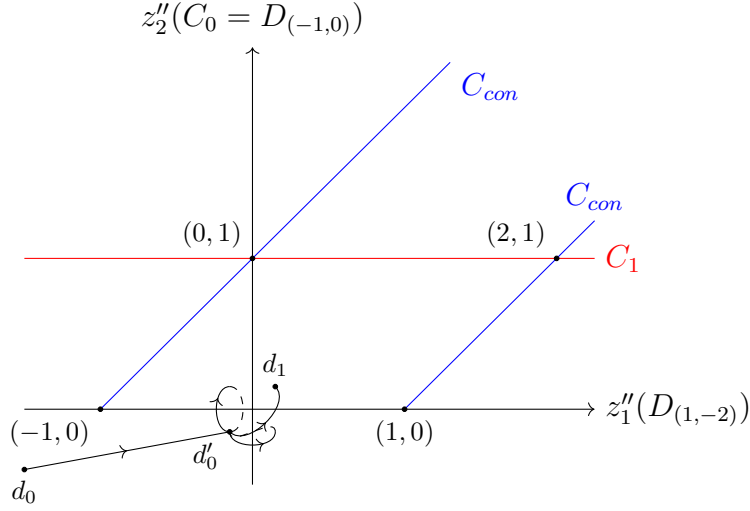

The two nilpotent cones we need to calculate in this affine chart are those associated to the blowup of the origin. We choose the base point $d_0=(-1,-1/10)$. This is $(1/10,100)$ in $\mathrm{Spec}\,\bC[z_1,z_2]$. Then we chose a path connecting $d_0$ to $b_0$.
The choice of the loop is drawn in Figure \ref{figthirdchart}, where the half loop represents a full loop in $B$. The numerical data is not so important in this case, since as it turns out, the monodromy around $C_0$ is the 4-th root of unity, which becomes trivial after passing to some finite cover of $B$.

For the last affine chart, we choose
$\mathrm{Spec}\,\bC[z_1''',z_2''']$
\[
\begin{split}
    \mathrm{Spec}\,\bC[z_1''',z_2'''] &\rightarrow \bar B\\
    (z_1''',z_2''')&\rightarrow [z_1''':1:1:z_2'''].
\end{split}
\] 
Similar to the last affine chart, this is also a double cover. The transition map on the intersection  $\mathrm{Spec}\,\bC[z_1''',z_2''']\cap\mathrm{Spec}\,\bC[z_1,z_2]$ is given by $(z_1''',z_2''') = \pm(z_1^{-1}z_2^{-1/2},z_2^{-1/2})$. The discriminant locus in the upper half plane is given in Figure \ref{figfourthchart}.

\begin{figure}
\begin{center}
\begin{tikzpicture}
  \draw[->] (-3, 0) -- (4.5, 0) node[right] {$z'''_1 (D_{(1,-2)})$};
  \draw[->] (0, -1) -- (0, 4.8) node[above] {$z_2''' (D_{(1,0)})$};
  \draw[scale=1, domain= 1.5:4.0, smooth, variable=\x, blue] plot ({\x}, {2+4/\x}) node[right] {$C_{con}$};
  \draw[scale=1, domain= 0.6:2, smooth, variable=\x, blue] plot ({\x}, {-2+4/\x})
  node[above right] {$C_{con}$};
  \draw[scale=1, domain= -2:-3, smooth, variable=\x, blue] plot ({\x}, {2+4/\x})
  node[left] {$C_{con}$};
  \draw[scale=1, domain=-3:4.5, smooth, variable=\x, red] plot ({\x}, {2}) node[right] {$C_{1}$};
   \draw [fill] (0,2) circle [radius=0.03];
 \node [above left] at (0,2) {\footnotesize{$(0,1)$}};
 \draw [fill] (-2,0) circle [radius=0.03];
 \node [below left] at (-2,0) {\footnotesize{$(-1,0)$}};
    \draw [fill] (1,2) circle [radius=0.03];
 \node [above right] at (1,2) {\footnotesize{$(1/2,1)$}};
\draw [fill] (2,0) circle [radius=0.03];
 \node [below] at (2,0) {\footnotesize{$(1,0)$}};
 \draw [fill] (-0.3,0.3) circle [radius=0.03];
 \node [above] at (-0.3,0.3) {\footnotesize{$e_0'$}};
 \draw [fill] (0.3,-0.3) circle [radius=0.03];
 \node [below] at (0.3,-0.3) {\footnotesize{$e_1$}};
 \draw [fill] (-2,1) circle [radius=0.03];
  \node [below] at (-2,1) {\footnotesize{$e_0$}};
 \draw[decoration={markings, mark=at position 0.7 with {\arrow{>}}}, postaction={decorate}] (-0.3,0.3)  to [out=180,in=180]   (-0.3,-0.3);
    \draw[dashed] (-0.3,-0.3)  to [out=0,in=0]   (-0.3,0.3);
  \draw[decoration={markings, mark=at position 0.7 with {\arrow{>}}}, postaction={decorate}] (-0.3,0.3) to [out=-90,in=-90] (0.3,0.3);
     \draw[dashed] (0.3,0.3)  to [out=90,in=90]   (-0.3,0.3);
   \draw[decoration={markings, mark=at position 0.5 with {\arrow{>}}}, postaction={decorate}] (-0.3,0.3) to [out= -90,in=-180] (0.3,-0.3);
   \draw[decoration={markings, mark=at position 0.5 with {\arrow{>}}}, postaction={decorate}] (-2,1) -- (-0.3,0.3);
\end{tikzpicture}
\end{center}
\caption{This figure shows the discriminant locus in the affine chart $\mathrm{Spec}\,\bC[z_1''',z_2''']$. }
\label{figfourthchart}
\end{figure}
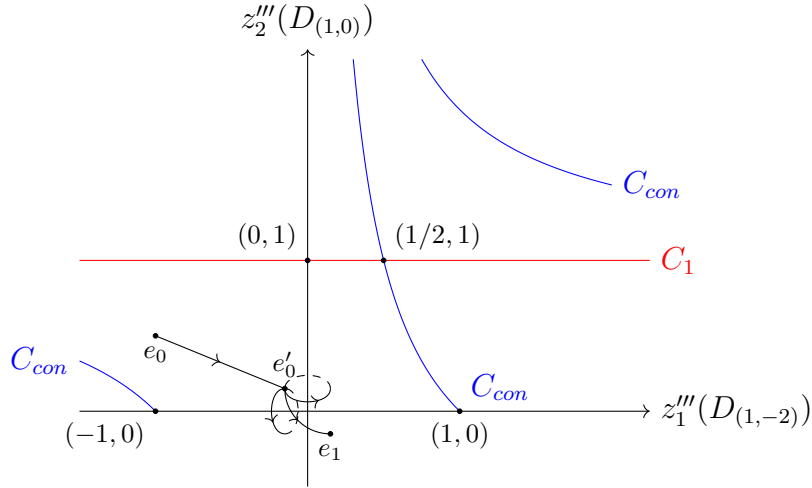

The two nilpotent cones we need to calculate in this affine chart are those between the coordinate axes and the blow-up of the origin. The choice of the loop is drawn in Figure \ref{figfourthchart}, where the point $e_0=(-1,1/2)$ is $(-1/2,4)$ in the affine chart $\mathrm{Spec}\,\bC[z_1,z_2]$. As we will see, the two nilpotent cones in this chart are not interesting since they both degenerate into 1-dimensional cones.

\subsection{Monodromy matrices in integral symplectic basis.}
Now, we want to represent the numerical complex monodromy matrices in the integral symplectic basis introduced in Proposition \ref{prop2.2sec2}. To do so, we need to find the transformation matrix between bases. That is, we would like to calculate the mirror map. However, for our purpose, it is sufficient to transform the basis $\{\omega_j\}$ to any integral symplectic basis. In this section, we will calculate the mirror map from the basis $\{\omega_j\}$ to an integral symplectic basis $\{\tilde\omega_j\}$ that the matrix representation of monodromies at MUM point has the same form as those for the integral symplectic basis in 
Proposition \ref{prop2.2sec2}.

This can be done by an algorithm that is a rigorization of those in \cite{candelas1994mirror1}. We note although the calculation is based on numerical data, the result is precise. That is, the integral matrix we give the reference equation is correct, without error. This is because, for example, given $x$ is an integer, and by numerical calculation, we get $x = 3.991$. If we also know the numerical error is less than $10^{-1}$, then $x$ has to be $4$.

\subsubsection{The mirror map}
To begin with, let $R$ be a $6\times 6$ matrix with  indeterminant  complex entries $R_{ij}$, that corresponds to the transformation matrix from basis $\{\omega_j\}$ to an integral symplectic basis $\{\tilde\omega_j:=R\omega_j\}$. We require $\{\tilde\omega_j\}$ to satisfy an additional condition that the matrix representation of monodromies at MUM point coincide with those for the integral symplectic basis in 
Proposition \ref{prop2.2sec2}. We want to fix $R_{ij}$ by the following procedure.

Since the matrices representations $S_{m_1},S_{m_2}$ in (\ref{mummonnum}) transform into $T_{m_1},T_{m_2}$ under the change of basis, by Proposition \ref{prop2.2sec2}, we have
$$
T_{m_i} = R S_{m_i}R^{-1}\Leftrightarrow T_{m_i}R =RS_{m_i},\quad i=1,2
$$
We can use these equations to fix most of the entries of the matrix $R$. The result is
\begin{equation}\label{Rpartfix}
R = \left[\begin{array}{cccccc}
1 & R_{12} & R_{13} & R_{14} & R_{15} & R_{16} \\
0 & -1 & 0 & 4-\frac{C_{11}}{2}-4R_{13} & -4R_{12}-2 & R_{26}\\
0 & 0 & -1 & -4R_{12}-2 & 0 & 2R_{12}+R_{15}-2 \\
0 & 0 & 0 & 0 & 0 & -1 \\
0 & 0 & 0 & 1 & 0 & -R_{12} \\
0 & 0 & 0 & -2 & 1 & -R_{13}
\end{array}\right],
\end{equation}
where $R_{26} = 2R_{13}+R_{14}+2R_{15}+\frac{C_{11}}{2}R_{12}-\frac{14}{3} $, and we have rescaled the basis $\omega$ by a constant $c_0$ to make $R_{11} = 1$. We note in particular, $\det(R)=-1$, then from the intersection matrix of basis $\{\omega_j\}$ near origin \eqref{Interori}, we know the scaling factor $c_0=\pm 2i\pi^3\cdot c$. Now,
there are only 5 complex ($R_{12}$ to $R_{16}$) and 1 integral ($C_{11}$) indeterminants left. However, these indeterminants are not independent of each other. To see this, we note the intersection matrix for the basis $R\cdot \omega$ is $Q$, and then from it, we can solve the intersection matrix of the basis $\{\omega_j\}$. This implies the intersection matrix with respect to the basis $\{\omega_j\}$ is 
$$
I' =
\left[
\begin{array}{cccccc}
   0  & I'_{12} & 4R_{12}^2-2R_{15}+2 & 0 & 0& -1\\
   -I'_{12} & 0 &0 & -1 & -2 &0\\
    -4R_{12}^2+2R_{15}-2 & 0 & 0 & 0 & -1 & 0 \\
    0 & 1 & 0 & 0 &0 &0\\
    0 & 2 & 1 & 0 & 0 & 0\\
    1 &0 & 0 & 0 & 0 & 0
\end{array}
\right],
$$
where
$$
I'_{12} = -2R_{14}-4R_{15}+8R_{12}R_{13}+8R_{12}^2+\frac{14}{3}.
$$

On the other hand, we have calculated the intersection matrix with respect to the basis $\{\omega_j\}$ in  \eqref{intermat}. Comparing these two expressions, we get $I'_{12} = I'_{13} = 0$, which gives
\begin{equation}\label{constraint45}
\begin{split}
R_{15} & = 2R_{12}^2+1,\\
R_{14} & = 4R_{12}R_{13}+\frac{1}{3}.
\end{split}
\end{equation}

Now, since the integral symplectic basis at which we are aiming is not unique, we would like to see how many degrees of freedom are left. Consider further changing the basis $R\cdot \omega$ by an integral symplectic transformation $U$, i.e., we get a new basis $U\cdot R\cdot \omega$. If the monodromy matrices are still $S_{m_1}$ and $S_{m_2}$ in this new integral symplectic basis, then $U\cdot R$ can only differ from $R$ by changing the indeterminants $R_{1i}$ by $R_{1i}\rightarrow R_{1i}-d_i$, $i=2,\cdots,6$, for some $d_i\in\bC$. Now, if we combine the condition $U$ is integral symplectic and the condition $S_{m_1}$ and $S_{m_2}$ are unchanged, then $U$ has to take the following form
$$
U = \left[\begin{array}{cccccc}
1 & d_2 & d_3 & U_{14} & U_{15} & 2d_2^2+2d_2 \\
0 & 1 & 0 & 2d_3 + \frac{C_{11}d_ 2}{2} - 4d_2d_3 - 4d_2^2 & 8d_2 + 4d_3 & 4d_2\\
0 & 0 & 1 & -2d_2^2+2d_2 & 4d_2 & 0 \\
0 & 0 & 0 & 1 & 0 & 0 \\
0 & 0 & 0 & -d_2 & 1 & 0 \\
0 & 0 & 0 & -d_3 & 0 & 1
\end{array}\right],
$$
where 
\[
\begin{split}
U_{14} & =  d_6-(2d_3 + \frac{14d_2}{3} - d_2R_{14} - 2d_2R_{15} - d_3R_{15} + 4d_2^2R_{12} + 2d_2^2R_{13} + 4d_2d_3R_{12}),\\
U_{15} & = 2d_3 + \frac{C_{11}d_2}{2} + 4d_2d_3 + 4d_2^2. 
\end{split}
\]
In particular, by \eqref{constraint45} or the symplecticity of $U$, one must has
\[
\begin{split}
d_4 & = -4d_2d_3 + 4d_2R_{13} + 4d_3R_{12},\\
d_5 & = 4R_{12}d_2-2d_2^2.\\
\end{split}
\]

Now, since $U$ is integral, we have $d_2,d_2\in\bZ$, and $d_6$ is determined up to an integer $U_{14}$. Therefore, by a suitable integral symplectic transformation of bases  $R$, we can always make $-\frac12\leq\mathrm{Re}(R_{1j})<\frac12$, $j=2,3,6$. These conditions uniquely determine $U$, thus we arrive at the following lemma.
\begin{lemma}
There exists a unique integral symplectic basis of $H^3(X_b,\bZ)$, such that the transformation  matrix from the basis $\{\omega_j\}$ \eqref{basisomega} to it has the form $R$, defined in \eqref{Rpartfix}, with $-\frac12\leq\mathrm{Re}(R_{1j})<\frac12$, $j=2,3,6$.
\end{lemma}

Next, we notice, that we have the following conditions that can be used to find the explicit values of $R_{12},R_{13},R_{16}$.
\begin{enumerate}
    \item For each complex numerical monodromy operator $S_{\alpha}$ we calculated in Section \ref{secmmbo}, we have the condition $T_{\alpha}:=RS_{\alpha}R^{-1}$ is an integral  matrix.
    \item $T_{\alpha}$ is symplectic, i.e., $T_{\alpha}'QT_{\alpha}-Q =0$.
    \item $T_{\alpha}$ is quasi-unipotent, i.e., $(T_{\alpha}^m-\mathrm{Id})^n=0$ for some $m,n\in\bZ_{+}$.
    \item For each pair of intersecting normal crossing divisors, their monodromies commute, i.e.,$[T_{\alpha},T_{\beta}]=0$.
\end{enumerate}
In fact, we have a lot more equations than we actually need. More precisely, consider the numerical matrix $S_{con_1}$ in \eqref{con1numma}. Keeping six significant digits, the entries (5,1) and (6,1) of the matrix $RS_{con_1}R^{-1}$ look like
\begin{equation}
    \begin{split}
        -(0.999111 + 0.000257384i)R_{12} - 0.000421127 + 2.3466i, \\
    -(0.999111 + 0.000257384i)R_{13} - 0.00021644 + 1.68483i.
    \end{split}
\end{equation}

Estimating by using \eqref{GMcon1} and \eqref{GMcon2}, one can easily show the numerical values above have an error less than $10^{-3}$. Now, by the integrality of these terms, and the fact that $|\mathrm{Re}(R_{12})|\leq\frac12$ and $|\mathrm{Re}(R_{13})|\leq\frac12$, we have  $\mathrm{Re}(R_{12}) = \mathrm{Re}(R_{13}) = 0$. Therefore, with errors less than $10^{-2}$, we have $$
R_{12} = 2.3466i,\quad R_{13} =  1.6848i.
$$
Substitute this back, the (1,1) term in $RS_{con_1}R^{-1}$ becomes
$$
(0.999111 + 0.000257384i)R_{16} + 1.00762 - 29.3909i,
$$
then since $|\mathrm{Re}(R_{16})|\leq \frac12$, we have $\mathrm{Re}(R_{16}) = 0$, and $R_{16} = 29.3909i$.

Furthermore, we notice that, by a half-integral symplectic transformation
$$
P_1 =\left[\begin{array}{cccccc}1 & 0 & 0 & 0 & 0 & 0 \\0 & 1 & 0 & 0 & \frac{C_{11}}{2} & 0 \\0 & 0 & 1 & 0 & 0 & 0 \\0 & 0 & 0 & 1 & 0 & 0 \\0 & 0 & 0 & 0 & 1 & 0 \\0 & 0 & 0 & 0 & 0 & 1\end{array}\right],
$$
we can cancel $\frac{C_{11}}{2}$ in $R$. In other words, if we allow ourselves to work with a half-integral basis, we can take $C_{11}$ to be any integer. In what follows we will take $C_{11} = 0$. The matrix $R$ now is fully determined. Keeping four significant digits, we have
\[
R = \left[
\begin{array}{rrrrrr}
    1.0 & 2.347i & 1.685i & -15.48 & -10.01 & 29.39i \\
    0 & -1.0 & 0 & 4.0 - 6.739i & - 2.0 - 9.386i & - 40.17 + 3.37i \\
    0 & 0& -1.0 & - 2.0 - 9.386i & 0  & - 12.01 + 4.693i  \\
    0 & 0  & 0 & 0 & 0 & -1.0  \\
    0 & 0  & 0 & 1.0 & 0 & -2.347i  \\
    0 & 0  & 0  & -2.0 & 1.0 & -1.685i
\end{array}\right].
\]

\subsubsection{Integral symplectic monodromy matrices}\label{secintsmm}
Now, using $R$, for any numerical monodromy matrix $S_\alpha$, we can calculate the resulting half-integral symplectic monodromy matrix $T_{\alpha}= RS_{\alpha}R^{-1}$. It turns out all the monodromy matrices are integral instead of half-integral. Moreover, all of these monodromy operators $T_\alpha$ are quasi-unipotent, i.e., they satisfy $(T_\alpha^{m_\alpha}-\mathrm{Id})^{n_\alpha}=0$ for some $m_\alpha,n_\alpha\in\bZ_+$, and $m_\alpha,n_\alpha$ are the smallest among all the possible choices. This is consistent with Deligne's monodromy theorem.

For $T_{C_0}$, we have
$$
T_{C_0} = \left[\begin{array}{rrrrrr}-1 & -1 & 2 & -2 & 8 & -4 \\2 & 3 & -2 & -8 & 4 & 12 \\0 & 0 & 1 & 0 & 4 & 0 \\0 & 0 & 1 & -3 & 6 & 0 \\0 & 0 & 0 & -1 & 1 & 0 \\-1 & -1 & 1 & 2 & 2 & -3\end{array}\right].
$$
As we mentioned at the end of the last section, this is a 4-th root of unity, thus the nilpotent cones formed by $\log(T_{C_0})$ with other operators are 1-dimensional. Similarly, since $T_{D_{(1,-2)}} = \mathrm{Id}_{6\times 6}$, and
$$
T_{D_{(1,-1)}} =  \left[\begin{array}{rrrrrr} 1 & 0 & -1 & 2 & -2 & 0 \\0 & 1 & 0 & 2 & -4 & 4 \\0 & 1 & -1 & -4 & -8 & -2 \\0 & 0 & 0 & 1 & 0 & 0 \\0 & 0 & 0 & 1 & 1 & 1 \\0 & 0 & 0 & -1 & 0 & -1\end{array}\right]
$$
is a square root of $T_{m_1}$, the nilpotent cones generated by $\log(T_{m_1})$, $\log(T_{D_{(1,-1)}})$ and $\log(T_{D_{(1,-1)}})$, $\log(T_{D_{(1,-2)}}^2)$ are both 1-dimensional.

Therefore, only 7 nilpotent cones that are formed by 7 different monodromy operators will be used to construct the candidate nilpotent fan. These  monodromy operators written in the form $T_\alpha^{m_\alpha}$ are

\begin{align*}
&T_{m_1} = 
\left[\begin{array}{rrrrrr}1 & -1 & 0 & 6 & 4 & 0 \\0 & 1 & 0 & -4 & -8 & -4 \\0 & 0 & 1 & -4 & -4 & 0 \\0 & 0 & 0 & 1 & 0 & 0\\0 & 0 & 0 & 1 & 1 & 0\\ 0 & 0 & 0 & 0 & 0 & 1\end{array}\right],\,
&&T_{m_2} = 
\left[\begin{array}{rrrrrr}1 & 0 & -1 & 2 & -2 & 0 \\0 & 1 & 0 & -2 & -4 & 0 \\0 & 0 & 1 & 0 & 0 & 0 \\0 & 0 & 0 & 1 & 0 & 0\\0 & 0 & 0 & 0 & 1 & 0\\ 0 & 0 & 0 & 1 & 0 & 1\end{array}\right],\\
&T_{C_1}^2=
\left[\begin{array}{rrrrrr}1 & 0 & 0 & 0 & 0 & 0 \\0 & 1 & 0 & 0 & 0 & 0 \\0 & 0 & 1 & 0 & 0 & 4 \\0 & 0 & 0 & 1 & 0 & 0\\0 & 0 & 0 & 0 & 1 & 0\\ 0 & 0 & 0 & 0 & 0 & 1\end{array}\right],\,
&&T_{con_1} = 
\left[\begin{array}{rrrrrr}1 & 0 & 0 & 0 & 0 & 0 \\0 & 1 & 0 & 0 & 0 & 0 \\0 & 0 & 1 & 0 & 0 & 0 \\-1 & 0 & 0 & 1 & 0 & 0\\0 & 0 & 0 & 0 & 1 & 0\\ 0 & 0 & 0 & 0 & 0 & 1\end{array}\right],\\
&T_{E_2}^2=
\left[\begin{array}{rrrrrr}1 & 0 & 0 & 0 & 0 & 0 \\0 & 1 & -4 & 0 & -24 & 0 \\0 & 0 & 1 & 0 & 0 & 0 \\0 & 0 & 0 & 1 & 0 & 0\\0 & 0 & 0 & 0 & 1 & 0\\ 0 & 0 & 2 & 0 & -4 & 1\end{array}\right],\,
&&T_{E_0}^4 = 
\left[\begin{array}{rrrrrr}-15 & -8 & 0 & 32 & 0 & -16 \\0 & 1 & 0 & 0 & 0 & 0 \\8 & 4 & 1 & -16 & 0 & 8 \\-8 & -4 & 0 & 17 & 0 & -8\\-4 & -2 & 0 & 8 & 1 & -4\\ 0 & 0 & 0 & 0 & 0 & 1\end{array}\right],
\end{align*}
$$T_{con_2}=
\left[\begin{array}{rrrrrr}-1 & -2 & 2 & 4 & 4 & -8 \\-2 & -1 & 2 & 4 & 4 & -8 \\4 & 4 & -3 & -8 & -8 & 16 \\-1 & -1 & 1 & 3 & 2 & -4\\-1 & -1 & 1 & 2 & 3 & -4\\ 1 & 1 & -1 & -2 & -2 & 5\end{array}\right].
$$

To simplify some of the monodromy matrices, we consider an integral symplectic change of basis induced by the following matrix
$$
P_2 =  
\left[\begin{array}{rrrrrr}1 & 0 & 0 & -2 & 0 & 0 \\0 & -1 & 0 & 0 & 0 & -4 \\0 & 0 & -1 & 0 & -4 & 0 \\0 & 0 & 0 & 1 & 0 & 0\\0 & 0 & 0 & 0 & -1 & 0\\ 0 & 0 & 0 & 0 & 0 & -1\end{array}\right].
$$
with respect to this basis, we have
\begin{align*}
&T_{m_1} = 
\left[\begin{array}{rrrrrr}1 & 1 & 0 & 6 & -4 & -4 \\0 & 1 & 0 & 4 & -8 & -4 \\0 & 0 & 1 & 0 & -4 & 0 \\0 & 0 & 0 & 1 & 0 & 0\\0 & 0 & 0 & -1 & 1 & 0\\ 0 & 0 & 0 & 0 & 0 & 1\end{array}\right],\,
&&T_{m_2} = 
\left[\begin{array}{rrrrrr}1 & 0 & 1 & 2 & -2 & 0 \\0 & 1 & 0 & -2 & -4 & 0 \\0 & 0 & 1 & 0 & 0 & 0 \\0 & 0 & 0 & 1 & 0 & 0\\0 & 0 & 0 & 0 & 1 & 0\\ 0 & 0 & 0 & -1 & 0 & 1\end{array}\right],\\
&T_{C_1}^2=
\left[\begin{array}{rrrrrr}1 & 0 & 0 & 0 & 0 & 0 \\0 & 1 & 0 & 0 & 0 & 0 \\0 & 0 & 1 & 0 & 0 & 4 \\0 & 0 & 0 & 1 & 0 & 0\\0 & 0 & 0 & 0 & 1 & 0\\ 0 & 0 & 0 & 0 & 0 & 1\end{array}\right],\,
&&T_{con_1} = 
\left[\begin{array}{rrrrrr}3 & 0 & 0 & 4 & 0 & 0 \\0 & 1 & 0 & 0 & 0 & 0 \\0 & 0 & 1 & 0 & 0 & 0 \\-1 & 0 & 0 & -1 & 0 & 0\\0 & 0 & 0 & 0 & 1 & 0\\ 0 & 0 & 0 & 0 & 0 & 1\end{array}\right],\\
&T_{E_2}^2=
\left[\begin{array}{rrrrrr}1 & 0 & 0 & 0 & 0 & 0 \\0 & 1 & 4 & 0 & -24 & 0 \\0 & 0 & 1 & 0 & 0 & 0 \\0 & 0 & 0 & 1 & 0 & 0\\0 & 0 & 0 & 0 & 1 & 0\\ 0 & 0 & 2 & 0 & -4 & 1\end{array}\right],\,
&&T_{E_0}^4 = 
\left[\begin{array}{rrrrrr}1 & 0 & 0 & 0 & 0 & 0 \\0 & 1 & 0 & 0 & 0 & 0 \\8 & -4 & 1 & 0 & 0 & 8 \\-8 & 4 & 0 & 1 & 0 & -8\\4 & -2 & 0 & 0 & 1 & 4\\ 0 & 0 & 0 & 0 & 0 & 1\end{array}\right],
\end{align*}
$$T_{con_2}=
\left[\begin{array}{rrrrrr}1 & 0 & 0 & 0 & 0 & 0 \\-2 & 3 & -2 & 0 & 4 & 0 \\0 & 0 & 1 & 0 & 0 & 0 \\-1 & 1 & -1 & 1 & 2 & 0\\1 & -1 & 1 & 0 & -1 & 0\\ -1 & 1 & -1 & 0 & 2 & 1\end{array}\right].
$$

\begin{remark}\label{ptfe}
    After passing to a finite {\'e}tale cover $U\rightarrow B$, and taking $\varphi: Y \rightarrow \bar B$ to be the normalization of $\bar B$ in the function field $k(U)$, we can assume $\Gamma$ is neat.
    
    The morphism $\varphi$ might be tamely ramified at $\bar B\setminus B$, which means $Y$ locally looks like $y^n-f(x_1,x_2)$, where $f$ is the local defining equation of the boundary divisor, and $\varphi$ can be singular over the normal-crossing points. Locally, over a normal-crossing point, we can write the local defining equation for $Y$ as $y^n-x_1x_2$. As shown in \cite{deng2021extension}, blowing up at $(y,x_1,x_2) = (0,0,0)$ will only subdivide the cone generated by the logarithmic of unipotent monodormy matrices $T_1^{m_1},T_2^{m_2}, m_1,m_2\in\bZ_+$ around $x_1=0$ and $x_2=0$.

    Therefore, we can take the relevant monodromy matrices to be some positive power of themselves so that they are unipotent, as already indicated above.
\end{remark}

\subsection{Nilpotent cones and mixed Hodge structures}

Using the explicit formula for the monodromy operators, we can now form the nilpotent cones that generate our candidate nilpotent fan via the adjoint action.
Since, we have $\log(T_\alpha') = \frac{4}{m_\alpha}\log(T_\alpha^{m_\alpha})$, and the nilpotent cone is generated over $\bQ$, we can use the matrices $\log(T_\alpha^{m_\alpha})$ as the generators of the nilpotent cones.

\subsubsection{Logarithmic Monodromies}

For future reference, we label the logarithmic monodromy operators as $N_1 := \log(T_{m_1})$, $N_2 := \log(T_{m_2})$, $N_3 := \log(T_{C_1}^2)$, $N_4 := \log(T_{con_1})$, $N_5 := \log(T_{E_2}^2)$, $N_6 := \log(T_{E_0}^4)$, $N_7 := \log(T_{con_2})$. The explicit forms are as follows
\begin{align*}
&N_1 = 
\left[\begin{array}{rrrrrr}0 & 1 & 0 & 14/3 & 0 & -2 \\0 & 0 & 0 & 0 & -8 & -4 \\0 & 0 & 0 & -2 & -4 & 0 \\0 & 0 & 0 & 0 & 0 & 0\\0 & 0 & 0 & -1 & 0 & 0\\ 0 & 0 & 0 & 0 & 0 & 0\end{array}\right],\,
&&N_2 = 
\left[\begin{array}{rrrrrr}0 & 0 & 1 & 2 & -2 & 0 \\0 & 0 & 0 & -2 & -4 & 0 \\0 & 0 & 0 & 0 & 0 & 0 \\0 & 0 & 0 & 0 & 0 & 0\\0 & 0 & 0 & 0 & 0 & 0\\ 0 & 0 & 0 & -1 & 0 & 0\end{array}\right],\\
& N_3=
\left[\begin{array}{rrrrrr}0 & 0 & 0 & 0 & 0 & 0 \\0 & 0 & 0 & 0 & 0 & 0 \\0 & 0 & 0 & 0 & 0 & 4 \\0 & 0 & 0 & 0 & 0 & 0\\0 & 0 & 0 & 0 & 0 & 0\\ 0 & 0 & 0 & 0 & 0 & 0\end{array}\right],\,
&& N_4 = 
\left[\begin{array}{rrrrrr}2 & 0 & 0 & 4 & 0 & 0 \\0 & 0 & 0 & 0 & 0 & 0 \\0 & 0 & 0 & 0 & 0 & 0 \\-1 & 0 & 0 & -2 & 0 & 0\\0 & 0 & 0 & 0 & 0 & 0\\ 0 & 0 & 0 & 0 & 0 & 0\end{array}\right],\\
&N_5=
\left[\begin{array}{rrrrrr}0 & 0 & 0 & 0 & 0 & 0 \\0 & 0 & 4 & 0 & -24 & 0 \\0 & 0 & 0 & 0 & 0 & 0 \\0 & 0 & 0 & 0 & 0 & 0\\0 & 0 & 0 & 0 & 0 & 0\\ 0 & 0 & 2 & 0 & -4 & 0\end{array}\right],\,
&& N_6 = 
\left[\begin{array}{rrrrrr}0 & 0 & 0 & 0 & 0 & 0 \\0 & 0 & 0 & 0 & 0 & 0 \\8 & -4 & 0 & 0 & 0 & 8 \\-8 & 4 & 0 & 0 & 0 & -8\\4 & -2 & 0 & 0 & 0 & 4\\ 0 & 0 & 0 & 0 & 0 & 0\end{array}\right],
\end{align*}

$$
N_7=
\left[\begin{array}{rrrrrr}0 & 0 & 0 & 0 & 0 & 0 \\-2 & 2 & -2 & 0 & 4 & 0 \\0 & 0 & 0 & 0 & 0 & 0 \\-1 & 1 & -1 & 0 & 2 & 0\\1 & -1 & 1 & 0 & -2 & 0\\ -1 & 1 & -1 & 0 & 2 & 0\end{array}\right].
$$

\subsubsection{Nilpotent cones and types}\label{secnilcone}

For $\sigma_{ij} = \braket{N_i|N_{ij}|N_j}$, we denote the nilpotent cone generated by $N_i$ and $N_j$, and $\sigma_{ij} = \{sN_i + tN_j:s,t\in \bQ_{\geq0}\}$, and $N_{ij}=\{sN_i + tN_j:s,t\in\bQ_{+}\}$ is the interior of $\sigma_{ij}$. The notation $\braket{N_i|N_{ij}|N_j} \xlongequal{\mathrm{Type}} \braket{\mathrm{Type}_i|\mathrm{Type}_{ij}|\mathrm{Type}_j}$ denotes the type of the mixed Hodge structure corresponding to each nilpotent operator. For the definition of the types, we refer to \cite{kerr2017polarized} (especially, Example 1.18). The seven nilpotent cones in question are the following.

\begin{enumerate}
    \item $\sigma_{12} = \braket{N_1|N_{12}|N_2} \xlongequal{\mathrm{Type}}  
    \braket{IV_2|IV_2|II_1}$;
    \item $\sigma_{13} =\braket{N_1|N_{13}|N_3}\xlongequal{\mathrm{Type}} 
    \braket{IV_2|IV_2|I_1}$;
    \item $\sigma_{34} = \braket{N_3|N_{34}|N_4} \xlongequal{\mathrm{Type}}  
    \braket{I_1|I_2|I_1}$;
    \item $\sigma_{45} =\braket{N_4|N_{45}|N_5}\xlongequal{\mathrm{Type}} 
    \braket{I_1|II_1|II_0}$;
    \item $\sigma_{52} =\braket{N_5|N_{52}|N_2} \xlongequal{\mathrm{Type}} 
    \braket{II_0|II_1|II_1}$;
    \item $\sigma_{63} = \braket{N_6|N_{63}|N_3} \xlongequal{\mathrm{Type}} 
    \braket{I_1|I_2|I_1}$;
    \item $\sigma_{67} =\braket{N_6|N_{67}|N_7} \xlongequal{\mathrm{Type}}
    \braket{I_1|I_2|I_1}$.
\end{enumerate}

The type of the nilpotent operators is determined as follows. For a nilpotent operator $N$, we first calculate the Jordan normal form. Since the Hodge structure is of type $(1,2,2,1)$, the mixed Hodge type is uniquely determined when the Jordan form has one 2-dimensional block and four 1-dimensional blocks. This is the type $I_1$ mixed Hodge structure. The operators $N_3,N_4,N_6,N_7$ are of type $I_1$.

Next, if the Jordan form consists of a 4-dimensional block and a 2-dimensional block, then the mixed Hodge type is $IV_2$. Both $N_1,N_{12}$ are of type $IV_2$. This coincides with Lemma \ref{lemHT}, since $\sigma_{12}$ is the cone at the MUM point. The Hodge diamonds for type $I_1$ and $IV_2$ are as follows.

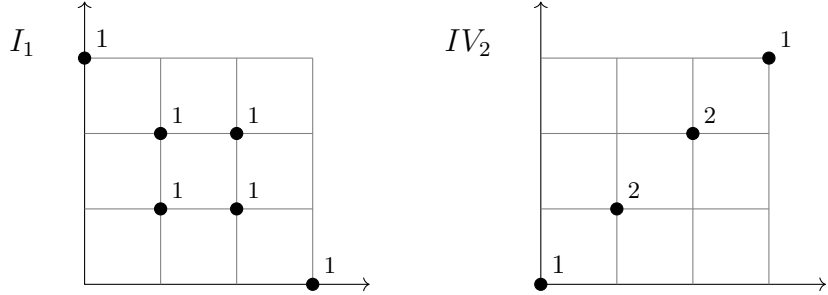
\begin{center}
\begin{tikzpicture}
 \node [left] at (-0.5,3.2) {$I_1$};
 \draw [<->] (0,3.75) -- (0,0) -- (3.75,0);
 \draw [gray] (1,0) -- (1,3);
 \draw [gray] (2,0) -- (2,3);
 \draw [gray] (3,0) -- (3,3);
 \draw [gray] (0,1) -- (3,1);
 \draw [gray] (0,2) -- (3,2);
 \draw [gray] (0,3) -- (3,3);
 \draw [fill] (0,3) circle [radius=0.08];
 \node [above right] at (0,3) {\footnotesize{$1$}};
 %
 \draw [fill] (1,2) circle [radius=0.08];
 \node [above right] at (1,2) {\scriptsize{$1$}};
 %
  \draw [fill] (2,2) circle [radius=0.08];
 \node [above right] at (2,2) {\scriptsize{$1$}};
  \draw [fill] (1,1) circle [radius=0.08];
 \node [above right] at (1,1) {\scriptsize{$1$}};
 \draw [fill] (2,1) circle [radius=0.08];
 \node [above right] at (2,1) {\scriptsize{$1$}};
 %
 \draw [fill] (3,0) circle [radius=0.08];
 \node [above right] at (3,0) {\scriptsize{$1$}};
 
 \node [left] at (5.5,3.2) {$IV_2$};
 \draw [<->] (6,3.75) -- (6,0) -- (9.75,0);
 \draw [gray] (7,0) -- (7,3);
 \draw [gray] (8,0) -- (8,3);
 \draw [gray] (9,0) -- (9,3);
 \draw [gray] (6,1) -- (9,1);
 \draw [gray] (6,2) -- (9,2);
 \draw [gray] (6,3) -- (9,3);
 \draw [fill] (6,0) circle [radius=0.08];
 \node [above right] at (6,0) {\footnotesize{$1$}};
 %
 \draw [fill] (7,1) circle [radius=0.08];
 \node [above right] at (7,1) {\scriptsize{$2$}};
 %
 \draw [fill] (8,2) circle [radius=0.08];
 \node [above right] at (8,2) {\scriptsize{$2$}};
 %
 \draw [fill] (9,3) circle [radius=0.08];
 \node [above right] at (9,3) {\scriptsize{$1$}};
\end{tikzpicture}
\end{center}

If the Jordan normal form has two 2-dimensional blocks and two 1-dimensional blocks, then the mixed Hodge structure that could associate to it is either $I_2$ or type $II_0$, with the following Hodge diamonds.

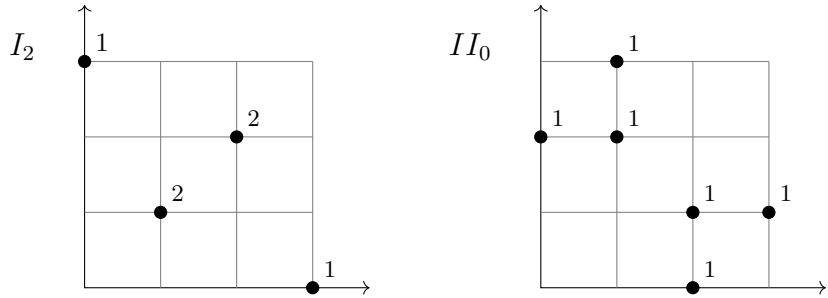
\begin{center}
\begin{tikzpicture}
 \node [left] at (-0.5,3.2) {$I_2$};
 \draw [<->] (0,3.75) -- (0,0) -- (3.75,0);
 \draw [gray] (1,0) -- (1,3);
 \draw [gray] (2,0) -- (2,3);
 \draw [gray] (3,0) -- (3,3);
 \draw [gray] (0,1) -- (3,1);
 \draw [gray] (0,2) -- (3,2);
 \draw [gray] (0,3) -- (3,3);
 \draw [fill] (0,3) circle [radius=0.08];
 \node [above right] at (0,3) {\footnotesize{$1$}};
  \draw [fill] (2,2) circle [radius=0.08];
 \node [above right] at (2,2) {\scriptsize{$2$}};
  \draw [fill] (1,1) circle [radius=0.08];
 \node [above right] at (1,1) {\scriptsize{$2$}};

 \draw [fill] (3,0) circle [radius=0.08];
 \node [above right] at (3,0) {\scriptsize{$1$}};

 \node [left] at (5.5,3.2) {$II_0$};
 \draw [<->] (6,3.75) -- (6,0) -- (9.75,0);
 \draw [gray] (7,0) -- (7,3);
 \draw [gray] (8,0) -- (8,3);
 \draw [gray] (9,0) -- (9,3);
 \draw [gray] (6,1) -- (9,1);
 \draw [gray] (6,2) -- (9,2);
 \draw [gray] (6,3) -- (9,3);
 \draw [fill] (6,2) circle [radius=0.08];
 \node [above right] at (6,2) {\scriptsize{$1$}};
 \draw [fill] (7,3) circle [radius=0.08];
 \node [above right] at (7,3) {\scriptsize{$1$}};
 \draw [fill] (7,2) circle [radius=0.08];
 \node [above right] at (7,2) {\scriptsize{$1$}};
 %
 \draw [fill] (8,1) circle [radius=0.08];
 \node [above right] at (8,1) {\scriptsize{$1$}};
 %
 \draw [fill] (8,0) circle [radius=0.08];
 \node [above right] at (8,0) {\scriptsize{$1$}};
 \draw [fill] (9,1) circle [radius=0.08];
 \node [above right] at (9,1) {\scriptsize{$1$}};
\end{tikzpicture}
\end{center}

In order to distinguish one from the other, we will make use of polarized relations.
These encode how the mixed Hodge type could degenerate. For example, if we have a nilpotent cone that has Hodge type $\braket{\mathrm{Type}_1|\mathrm{Type}_{12}|\mathrm{Type}_2}$, then one has polarized relations $\mathrm{Type}_1\prec \mathrm{Type}_{12}$ and $\mathrm{Type}_2\prec\mathrm{Type}_{12}$. 
Example 1.18 in \cite{kerr2017polarized} works out all the possible polarized relations for Hodge type $(1,2,2,1)$. From there we see the $II_0\prec II_1$ but $I_2\not\prec II_1$. Similarly, we have $I_1\prec I_2$ but $I_1\not\prec II_0$. These are enough to determine the mixed Hodge types for all the nilpotent cones in our list.


\section{Main result}\label{Secinint}

As we discussed in the roadmap, the candidate fan we need for the Kato-Usui's completion of the period map is the $\Gamma$-adjoint orbit of the union of the seven nilpotent cones listed in Section \ref{secnilcone}, where $\Gamma$ is the global monodromy group. However, this object is indeed a fan if and only if there are no interior intersections among the cones in the orbit. In this section, we will show there are no infinite interior intersections after we lift the period map to a finite cover. More precisely, we will show the following stronger statement.

\begin{lemma}\label{lem41}
After lifting the family to a finite cover, the $Sp(6,\bZ)$-orbit of the union of the (lift of the) 7 cones in Section \ref{secnilcone} has no interior intersection. 
\end{lemma}
This statement is stronger than we need since we are showing there is no interior intersection at all and, $\Gamma <Sp(6,\bZ)$. Then we arrive at the following crucial theorem.
\begin{thm4.2}
There is a finite cover $B'\xrightarrow{48:1}B$ such that the nilpotent fan $\Sigma'$ associated to the lifted period map $\Phi':B'\rightarrow \Gamma'\backslash D$ is strongly compatible with the monodromy group.
\end{thm4.2}
\begin{proof}
By the construction of our fan, it is naturally strongly $\Gamma$-compatible as long as it is indeed a fan. 
\end{proof}

Finally, as a corollary, we arrive at the main result of our paper.
\begin{theorem1.1}
There is a finite cover $B'\xrightarrow{}B$ with the property that the lifted period map $\Phi':B'\rightarrow \Gamma'\backslash D$ admits a Kato-Usui type completion. In particular, points at infinity parametrize nilpotent orbits. Moreover, at a generic point of $B$ (or $B'$) the real special Mumford-Tate group is $Sp(6,\bR)$.
\end{theorem1.1}
\begin{proof}
This is an immediate consequence of Theorem 4.2, Theorem \ref{KTtheorem} and Corollary \ref{coMT} after we lift the family to a cover such that all the monodromy operators are unipotent.  
\end{proof}

\subsection{Strategy of the proof of Lemma \ref{lem41}}\label{secstra}
The proof is done by direct calculation, following the strategy in \cite{deng2021extension}. First, we note, 
study the $\Sp(6,\bZ)$-orbit in an $n$-th lifting of the family is the same as study the $\Gamma_n$-orbit of the original family, where $\Gamma_n:=\{H\in Sp(6,\bZ): H\equiv \mathrm{Id},\mod n\}$.
Therefore, one only needs to find an appropriate number $n$, such that the $\Gamma_n$-adjoint orbits of nilpotent cones have no interior intersections. We will show that $n=12$ will be sufficient.

Next, recall the adjoint action by $Sp(6,\bZ)$ preserves the type of the LMHS induced by the nilpotent operators (Theorem 3.3 in \cite{cattani1982polarized}). Then by the mixed Hodge type we calculated in Section \ref{secnilcone}, there are 12 possible ways that the nilpotent orbits can intersect one another in interior points.

To be more precise, for any two cones $\sigma,\sigma'$  from the list in Section \ref{secnilcone}, and any $\gamma\in\Gamma_n$, we wish to show $\sigma'\cap \mathrm{Ad}_\gamma\sigma'\neq 0$ unless $\sigma'\cap \mathrm{Ad}_\gamma\sigma'$ is a face of both $\sigma'$ and $ \mathrm{Ad}_\gamma\sigma'$. Using the mixed Hodge type reduces the relevant number of cases to 9, which we will discuss separately.

\subsubsection*{Convention}

We will denote the spanning parameters of the two cones in question as $s_1,t_1$ and $s_2,t_2$, respectively. Moreover, if the two cones that are under consideration are identical, we will denote them as $\sigma_{ij} = \braket{N_i|N_{ij}|N_j}$ and $\sigma'_{ij} = \braket{N'_i|N'_{ij}|N'_j}$. Note in this case, the two nilpotent cones do intersect trivially, i.e., when two cones are identical. So when we say interior intersection, we are excluding this case.

All the calculations in this section are done with the help of Matlab\textsuperscript \textregistered.

\subsection{Type $IV_2$} There are two cones whose interior has this type, they are $\sigma_{12}$ and $\sigma_{13}$, thus there are 3 possible ways, their adjoint orbit can intersect.

\subsubsection{$\Ad(\Gamma)\cdot\sigma_{12}\cap \sigma_{12}'$}\label{sec1212}
By the alignment of the type of LMHS, without losing generality, we can assume the possible non-trivial interior intersections are in $\Ad(\Gamma)\cdot N_{12} \cap N'_{12}$, and at $s_{1} > 0, t_{1}>0, s_{2}> 0, t_{2}\geq0$ with $(s_1,t_1)\neq (s_2,t_2)$. Here we have slightly abused the notation to allow the parameter in $N_{12}'$ to take value 0, this will be a common practice for the rest of this section. Now, we will show $\gamma \sigma_{12}$ and $\sigma_{12'}$ have no non-trivial interior intersection, $\forall \gamma\in \Gamma_3\cap\Sp(6,\bZ)$.

The Jacobson-Morosov weight filtrations induced by $N_{12}$ and $N_{1}$ are both 
\begin{center}
    \begin{tikzcd}
  W_{-3}=W_{-2}\arrow[hookrightarrow]{r}\arrow[equal]{d} 
  & W_{-1}=W_{0}\arrow[hookrightarrow]{r}\arrow[equal]{d}
  & W_{1}=W_{2}\arrow[hookrightarrow]{r}\arrow[equal]{d} 
  & W_{3} \arrow[equal]{d} \\
  \braket{e_1}
  & \braket{e_1,e_2,e_3}
  & \braket{e_1,e_2,e_3,e_5,e_6}
  & V 
    \end{tikzcd}.
\end{center}

Now, we transform the basis $\{e_i\}_{i=1}^6$ symplectically to a basis $\{e_i'\}_{i=1}^6$ that is adapted to the associated graded module, that is
\begin{equation}\label{adpdef}
Gr^{N_{12}}V \simeq
\braket{e_1}\oplus\braket{e_2,e_3}\oplus\braket{e_4,e_5}\oplus\braket{e_6}.
\end{equation}
More precisely, in the new basis, we want $N_{12}\cdot e_i'\cap W_{k-3}= \emptyset$, if $e_i'\in W_k$. One such choice is
\[
\begin{split}
    &e_1' = e_1,\quad e_2' = e_2+e_3,\quad e_3' = -e_2,\quad e_4' = -\frac73 e_2-e_3+e_4,\\
    &e_5' = -e_1+2e_2+e_6,\quad e_6' = \frac43e_1+2e_2 -2e_3-e_5+e_6.
\end{split}
\]
Put it in the form of Hodge diamond, we get Figure \ref{HD1213}.

\begin{figure}
\begin{center}
\begin{tikzpicture}
 \node [left] at (-0.5,3.2) {$N_{12}(N_{13})$};
 \draw [<->] (0,3.75) -- (0,0) -- (3.75,0);
 \draw [gray] (1,0) -- (1,3);
 \draw [gray] (2,0) -- (2,3);
 \draw [gray] (3,0) -- (3,3);
 \draw [gray] (0,1) -- (3,1);
 \draw [gray] (0,2) -- (3,2);
 \draw [gray] (0,3) -- (3,3);
 \draw [fill] (0,0) circle [radius=0.08];
 \node [above right] at (0,0) {\footnotesize{$1$}};
 \node [below right] at (0,0) {\footnotesize{$e_1'$}};
 \draw [fill] (1,1) circle [radius=0.08];
 \node [above right] at (1,1) {\scriptsize{$2$}};
 \node [below right] at (1,1) {\footnotesize{$e_2',e_3'$}};
 \draw [fill] (2,2) circle [radius=0.08];
 \node [above right] at (2,2) {\scriptsize{$2$}};
 \node [below right] at (2,2) {\footnotesize{$e_5',e_6'$}};
 \draw [fill] (3,3) circle [radius=0.08];
 \node [above right] at (3,3) {\scriptsize{$1$}};
 \node [below right] at (3,3) {\scriptsize{$e_4'$}};
\end{tikzpicture}
\caption{The Hodge diamond for the case of $N_{12}$ and $N_{13}$.}
\label{HD1213}
\end{center}
\end{figure}
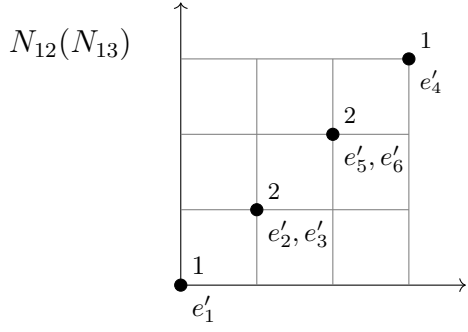
Denote the change of basis matrix as $T_B$, and 
\begin{equation}\label{1213G}
G = (T_B^T)^{-1} = \left[\begin{array}{rrrrrr} 1 & 0 & 0 & 0 & \frac73 & 1 \\0 & 0 & 1 & 1 & -2 & 0 \\0 & -1 & 1 & -\frac43 & -2 & 2 \\0 & 0 & 0 & 1 & 0 & 0\\0 & 0 & 0 & 0 & 1 & 1\\ 0 & 0 & 0 & 0 & -1 & 0\end{array}\right],
\end{equation}
then we have
\begin{equation}\label{1212N1}
    \Ad(G)\cdot N_{12} = \left[\begin{array}{rrrrrr} 0 & s_1+t_1 & -s_1 & 0 & 0 & 0  \\0 & 0 & 0 & 0 & 0 & 4s_1 \\0 & 0 & 0 & 0 & 4s_1 & -4t_1 \\0 & 0 & 0 & 0 & 0 & 0\\0 & 0 & 0 & -s_1-t_1 & 0 & 0\\ 0 & 0 & 0 & s_1 & 0 & 0\end{array}\right],
\end{equation}
and $\Ad(G)\cdot N_{12}'$ corresponding to the replacement $(s_1,t_1)\rightarrow (s_2,t_2)$.

Now, assume there exists an interior intersection, that is $ H\in\Sp(6,\bZ)$, $\Ad(H)\cdot N_{12} \cap N_{12}'\neq\emptyset$ which is equivalent to $\Ad(G HG^{-1})\cdot \Ad(G) \cdot N_{12}\cap \Ad(G)\cdot N_{12}'\neq\emptyset $. Since $\Ad(G)\cdot N_{12}$ and $\Ad(G)\cdot N_{12}'$ induce the same weight filtration, $G H G^{-1}$ must be parabolic an element of the parabolic group stabilizing the weight filtration. Moreover, the basis $\{e_i'\}_{i=1}^6$ is adapted to both $Gr^{\Ad(G)\cdot N_{12}}V$ and $Gr^{\Ad(G)\cdot N_{12}}V$ in the sense of \eqref{adpdef}. Consider a Levi decomposition of $G H G^{-1} = LU$, where $L$ and $U$ are the Levi part and unipotent part respectively. The unipotent part $U$ preserves the graded quotient, i.e., $\Ad(U)\cdot (N_{12}W_k)\cap W_{k-2}=(N_{12}W_k)\cap W_{k-2}$.  
Therefore, we can take $G H G^{-1}$ to be of the following form
\begin{equation}\label{levi1213}
G H G^{-1} = \left[\begin{array}{rrrrrr}k_1 & 0 & 0 & 0\\0 & A & 0 & 0  \\0 & 0 & k_2 & 0 \\ 0 & 0 & 0 & B\end{array}\right],\quad k_{1},k_2\in\bQ^\times, A,B\in\mathrm{GL}(2,\bQ),
\end{equation}
with $k_1k_2=1,A^TB=\mathrm{Id}$. Now we want to solve the equations \begin{equation}\label{Sequat}
S:=GHG^{-1} \cdot (\Ad(G) \cdot N_{12}) -  (\Ad(G) \cdot N_{12}')\cdot GHG^{-1}=0_{6\times 6}.  
\end{equation}
First, we reduce the degree of freedom in (\ref{levi1213}) by using the fact that $H\in\Sp(6,\bZ)$. Starting with a general element $H = (h_{ij})_{i,j = 1,\cdots,6}\in\mathrm{GL}(6,\bZ)$, we can reduce the degree of freedom of $H$ by using $GHG^{-1}$ is block diagonal in the form (\ref{levi1213}), which gives
\begin{equation}\label{ABform}
\begin{split}
&k_1 = h_{11},\quad A =  \left[\begin{array}{cc}h_{32}+h_{44} & -h_{32}\\ h_{32}-h_{23}-h_{22}+h_{44}& h_{22}-h_{32}\end{array}\right],\\
&k_2 = h_{44},\quad  B = \left[\begin{array}{cc}h_{56}+h_{66} & h_{56}-h_{55}-h_{65}+h_{66}\\ -h_{56}& h_{55}-h_{56}\end{array}\right],
\end{split}
\end{equation}
Since $h_{ij}\in\bZ$, we have $k_1 = k_2 = \pm1$. On the other hand, $\det(A),\det(B)\in\bZ$, gives $\det(A),\det(B)=\pm1$. Then $G H G^{-1}$ can be taken to be the following form
    
\begin{equation}\label{goodlevi}
    G H G^{-1} = \left[\begin{array}{rrrrrr}(-1)^n & 0 & 0 & 0 & 0 & 0 \\0 & a & b & 0 & 0 & 0 \\0 & c & d & 0 & 0 & 0 \\0 & 0 & 0 & (-1)^n & 0 & 0\\0 & 0 & 0 & 0 & (-1)^{1+\sdet(A)}d & (-1)^{\sdet(A)}c\\ 0 & 0 & 0 & 0 & (-1)^{\sdet(A)}b  & (-1)^{1+\sdet(A)}a\end{array}\right],
\end{equation}
where $\det(A) = ad-bc = \pm1 =:(-1)^{1+\sdet(A)}$ or $\sdet: = \frac{1+\det(A)}{2}$. Substitute these into (\ref{Sequat}), we arrive at the following equations
\[
\begin{split}
    &S_{12} = (-1)^n(s_1 + t_1) - a(s_2 + t_2) + cs_{2} = 0,\\
    &S_{13} = ds_2 - b(s_2 + t_2) + (-1)^{1+n}s_1 = 0,\\
    &S_{25} = 4b(s_1 + (-1)^{1+\sdet(A)}s_2) =0 ,\\
    &S_{26} = 4as_1 - 4bt_1 + (-1)^{\sdet(A)}4as_2=0,\\ 
    &S_{35} = 4ds_1 + (-1)^{\sdet(A)}4bt_2 + (-1)^{\sdet(A)}4ds_2=0,\\
    &S_{36} = 4cs_1 - 4dt_1 + (-1)^{1+\sdet(A)}4at_2 + (-1)^{1+\sdet(A)}4cs_2 = 0,\\
    &S_{54} = (-1)^n(s_2 + t_2) + (-1)^{\sdet(A)}d(s_1 + t_1) + (-1)^{\sdet(A)}cs_1 =0,\\
    &S_{64} = (-1)^{n+1}s_2 + (-1)^{1+\sdet(A)}b(s_1 + t_1) + (-1)^{1+\sdet(A)}as_1 =0.
\end{split}
\]
If $b=0$, then the above equation has solution $a=d = (-1)^n$, $c=0,t_1 = t_2,s_1 = s_2$, which corresponds to a trivial interior intersection. Now, if $b\neq 0$, by $s_{1}>0$ from $S_{25} = 0$, we have $s_1 = s_2$ and $\det(A) = -1$. Eliminating $s_{1},t_{1},t_{2}$ we arrive at 
\begin{equation}\label{gb12}
\begin{split}
    & 2ad - ab + bc + (-1)^n2a + (-1)^nb =0,\\
    & b - 3d + (-1)^n = 0,\\
    & 2ad + bc + bd + (-1)^nb +(-1)^{n+1} 2d =0 ,\\
    & 3a + b + (-1)^n=0.\\ 
\end{split}
\end{equation}
Simplifying  these equations gives $d = -a$, $b = -3a-(-1)^n$, and $c = \frac{-a^2-1}{3a+(-1)^n}$. Now recall from (\ref{ABform}) we have $b = -h_{32}$. Since $H\in \Gamma_3$, we have $h_{32}\equiv0\mod 3$ implying  $b\equiv 0\mod 3$, which gives the desired contradiction. Therefore, $\Ad(\Gamma_3)\cdot\sigma_{12}$ and $\sigma_{12}'$ has no non-trivial interior intersection.

Now for the case $b = c$ and not equal to 0, substitute $b=c$ back into (\ref{gb12}), we have
$$
P_3=0\Rightarrow 3a + b + (-1)^n=0,\quad P_4 = 0\Rightarrow b - 3d + (-1)^{n+\sdet(A)}=0,
$$
Then from $ad-bc = \det(A) = (-1)^{\sdet(A)}$, we have
$$
-10b^2 -((-1)^n+(-1)^{n+\det(A)})b+8(-1)^{\det(A)}=0.
$$
Since $b\in\bZ$, the discriminant is real only for $\det(A) = -1$. Then
$$
10b^2+8 = 0,
$$
which is impossible. Therefore, we can conclude $\Ad(\Gamma_3\cap\Sp(6,\bZ))\cdot \sigma_{12}$ and $\sigma_{12'}$ has no non-trivial interior intersection.

\subsubsection{$\Ad(\Gamma)\cdot\sigma_{12}\cap \sigma_{13}$}
In this case, the possible non-trivial interior intersections are at $s_{1} > 0, t_{1}\geq0, s_{2}> 0, t_{2}\geq0$ with $(s_1,t_1)\neq (s_2,t_2)$, and $t_1,t_2$ are not simultaneously $0$. We will show  $\Ad(\Gamma_2\cap\Sp(6,\bZ))\cdot \sigma_{12}$ and $\sigma_{12'}$ has no non-trivial interior intersection.

The LMHS with cone $N_{13}$ has Hodge diamond as the LMHS with cone $N_{12}$ (Figure \ref{HD1213}). Furthermore, by direct calculation, one can see the change of basis (\ref{1213G})  yields the following grading adapted matrix representation of the operator
\begin{equation}\label{1213N2}
\mathrm{Ad}(G)\cdot N_{13} = \left[\begin{array}{rrrrrr} 0 & s_2 & -s_2 & 0 & 0 & 0  \\0 & 0 & 0 & 0 & 4t_2 & 4s_2+4t_2 \\0 & 0 & 0 & 0 & 4s_2+4t_2 & 4t_2 \\0 & 0 & 0 & 0 & 0 & 0\\0 & 0 & 0 & -s_2 & 0 & 0\\ 0 & 0 & 0 & s_2 & 0 & 0\end{array}\right],
\end{equation}

Assume there exists an interior intersection. Then using the same argument and same notation as in Section \ref{sec1212}, the unipotent part of a fixed Levi decomposition of $GHG^{-1}$ fixes both (\ref{1212N1}) and (\ref{1213N2}). Therefore, we can take $GHG^{-1}$ to be the Levi part as in (\ref{goodlevi}), and considering the intersection in $\Ad(G HG^{-1})\cdot \Ad(G) \cdot N_{12}\cap \Ad(G)\cdot N_{13}$. Now, as before, we want to solve the equations $S:= GHG^{-1}\cdot (\Ad(G) \cdot N_{12}) -  (\Ad(G)\cdot N_{13})\cdot GHG^{-1}=0_{6\times 6}$. More precisely, we have
\[
\begin{split}
    &S_{12} = (-1)^n(s_1 + t_1) - as_2 + cs_2 = 0,\\
    &S_{13} = ds_2 - bs_2 + (-1)^{1+n}s_1 = 0,\\
    &S_{25} = 4bs_1 +(-1)^{1+\sdet(A)} 4b(s_2 + t_2) + (-1)^{\sdet(A)}4dt_2 =0 ,\\
    &S_{26} = 4as_1 - 4bt_1 + (-1)^{\sdet(A)}4a(s_2 + t_2) +  (-1)^{1+\sdet(A)} 4ct_2=0,\\ 
    &S_{35} = 4ds_1 + (-1)^{\sdet(A)}4d(s_2 + t_2) + (-1)^{1+\sdet(A)}4bt_2 = 0,\\ 
    &S_{36} = 4cs_1 - 4dt_1 + (-1)^{1+\sdet(A)} 4c(s_2 + t_2) + (-1)^{\sdet(A)} 4at_2 = 0,\\
    &S_{54} = (-1)^ns_2 + (-1)^{\sdet(A)}d(s_1 + t_1) + (-1)^{\sdet(A)}cs_1 =0,\\
    &S_{64} = (-1)^{n+1}s_2 + (-1)^{1+\sdet(A)}b(s_1 + t_1) + (-1)^{1+\sdet(A)}as_1 =0.
\end{split}
\]
If $b = 0$, $S_{25} = 0$ $\Rightarrow$ $dt_2 = 0$. Since $\det(A)\neq 0$, $a,d \neq 0$, we have $t_2 = 0$. Then $S_{26} = 4a(s_1+(-1)^{\sdet(A)}s_2) =0$, $s_1,s_2>0$ $\Rightarrow$ $s_1 = s_2$, $\det(A) = 1$. Then $S_{13} = S_{64} = 0$ $\Rightarrow$ $a = d = (-1)^n$. Then $S_{12} = S_{36} = 0$ $\Rightarrow$ $cs_1 =  t_1 = 0$. Since we assumed $t_1$ and $t_2$ are not simultaneously 0, we have reached a contradiction.

Now if $b\neq0$, eliminating  $s_{1},t_{1},t_2$ and using the fact $s_{2}>0$, we have
\begin{equation*}
\begin{split}
    P_1 := & ab^2 - ad^2 - 2bd^2 + b^2d + 2cd^2 + d^3 - 2bcd + (-1)^{n+\sdet(A)}ab \\
    & +(-1)^{1+n+\sdet(A)} cd =0,\\
    P_2 := & b^2 - 2bd + (-1)^{n+\sdet(A)}b + d^2 + (-1)^{n+\sdet(A)}d = 0,\\
    P_3 := & d - b + (-1)^{\sdet(A)}d^3 + ab^2 + (-1)^{\sdet(A)}b^2c +(-1)^{1+\sdet(A)} 2bd^2\\
    & + (-1)^{\sdet(A)}b^2d + (-1)^{\sdet(A)}2cd^2 + (-1)^{n}ab +(-1)^{1+n} cd \\
    & + (-1)^{\sdet(A)}abd +(-1)^{1+\sdet(A)} 3bcd =0 ,\\
    P_4 := & d^2 +(-1)^{1+\sdet(A)} 2b^2d^2 - bd + (-1)^{\sdet(A)}ab^3 + (-1)^{\sdet(A)}ad^3 \\
    & + (-1)^{\sdet(A)}b^3d +(-1)^{1+\sdet(A)} 2abd^2 + (-1)^{\sdet(A)}bcd^2 \\
    & +(-1)^{1+\sdet(A)} b^2cd + (-1)^nab^2 +(-1)^{1+n} bcd + (-1)^{\sdet(A)}bd^3  =0.\\ 
    \end{split}
\end{equation*}
Now, one way to simplify these equations is to calculate the Gr{\"o}ebner basis for the ideal $I=\braket{P_1,P_2,P_3,P_4} \subset\bC[a,b,c,d]$, with respect to some monomial order. For each of the four cases $(-1)^n = \pm1$ and $(-1)^{\sdet{A}} = \pm1$ we calculate the Gr{\"o}ebner basis in the graded reverse lexicographic order with the help of Matlab\textsuperscript \textregistered. Under the total order $a>c>d>b$, we have

\begin{enumerate}
    \item For $(-1)^{n+\sdet(A)} = 1$, the last generator in the Gr{\"o}ebner basis is $P^I_7 = b^2 - 2bd + b + d^2 + d$. The discriminant $\Delta_b$ for $b$ is $\Delta_b = 1-8d$, which gives $d = 0$, $b = -1$. Now recall from (\ref{ABform}), if $H\in \Gamma_2$, we have $b\equiv 0\mod 2$; in this circumstance, there is no solution to the system $\{P_j=0\}$.
    \item For $(-1)^{n+\sdet(A)} = -1$, the last generator in the Gr{\"o}ebner basis is $P^I_7 = b^2 - 2bd - b + d^2 - d$. The discriminant $\Delta_b$ for $b$ is $\Delta_b = 1+8d$. For $b\in\bZ$, we must have $1+8d = (1+2m)^2$, for some $m\in\bZ$. Then $d = \frac{m(m+1)}{2}$, and $b = \frac{m(m-1)}{2},\frac{(m+1)(m+2)}{2}$. Now, since $\det(A) = ad-bc = \pm1$, we have $\gcd(d,b) = 1$, therefore, we have $m =2,b=1,d=3$; or $m=-2,b=-3,d=-1$; or $m = 1,b = 3,d = 1$; or $m = -3,b=-1,d = -3$. Again, there is no solution with $b\equiv 0\mod 2$.
\end{enumerate}

Therefore, we arrive at $\Ad(\Gamma_2\cap\Sp(6,\bZ))\cdot \sigma_{12} \cap \sigma_{13} = \emptyset$.

\subsubsection{$\Ad(\Gamma)\cdot\sigma_{13}\cap \sigma_{13}'$} The calculation for this case is essentially the same as in Section \ref{sec1212}, so we will be brief. In this case, we can also assume the possible non-trivial interior intersections are in $\Ad(\Gamma)\cdot N_{13} \cap N'_{13}$, i.e., at $s_{1} > 0, t_{1}>0, s_{2}> 0, t_{2}\geq0$ with $(s_1,t_1)\neq (s_2,t_2)$. We will show $\Ad(\Sp(6,\bZ))\cdot \sigma_{13}$ and $\sigma_{13}'$ has no non-trivial interior intersection.

Now transform both $N_{13}$ and $N_{13}'$ by the adjoint action of matrix $G$ in (\ref{1213G}), we arrive at (\ref{1213N2}), with $(s_1,t_1)\rightarrow (s_2,t_2)$ for $N_{13}'$. Now assume there is a non-trivial interior intersection. By the same argument in Sec \ref{sec1212}, we can take $GHG^{-1}$ to be the Levi part as in (\ref{goodlevi}), and we want to solve the equations $S:= GHG^{-1}\cdot (\Ad(G) \cdot N_{13}) -  (\Ad(G)\cdot N_{13})\cdot GHG^{-1}=0_{6\times 6}$. We have
\begin{equation}\label{1313seq}
\begin{split}
    &S_{12} = (-1)^ns_1  - as_2 + cs_2 = 0,\\
    &S_{13} = ds_2 - bs_2 + (-1)^{1+n}s_1 = 0,\\
    &S_{25} = 4at_1 + 4b(s_1+t_1)  + (-1)^{1+\sdet(A)}4b(s_2 + t_2) +  (-1)^{\sdet(A)} 4dt_2=0,\\
    &S_{26} = 4bt_1 + 4a(s_1 + t_1) + (-1)^{\sdet(A)}4a(s_2 + t_2) +(-1)^{1+\sdet(A)} 4ct_2,\\ 
    &S_{35} = 4ct_1 + 4d(s_1+t_1)  + (-1)^{\sdet(A)}4d(s_2 + t_2) +  (-1)^{1+\sdet(A)} 4bt_2= 0,\\ 
    &S_{36} = 4dt_1 + 4c(s_1 + t_1) + (-1)^{1+\sdet(A)}4c(s_2 + t_2) +(-1)^{\sdet(A)} 4at_2 = 0,\\
    &S_{54} = (-1)^ns_2 + (-1)^{\sdet(A)}cs_1 + (-1)^{\sdet(A)}ds_1 =0,\\
    &S_{64} = (-1)^{n+1}s_2 + (-1)^{1+\sdet(A)}as_1 + (-1)^{1+\sdet(A)}bs_1 =0.
\end{split}
\end{equation}
This can be solved as follows, first we substitute $s_1$ and $t_2$ by using $S_{13}=S_{26}=0$, the expression for $s_1$ is 
\begin{equation}\label{expressionfors1}
s_1  = (-1)^n(d-b)s_2.
\end{equation}
Then we get $S_{12} = -s_2(a + b - c - d)=0$, this gives $c = a+b-d$, since $s_2\neq 0$. Substitute it  back, and using $s_2\neq 0$ again, we have
\begin{equation*}
\begin{split}
    P_1 := & ad - ab - bd + b^2 + (-1)^{n+\sdet(A)}a +  (-1)^{n+\sdet(A)}b =0,\\
    P_2 := & (a - d)(d - b + (-1)^{n+\sdet(A)}) = 0,\\
    P_3 := & b^2 - 2bd + (-1)^{n+\sdet(A)}b + d^2 + (-1)^{1+n+\sdet(A)}d + (-1)^{n+\sdet(A)}2a =0 ,\\
    P_4 := & (-1)^{\sdet(A)} - b^2 - ab + ad + bd  =0,\\ 
    P_5 := &  a(b - d) + b(b - d) + (-1)^{1+\sdet(A)}=0.
    \end{split}
\end{equation*}
From $P_2 = 0$, we have two possibilities
\begin{enumerate}
    \item If $a=d$, then $b = c$. Since $\det(A) = ad-bd = \pm1$, we have $a^2-b^2=\pm1$. This is only possible for $(a,b) = (\pm1,0), (0,\pm1)$, implying $ d-b = a-b =  \pm1$.  By (\ref{expressionfors1}), and $s_1,s_2\geq 0$, we have $s_1 = s_2$. Substitute these back into (\ref{1313seq}), for all four possibilities of $(a,b)$, we get $S_{25}=4(t_2-t_1)=0$. Therefore, we have only trivial interior intersections.
        \item If $b = d+(-1)^{n+\sdet(A)}$, then substituting this back into $P_1=P_3=0$,  we get $a = d = (-1)^{1+n+\sdet(A)}$. This gives $b = 0$. By $c = a+b-d$, this in turn gives $c =0$. Again, we have only trivial interior intersections.
\end{enumerate}
Therefore, we arrive at $\Ad(\Sp(6,\bZ))$ and $\sigma_{12} \cap \sigma_{13}$ has no non-trivial interior intersections. This finishes the proof that there are no interior intersections between the adjoint orbit of type $IV_2$.

\subsection{Type $II_1$} There are two cones with the interior of this type: they are $\sigma_{45}$ and $\sigma_{52}$. Thus there are three intersections to consider.

\subsubsection{$\Ad(\Gamma)\cdot \sigma_{52} \cap \sigma'_{52}$}\label{subsec5252prime}

This case also includes the case $\Ad(\Gamma)\cdot\sigma_{12}\cap \sigma_{52}$, since the right boundaries of $\sigma_{12}$ and $\sigma_{52}$ are the same.

Without losing generality, the possible non-trivial interior intersections  in $\Ad(\Gamma)\cdot N_{52} \cap N'_{52}$ are at points with parameters $s_{1} \geq 0, t_{1}>0, s_{2}> 0, t_{2}>0$, and $[s_1:t_1]\neq [s_2:t_2]$. We will show, $|\{H\in \Sp(6,\bZ):\Ad(\Sp(6,\bZ))\cdot N_{52} \cap N_{52}|\}|=\emptyset$.

First, we symplectically transform the basis, so that the matrix representation of $N_{52}$ becomes the standard form $N_{st}^{II_1}$ of type $II_1$
\begin{equation}\label{II_1stand}
N_{st}^{II_1}:= \Ad(G)\cdot N_{52}  = 
\left[\begin{array}{rrrrrr}0 & 0 & 0 & 1 & 0 & 0 \\0 & 0 & 0 & 0 & 1 & 0 \\0 & 0 & 0 & 0 & 0 & 1 \\0 & 0 & 0 & 0 & 0 & 0\\0 & 0 & 0 & 0 & 0 & 0\\ 0 & 0 & 0 & 0 & 0 & 0\end{array}\right],
\end{equation}
with
$$
G = \left[\begin{array}{rrrrrr}\frac{1}{\sqrt{2t_1}} & 0 & 0 & 0 & 0 & 0 \\
\frac{i}{2u} & \frac{\sqrt{2}i}{4u} & 0 & 0 & 0 & \frac{i-\frac{i}{\sqrt{2}}}{k} \\
-\frac{i}{2u} & \frac{\sqrt{2}i}{4u} & 0 & 0 & 0 & \frac{-i-\frac{i}{\sqrt{2}}}{u} \\
0 & 0 & \sqrt{\frac{t_1}{2}} & \sqrt{2t_1} & -\sqrt{2t_1} & 0\\
0 & 0 & \frac{iu}{2} & 0 & -(1+\sqrt{2})iu & 0\\
0 & 0 & -\frac{iu}{2} & 0 & (1-\sqrt{2})iu & 0\end{array}\right],
$$
where $u = (4s_1+t_1)^{1/2}$, and $G\in \Sp(6,\bQ[i,\sqrt{2},\sqrt{t_1},u])$. Now, assume there exists an interior intersection. Notice, $H\in\Sp(6,\bZ)$ and $\Ad(H)\cdot N_{52} \cap N_{52}'\neq\emptyset$ if and only if$\Ad(G HG^{-1})\cdot N^{II_1}_{st}\cap \Ad(G)\cdot N_{52}'\neq\emptyset $. Since $N_{52}$ and $N_{52}'$ induce the same Jacobson-Morosov filtration, $N_{st}^{II_1}$ and $\Ad(N_{52}')$ also induce the same Jacobson-Morosov filtration, which means, $G H G^{-1}$ has to be parabolic of the form $\left(\begin{array}{cc}* & * \\0 & *\end{array}\right)$, where each $*$ is a $3\times 3$ matrix. 

Now, if $G H G^{-1}$ stabilizes $N_{st}^{II_1}$, then we have $N^{II_1}_{st}= \Ad(G)\cdot N_{52}'$, since
\begin{equation}\label{52N52prime}
\Ad(G)\cdot N_{52}' = 
\left[\begin{array}{rrrrrr}0 & 0 & 0 & \frac{t_2}{t_1} & 0 & 0 \\0 & 0 & 0 & 0 & \frac{4s_1+t_1}{4s_2+t_2} & 0 \\0 & 0 & 0 & 0 & 0 & \frac{4s_1+t_1}{4s_2+t_2} \\0 & 0 & 0 & 0 & 0 & 0\\0 & 0 & 0 & 0 & 0 & 0\\ 0 & 0 & 0 & 0 & 0 & 0\end{array}\right].
\end{equation}
This means $(s_1,t_1) = (s_2,t_2)$, which corresponds to trivial intersections. Therefore, in order to find a potential non-trivial interior intersection, we would like to eliminate the degrees of freedom in $G H G^{-1}$ that stabilize $N_{st}^{II_1}$ in \eqref{II_1stand}.

To do so, first consider the Levi decomposition of $G H G^{-1}$. The unipotent part $\left[\begin{array}{cc} \mathrm{Id} & M\\0 & \mathrm{Id}\end{array}\right]$ stabilizes the standard form $N_{st}^{II_1}$
$$
\left[\begin{array}{cc} \mathrm{Id} & M\\0 & \mathrm{Id}\end{array}\right]\left[\begin{array}{cc} 0 & \mathrm{Id}\\0 & 0\end{array}\right]\left[\begin{array}{cc} \mathrm{Id} & -M\\0 & \mathrm{Id}\end{array}\right] =  \left[\begin{array}{cc} 0 & \mathrm{Id}\\0 & 0\end{array}\right].
$$
Thus we may assume
\begin{equation}\label{52simpcond1}
G H G^{-1} = \left[\begin{array}{cc} A & 0\\0 & (A^{T})^{-1}\end{array}\right],\quad A\in \mathrm{GL}(3,\bQ[i,\sqrt{2},\sqrt{t_1},u])
\end{equation}
is block diagonal. The lower right block is determined by the symplecticity. Further, since $$\Ad(G H G^{-1})\cdot N_{st}^{II_1} = \left[\begin{array}{cc} 0 & AA^{T}\\0 & 0\end{array}\right],$$ $A\in O(3,\bQ[i,\sqrt{2},\sqrt{t_1},u])$, the QR decomposition (a rudimentary case of Cartan decomposition) allows us to restrict to the case that $A$ to be upper triangular
\begin{equation}\label{52simpcond2}
A =  \left[\begin{array}{ccc} a_{11} & a_{12} & a_{13}\\0 & a_{22} & a_{23}\\ 0&0&a_{33}\end{array}\right], \quad a_{ij}\in \bQ[i,\sqrt{2},\sqrt{t_1},u].
\end{equation}
Aligning (\ref{52simpcond2}) and (\ref{52N52prime}), we have
\begin{equation}\label{52simpcond3}
\left[\begin{array}{ccc} a_{11}^2+a_{12}^2+a_{13}^2 & a_{12}a_{22}+a_{13}a_{23} & a_{13}a_{33}\\a_{12}a_{22}+a_{13}a_{23} & a_{22}^2+a_{23}^2 & a_{23}a_{33}\\  a_{13}a_{33} &  a_{23}a_{33} &a_{33}^2\end{array}\right]= \left[\begin{array}{rrr} \frac{t_2}{t_1} & 0 & 0 \\  0 & \frac{4s_1+t_1}{4s_2+t_2} & 0 \\ 0 & 0 & \frac{4s_1+t_1}{4s_2+t_2} \end{array}\right].
\end{equation}
Then $\det(A)\neq 0$ implpies $a_{22},a_{33}\neq 0$ and then $a_{13}=a_{23}=a_{12} = 0$, which in turn gives $a_{22} = a_{33}$.
Therefore, up to multiplication by a diagonal sign matrix of the form $\mathrm{diag}((-1)^{k_1},(-1)^{k_2},(-1)^{k_3},(-1)^{k_1},(-1)^{k_2},(-1)^{k_3}), (k_1,k_2,k_3)\in \bZ^3$, we have
\begin{equation}\label{52GHGform}
G H G^{-1} = \left[\begin{array}{rrrrrr} \sqrt{\frac{t_2}{t_1}} & 0 & 0 & 0 & 0 & 0 \\0 & \sqrt{\frac{4s_1+t_1}{4s_2+t_2}} & 0 & 0 & 0 & 0 \\0 & 0 & \sqrt{\frac{4s_1+t_1}{4s_2+t_2}} & 0 & 0 & 0 \\0 & 0 & 0 & \sqrt{\frac{t_1}{t_2}} & 0 & 0\\0 & 0 & 0 & 0 & \sqrt{\frac{4s_2+t_2}{4s_1+t_1}} & 0\\ 0 & 0 & 0 & 0 & 0 & \sqrt{\frac{4s_2+t_2}{4s_1+t_1}}\end{array}\right].
\end{equation}

Now we wish to solve \eqref{52GHGform} for  $H=(h_{ij})_{i,j=1,\cdots,6}\in \mathrm{GL}(6,\bZ)$. The 30 equations $(G H G^{-1})_{i,j}=0,i\neq j$ yields
$$
GHG^{-1} =  \mathrm{diag}
\left[\begin{array}{c}6h_{62}-h_{21}+h_{66} \\ 2h_{62}+h_{66}+2\sqrt{2}h_{62} \\ 2h_{62}+h_{66}-2\sqrt{2}h_{62}\\h_{33}+2h_{43}-2h_{53}\\ h_{33}-2h_{53}-2\sqrt{2}h_{53} \\  h_{33}-2h_{53}+2\sqrt{2}h_{53}\end{array}\right]:= \left[\begin{array}{cc} U & 0\\0 & L\end{array}\right].
$$
Now, by (\ref{52simpcond1}), we have $U^{T}\cdot L - \mathrm{Id}_{3\times 3} = 0$. The diagonal gives the following equations
\begin{equation}\label{52const1}
\begin{aligned}
    &(h_{33} + 2h_{43} - 2h_{53})\cdot(2h_{62} - 2h_{61} + h_{66}) - 1=0,\\
      &(h_{33} - 2h_{53} - 2\sqrt{2}h_{53} )\cdot(2h_{62} + h_{66} + 2\sqrt{2}h_{62} ) - 1=0,\\
     &(h_{33} - 2h_{53} - 2\sqrt{2}h_{53} )\cdot(2h_{62} + h_{66} + 2\sqrt{2}h_{62} ) - 1=0,
\end{aligned}
\end{equation}
where the second and the third one are identical as expected from (\ref{52GHGform}). Since $h_{ij}\in\bZ$, the first equation in (\ref{52const1}) tells us $a_{11} = 2h_{62} - 2h_{61} + h_{66} = \pm1$. Combined with (\ref{52simpcond3}), this gives $t_1 = t_2$. Moreover, $\frac{4s_1+t_1}{4s_2+t_2}\in\bQ$ implies $ a_{22}^2\in\bQ$ then $ (2h_{62}+h_{66}+2\sqrt{2}h_{62})^2\in\bQ \Rightarrow h_{62} = 0$. Now consider the second equation in (\ref{52const1}), we have $h_{62} = 0$ implies $ h_{53} = 0$ then $ a_{22} = 2h_{62} + h_{66} + 2\sqrt{2}h_{62}\in\bZ$ then $  a_{22} = 2h_{62} + h_{66} + 2\sqrt{2}h_{62} =\pm1 $ then $ \frac{4s_1+t_1}{4s_2+t_2} =1$, which further implies $ s_1 = s_2$. Therefore, we can conclude $\Ad(\Sp(6,\bZ))\cdot\sigma_{52} \cap \sigma'_{52}$ contains no nontrivial interior intersection. In other words, we have $|\{H\in \Sp(6,\bZ):\Ad(\Sp(6,\bZ))\cdot N_{52} \cap N_{52}|\}|=\emptyset $.

\subsubsection{$\Ad(\Gamma)\cdot \sigma_{45} \cap \sigma'_{45}$}

The possible non-trivial interior intersections are in $\Ad(\Gamma)\cdot N_{45} \cap N'_{45}$, i.e., at $s_{1} > 0, t_{1}>0, s_{2}> 0, t_{2}>0$ with $(s_1,t_1)\neq (s_2,t_2)$. We will show a stronger result, $|H\in \Sp(6,\bZ):\Ad(\Sp(6,\bZ))\cdot N_{45} \cap N'_{54}|\}|<\infty$. The procedure will be very similar to what we did in Section \ref{subsec5252prime}, so we will skip some details.

First, we symplectically transform $N_{45}$ into the standard form of type $II_1$ by the adjoint action of a matrix $G\in\bQ[i,\sqrt{2},\sqrt{s_1},\sqrt{t_1}]$
\begin{equation}\label{45trans}
G = \left[\begin{array}{rrrrrr}0 & 0 & 0 & \frac{1}{\sqrt{s_1}} & 0 & 0 \\0 & \frac{-i}{4\sqrt{t_1}} & 0 & 0 & 0 & \frac{i}{2\sqrt{t_1}} \\0 & 0 & 0 & 0 & 0 & \frac{-i}{\sqrt{2t_1}} \\-\sqrt{s_1} & 0 & 0 & -2\sqrt{s_1} & 0 & 0\\0 & 0 & 0 & 0 & 4i\sqrt{t_1} & 0\\ 0 & 0 & -i\sqrt{2t_1} & 0 & 2i\sqrt{2t_1} & 0\end{array}\right],\quad
\end{equation}

Beside, $\Ad(G)\cdot N_{45} = N_{st}^{II_1}$, we also have
$$
\Ad(G)\cdot N_{45}' = \left[\begin{array}{rrrrrr}0 & 0 & 0 & \frac{s_2}{s_1} & 0 & 0 \\0 & 0 & 0 & 0 & \frac{t_2}{t_1} & 0 \\0 & 0 & 0 & 0 & 0 & \frac{t_2}{t_1} \\0 & 0 & 0 & 0 & 0 & 0\\0 & 0 & 0 & 0 & 0 & 0\\ 0 & 0 & 0 & 0 & 0 & 0\end{array}\right],
$$
which is diagonal. Now assume $\Ad(\Sp(6,\bZ))\cdot N_{45}$ do intersect $N_{45}.'$ non-trivially, which means $\exists H\neq \mathrm{Id}$, $H\in\Sp(6,\bZ)$, such that, $\Ad(GHG^{-1})\cdot N_{st}^{II_1} = \Ad(G)\cdot N_{45}$, has solution at $(s_1,t_1)\neq (s_2,t_2)$. By the essentially the same argument as in Section \ref{subsec5252prime}, up to signs, $GHG^{-1}$ is diagonal of the form
\begin{equation}\label{45GHGform}
GHG^{-1} =\left[\begin{array}{rrrrrr} \sqrt{\frac{s_2}{s_1}} & 0 & 0 & 0 & 0 & 0 \\0 & \sqrt{\frac{t_2}{t_1}} & 0 & 0 & 0 & 0 \\0 & 0 & \sqrt{\frac{t_2}{t_1}} & 0 & 0 & 0 \\0 & 0 & 0 & \sqrt{\frac{s_1}{s_2}} & 0 & 0\\0 & 0 & 0 & 0 & \sqrt{\frac{t_1}{t_2}} & 0\\ 0 & 0 & 0 & 0 & 0 & \sqrt{\frac{t_1}{t_2}}\end{array}\right].
\end{equation}

Now, starting with a general element $H = (h_{ij})_{i,j = 1,\cdots,6}\in\mathrm{GL}(6,\bZ)$, and reduce the degree of freedom of $H$ by using $GHG^{-1}$ is diagonal, we arrive at
$$
GHG^{-1} = \left[\begin{array}{rrrrrr}h_{44} & 0 & 0 & 0 & 0 & 0 \\0 & h_{22} & 0 & 0 & 0 & 0 \\0 & 0 & h_{66} & 0 & 0 & 0 \\0 & 0 & 0 & h_{11} & 0 & 0\\0 & 0 & 0 & 0 & h_{55} & 0\\ 0 & 0 & 0 & 0 & 0 & h_{33}\end{array}\right].
$$
Combined with (\ref{45GHGform}), we see $h_{11}h_{44} = 1$, $h_{22}h_{55} = 1$, $h_{33}h_{66}=1$, since $h_{ij}\in\bZ$, we have $h_{ii} = \pm 1, i=1,\cdots,6$ $\Rightarrow (s_1,t_1)  = (s_2,t_2)$, which gives the desired contradiction. Therefore, we can conclude $\Ad(\Sp(6,\bZ))\cdot\sigma_{52} \cap \sigma'_{52}$ contains no nontrivial interior intersection, in other words, we have $|\{H\in \Sp(6,\bZ):\Ad(\Sp(6,\bZ))\cdot N_{52} \cap N_{52}|\}|<\infty$.

\subsubsection{$\Ad(\Gamma)\cdot \sigma_{45} \cap \sigma_{52}$}
The discussion here also includes the case  $\Ad(\Gamma)\cdot\sigma_{45}\cap \sigma_{12}$, where the intersection is at the boundary of $\sigma_{12}$.

By the alignment of the type of LMHS, the possible interior intersections are at $s_{1} > 0, t_{1} > 0, s_{2} \geq 0, t_{2} > 0$. We will show $\Ad(\Sp(6,\bZ))\cdot N_{45} \cap N_{52} = \emptyset$.

Assume not, then there exists $ G\in \Sp(6,\bZ)$ and $s_{1} \geq 0, t_{1} > 0, s_{2} > 0, t_{2} > 0$, such that $\Ad(G)\cdot (s_1N_5+t_1N_2) = s_2N_4+t_2N_5$. Especially, the operators on both sides of the equality should induce the same Jacobson-Morosov filtration. Now, the Jacobson-Morosov filtration for $N_{52}$ and $N_{45}$ are given as follows
\[
\begin{split}
    &W^{N_{52}}: 0\xhookrightarrow{}\braket{e_1,e_2,e_6}\xhookrightarrow{} V,\\
    &W^{N_{45}}:0\xhookrightarrow{}\braket{2e_1-e_4,e_4,e_6}\xhookrightarrow{}V.
\end{split}
\]
To proceed, we first symplectically transform the basis $\{e_i\}_{i=1}^6$ by $G_1\in\Sp(6,\bZ)$, such that the induced action on filtration satisfies $G_1\cdot W^{N_{52}} = W^{N_{45}}$. We take
$$
G_1 = \left[\begin{array}{rrrrrr}0 & 2 & 0 & 0 & 1 & 0 \\1 & 0 & 0 & 0 & 0 & 0 \\0 & 0 & 1 & 0 & 0 & 0 \\0 & -1 & 0 & 0 & 0 & 0\\0 & 0 & 0 & 1 & 0 & 0\\ 0 & 0 & 0 & 0 & 0 & 1\end{array}\right],
$$
Next, we do another symplectic transformation by (\ref{45trans}) for the new basis, such that the operator $N_{45}$ has the matrix representation as the standard form $N^{II_1}_{st}$ (\ref{II_1stand}). In this section, we denote it as $G_2$. Then we have
$$
\Ad(G_2G_1)\cdot N_{52} =  
\left[\begin{array}{rrrrrr}0 & 0 & 0 & -\frac{24s_2+4t_2}{s_1} & i\frac{4s_2-t_2}{2\sqrt{s_1t_1}} & \frac{-i2\sqrt{2}s_2}{\sqrt{s_1t_1}} \\0 & 0 & 0 & i\frac{4s_2-t_2}{2\sqrt{s_1t_1}} & \frac{4s_2-3t_2}{8t_1} & -\frac{\sqrt{2}(4s_2-t_2)}{8t_1} \\0 & 0 & 0 & \frac{-i2\sqrt{2}s_2}{\sqrt{s_1t_1}} & -\frac{\sqrt{2}(4s_2-t_2)}{8t_1} & \frac{s_2}{t_1} \\0 & 0 & 0 & 0 & 0 & 0\\0 & 0 & 0 & 0 & 0 & 0\\ 0 & 0 & 0 & 0 & 0 & 0\end{array}\right].
$$
To show $\Ad(\Sp(6,\bZ))\cdot N_{45} \cap N_{52} = \emptyset$,  is equivalent to show  $\Ad(G_2HG_2^{-1})\cdot N^{II_1}_{st} \neq \Ad(G_2G_1)\cdot N_{52}$, forall $ H\in\Sp(6,\bZ)$. Then by the same argument as in Section \ref{subsec5252prime}, we can take $G_2HG_2^{-1}$ to be of the form (\ref{52simpcond1}) and with block $A\in\mathrm{GL}(3,\bQ[i,\sqrt{2},\sqrt{s_1},\sqrt{t_1}])$ in the form (\ref{52simpcond2}). Moreover, we have
\begin{equation}\label{4552simpcond3}
\begin{split}
AA^{T} = & \left[\begin{array}{ccc} a_{11}^2+a_{12}^2+a_{13}^2 & a_{12}a_{22}+a_{13}a_{23} & a_{13}a_{33}\\a_{12}a_{22}+a_{13}a_{23} & a_{22}^2+a_{23}^2 & a_{23}a_{33}\\  a_{13}a_{33} &  a_{23}a_{33} &a_{33}^2\end{array}\right]\\
= & \left[\begin{array}{rrr} -\frac{24s_2+4t_2}{s_1} & i\frac{4s_2-t_2}{2\sqrt{s_1t_1}} & \frac{-i2\sqrt{2}s_2}{\sqrt{s_1t_1}} \\  i\frac{4s_2-t_2}{2\sqrt{s_1t_1}} & \frac{4s_2-3t_2}{8t_1} & -\frac{\sqrt{2}(4s_2-t_2)}{8t_1} \\ \frac{-i2\sqrt{2}s_2}{\sqrt{s_1t_1}} & -\frac{\sqrt{2}(4s_2-t_2)}{8t_1} & \frac{s_2}{t_1} \end{array}\right].
\end{split}
\end{equation}

Now, we want to reduce the degrees of freedom in $H = (h_{ij})_{i,j = 1,\cdots,6}\in\mathrm{GL}(6,\bZ)$, by using the fact that $G_2HG_2^{-1}$ is of the form (\ref{52simpcond1}). More precisely, we use the following equations to reduce the number of free variables in $H$
$$
S = G_2H - \left[\begin{array}{cc} A & 0\\0 & (A^{T})^{-1}\end{array}\right] = 0_{6\times 6}, \quad A = \left[\begin{array}{ccc} a_{11} & a_{12} & a_{13}\\0 & a_{22} & a_{23}\\ 0&0&a_{33}\end{array}\right].
$$
This process is apparently not unique, and one solution is
$$
H = \left[\begin{array}{rrrrrr}h_{11} & -\frac{2h_{22}h_{51}}{h_{11}} & 0 & \frac{2h_{11}^2-2}{h_{11}} & 0 & -2h_{46} \\0 & h_{22} & 0 & 0 & 0 & h_{26} \\h_{33}\alpha & 0 & h_{33} & 2h_{33}\alpha & \frac{h_{26}h_{33}}{h_{22}} & 0 \\0 & \frac{h_{22}h_{51}}{h_{11}} & 0 & \frac{1}{h_{11}} & 0 & h_{46}\\h_{51} & 0 & 0 & 2h_{51} & \frac{1}{h_{22}} & 0\\ 0 & 0 & 0 & 0 & 0 & \frac{1}{h_{33}}\end{array}\right],
$$
with  $\alpha = h_{26}h_{51}-h_{11}h_{46}$ and
$$
A = \left[\begin{array}{ccc} \frac{1}{h_{11}} & \frac{4ih_{22}h_{51}\sqrt{t_1}}{h_{11}\sqrt{s_1}} & \frac{i\sqrt{2t_1}(h_{11}h_{46}+2h_{22}h_{51})}{h_{11}\sqrt{s_1}}\\0 & h_{22} & \frac{\sqrt{2}(2h_{22}h_{33}+h_{26}h_{33}-2)}{4h_{33}}\\ 0&0&\frac{1}{h_{33}}\end{array}\right]. 
$$
We note the equations $S=0_{6\times6}$ combined with the (semi-)positivity of the parameters forces $h_{ii}\neq 0, i=1,2,3$, so dividing by them is valid. Moreover, since $H\in\Sp(6,\bZ)$, we have $h_{ii} = (-1)^{k_i}$, $i=1,2,3$, $k_i\in\bZ$. Then substituting these back into (\ref{4552simpcond3}), the $(2,2),(3,3)$ entries give
\begin{equation}\label{4552fineq}
\begin{split}
    P_1 := &(2(-1)^{k_2+k_3} - 2 + (-1)^{k_3}h_{26})^2 + 8 - \frac{4s_2-3t_2}{t_1} = 0,\\
    P_2 := & 1 -\frac{s_2}{t_1} = 0.\\
    \end{split}
\end{equation}
From $P_2 = 0$, we have $s_2 = t_1$. Then substitute back, we have $P_1 = (2(-1)^{k_2+k_3} - 2 + (-1)^{k_3}h_{26})^2 +  \frac{4t_1+3t_2}{t_1}=0$. Since $t_1,t_2>0$, this can never happen. 
Therefore, we can conclude $\Ad(\Sp(6,\bZ))\cdot\sigma_{52} \cap \sigma_{45}$ contains no nontrivial interior intersection. This completes the proof that there are no interior intersections between the adjoint orbit of type $II_1$.

\subsection{Type $I_2$}
There are 3 cones, $\sigma_{34},\sigma_{63},\sigma_{67}$, whose interior has type $I_2$ type. The Jacobson-Morosov filtration for $N_{34}$, $N_{63}$ and $N_{67}$ are given as follows
\[
\begin{split}
    W^{N_{34}}: 0 & \xhookrightarrow{}\braket{2e_1-e_4,e_3}\xhookrightarrow{} \braket{2e_1-e_4,e_2,e_3,e_5}\xhookrightarrow{} V,\\
    W^{N_{63}}:0 & \xhookrightarrow{}\braket{2e_4-e_5,e_3}\xhookrightarrow{}\braket{e_1+2e_2,e_3,e_4,e_5}\xhookrightarrow{} V,\\
    W^{N_{67}}:0 & \xhookrightarrow{}\braket{2e_2+e_4-e_5+e_6,2e_3-2e_4+e_5}\\
    & \xhookrightarrow{}\braket{2e_1+e_5-2e_6,2e_2-e_5+e_6,2e_3+e_5,e_4}\xhookrightarrow{} V.
\end{split}
\]

A priori, we have 6 cases need to discuss. 
However, for each of the three operators  $N_{34}$, $N_{63}$ and $N_{67}$, one can implement an integral symplectic transformation of the basis $\{e_i\}_{i=1}^6$ such that in the new basis the corresponding operator has matrix representation $N:=\{sE_{14}+tE_{25}:s,t>0\},$
where $E_{ij}$ is the matrix has 1 at $(i,j)$ entry and 0 elsewhere. One can check these three integral symplectic transformation can be induced by the following matrices
$$
G_{34} = \left[\begin{array}{rrrrrr}1 & 0 & 0 & 1 & 0 & 0 \\0 & 0 & 1 & 0 & 0 & 0 \\0 & 1 & 0 & 0 & 0 & 0 \\1 & 0 & 0 & 2 & 0 & 0\\0 & 0 & 0 & 0 & 0 & 1\\ 0 & 0 & 0 & 0 & 1 & 0\end{array}\right],\quad G_{63} = \left[\begin{array}{rrrrrr}0 & 0 & 1 & 0 & -2 & 0 \\0 & 0 & 0 & 0 & 1 & 0 \\1 & 0 & 0 & 0 & 0 & 0 \\0 & 0 & 0 & 0 & 0 & 1\\2 & -1 & 0 & 0 & 0 & 2\\ 0 & 0 & 0 & 1 & 2 & 0\end{array}\right],
$$
and
$$
G_{67} = \left[\begin{array}{rrrrrr}0 & 0 & 0 & 1 & 1 & 0 \\0 & 0 & 0 & 1 & 2 & 0 \\0 & 0 & 1 & 2 & 2 & 0 \\-2 & 1 & 0 & 0 & 0 & -2\\1 & -1 & 1 & 0 & -2 & 0\\ 0 & 0 & 0 & 1 & 2 & 1\end{array}\right],\quad G_{34}, G_{63}, G_{67}\in\Sp(6,\bZ).
$$
In other words, we have $N:=\Ad(G_{34})\cdot N_{34} = \Ad(G_{63})\cdot N_{63} = \Ad(G_{67})\cdot N_{67} $. $N$ induces a grading-adapted weight filtration
\begin{equation}\label{JMfilI2}
W^{I_2}: 0\xhookrightarrow{}\braket{e_1',e_2'}\xhookrightarrow{} \braket{e_1',e_2',e_3',e_6'}\xhookrightarrow{} V.
\end{equation}
Denote the closure of $N$ as the cone $\sigma$. We only need to show $\Ad(G)\cdot \sigma^\circ$ and $\sigma'^{\circ}$ have no non-trivial interior intersection, for all $G \in \Sp(6,\bZ)$. The possible interior intersections are at points with parameters $s_{1} > 0, t_{1} > 0, s_{2} > 0, t_{2} > 0$.

Now consider a Levi decomposition of $G$ with respect to the filtration (\ref{JMfilI2}). As before, the unipotent part of $G$ fixes $\sigma$, so they only have trivial intersections. Therefore, we only need to consider the Levi part of $G$
$$
G = \left[\begin{array}{rrrrrr}h_{11} & h_{12} & 0 & 0 & 0 & 0 \\h_{21} & h_{22} & 0 & 0 & 0 & 0 \\0 & 0 & h_{33} & 0 & 0 & h_{36} \\0 & 0 & 0 & h_{44} & h_{45} & 0\\0 & 0 & 0 & h_{54} & h_{55} & 0\\ 0 & 0 & h_{36} & 0 & 0 & h_{66}\end{array}\right],\quad G\in\Sp(6,\bZ).
$$
Assume there exist interior intersections. Then the system of equations $GN-N'G = 0_{6\times 6}$ have solutions. Now we reduce the degrees of freedom in $G$ by using the condition $GN-N'G = 0_{6\times 6}$. This gives
$$
G = \left[\begin{array}{rrrrrr}(-1)^{k_1}\sqrt{\frac{s_2}{s_1}} & 0 & 0 & 0 & 0 & 0 \\0 & (-1)^{k_2}\sqrt{\frac{t_2}{t_1}} & 0 & 0 & 0 & 0 \\0 & 0 & h_{33} & 0 & 0 & h_{36} \\0 & 0 & 0 & (-1)^{k_2}\sqrt{\frac{s_1}{s_2}} & 0 & 0\\0 & 0 & 0 & 0 & (-1)^{k_1}\sqrt{\frac{t_1}{t_2}} & 0\\ 0 & 0 & h_{36} & 0 & 0 & h_{66}\end{array}\right],
$$
Now, in order for $G\in \Sp(6,Z)$, one must have $(s_1,t_1) = (s_2,t_2)$. One can easily check for all $ k_{1,2}\in\bZ$, the resulting $G$ fixes all points in $\sigma^{\circ}$. Therefore, we can conclude, the $\Sp(6,\bZ)$ orbits for $\sigma_{34}$, $\sigma_{63}$ and $\sigma_{67}$ have no non-trivial interior intersection.

\appendix
\section{Proof of Proposition \ref{lemhosin}}\label{appa}
In this section, we give a proof of Proposition \ref{lemhosin}.
Our proof is constructive, which is suggested in the heuristic arguments in \cite{hosono2000local}. Note, in this appendix we will go back to using the standard symplectic form $Q^{st}$ defined in (\ref{strdsympl}).
For convenience, we recall the statement of the proposition.
\begin{proposition21}Let $\hat X_0$ be a smooth Calabi-Yau 3-fold. Given an arbitrary basis $\{T_i\}_{i=1}^r$ of $H^2(\hat X_0,\bZ)/torsion$,, define $\{K_j\}_{j=1}^r\subset H^4(\hat X_0,\bZ)/torsion$ by $\int_{\hat X_0}T_kK_j = \delta_{kj}$, and define $V$ via $\int_{\hat X_0} V=-1$. Then there exists an integral symmetric matrix $C:=(C_{ij})_{r\times r}$, such that 
\begin{equation}\label{hosobaisis2}
\{2,(2T_k-\sum_{l}C_{kl}K_l)td^{-1}_{\hat X_0},K_{r-j},V\},\quad j,k=1,\cdots,r
\end{equation}
gives an integral symplectic basis for $K_0(\hat X_0)\otimes_\bZ\bC$ under the inverse of the Chern character. Moreover, if one assumes homological mirror symmetry or that Conjecture \ref{conjihms} holds for $\hat X_0$, then 
\[
\{1,(T_k-\sum_{l}\frac 12 C_{kl}K_l)td^{-1}_{\hat X_0},K_{r-j},V\}    
\]
is a symplectic basis for $K_0(\hat X_0)/torsion$.
\end{proposition21}

In our proof, we will need the integral Hodge conjecture for smooth Calabi-Yau 3-fold, which is proved by Voisin \cite{voisin2006integral}.

\begin{theorem}[Theorem 2 in \cite{voisin2006integral}]\label{voisinint}
Let $\hat X_0$ be a smooth complex projective threefold such that it is either uniruled or satisfies
$$
K_{\hat X_0} \cong \mathcal O_{\hat X_0},\quad H^{2}(X,\mathcal O_X)=0.
$$
Then the Hodge conjecture is true for integral degree 4 Hodge classes on $\hat X_0$, i.e., elements in $H^{2,2}(\hat X_0)\cap H^4(\hat X_0,\bZ)$ are algebraic.
\end{theorem}
This theorem covers the Calabi-Yau 3-folds in our paper (Definition \ref{defcy3f}). Another classical result we will need is the following positive answer for the Borel-Haefliger problem \cite{borel1961classe}.

\begin{theorem}[Hironaka-Kleiman \cite{hironaka1968smoothing,kleiman1969geometry}]\label{thmhirokl}
For a smooth complex projective variety $X$ with dimension $n$, the Chow group $CH_d(X)$ of $d$-dimensional algebraic cycles are generated by smooth subvarieties if one of the following two condition holds:
\begin{enumerate}
    \item $d\leq\max\{3,(n-1)/2\}$,
    \item $n-d = 1,2$, and $d<(n+2)/2$.
\end{enumerate}
\end{theorem}

Now, apply the above two theorems to Calabi-Yau 3-fold, one arrives at the following corollary.
\begin{corollary}\label{corshc}
For a smooth Calabi-Yau 3-fold $\hat X_0$, $\mathrm{Hdg}_\bZ^p:=H^{2p}(\hat X_0,\bZ)\cap H^{p,p}(\hat X_0)$ is generated by smooth ($3-p$)-dimensional subvarieties of $\hat X_0$, for $p=0,1,2,3$.
\end{corollary}
\begin{proof}
Since the cycle map factor through the Chow groups \cite{voisin2003hodge}, we can replace $CH_d(\hat X_0)$ by $H^{n-d}(\hat X_0,\bZ)$ in Theorem \ref{thmhirokl} if the integral Hodge conjecture holds in degree $d$.

The group $\mathrm{Hdg}_\bZ^0$ is generated by the fundamental class which is represented by the smooth Calabi-Yau 3-fold itself. Meanwhile the group $\mathrm{Hdg}_\bZ^3$ is generated by the point class whose representatives are trivially smooth.

For $p=1$, by the Lefschetz theorem on (1,1)-class (c.f., \cite{voisin2002hodge}), the group $\mathrm{Hdg}^1_\bZ$ is generated by algebraic 2-cycles, which can be represented by smooth hypersurfaces by the case (2) in Theorem \ref{thmhirokl} for $n=3,d=2$.

Finally, for $p=2$, this is an immediate consequence of Theorem \ref{voisinint}, and Theorem \ref{thmhirokl} applied to $n=3$ and $d=1$. In particular, both of the two conditions in Theorem \ref{thmhirokl} cover the case $n=3$ and $d=1$. 
\end{proof}

We recall that the Grothendieck-Riemann-Roch (GRR) theorem (c.f. \cite{huybrechts2006fourier}) apply to a closed embedding $i:Y\rightarrow X$ between smooth quasi-projective varieties takes the simpler form
\begin{equation}\label{GRR}
ch(i_*\mathcal F)td(X) = i_*(ch(\mathcal F) td(Y)), 
\end{equation}
where $\mathcal F$ is a coherent sheaf on $Y$. This is due to the fact that $i_*$ is exact.

Now we will begin our proof of Proposition \ref{lemhosin}. From now on, all the integral cohomology and K-theory groups will be understood as modulo torsion. We first prove the following lemma.

\begin{lemma}\label{lemaexie}
For a Calabi-Yau 3-fold $\hat X_0$, for all $\alpha\in H^{k}(X,\bZ), k = 0,4,6$, there exists $E_\alpha\in K(\hat X_0)$, such that $ch(E_\alpha) = \alpha$.
\end{lemma}
\begin{proof}
We first observe $ch([\mathcal O_{\hat X_0}]) = 1$. Also, by GRR formula (\ref{GRR}), we have $ch(-[\mathcal O_p]) = -\mathrm{vol}_{\hat X_0} =: V $, where $\mathcal O_p$ is the skyscraper sheaf supported on $p\in \hat X_0$. Therefore, the cases  $k=0,6$ are proven. 

For $k=4$, by Corollary \ref{corshc}, we have $H^4(\hat X_0,\bZ)$  is generated by smooth curves in $\hat X_0$. Denote a set of smooth generators of $H^4(\hat X_0,\bZ)$ as $\{[C_l]\}_{l=1,\cdots,r}$, where $C_l$ are smooth curves in $\hat X_0$. For any $l=1,\cdots,r$, consider the a rank 1 coherent sheaf $\mathcal F_l$ on $C_l$, and denote the closed embedding $i_l:C_l\xhookrightarrow{}\hat X_0$. Recall the Calabi-Yau condition implies $c_1(\hat X_0) = 0$, and $td(\hat X_0) = 1+\frac{c_2(\hat X_0)}{12}$. Using the GRR formula (\ref{GRR}), we have
$$
ch(i_{l*}\mathcal F_l) = i_{l*}(ch(\mathcal F)td(C_l))td(\hat X_0)^{-1} = i_{l*}(ch(\mathcal F)td(C_l))\left(1-\frac{c_2(\hat X_0)}{12}\right),
$$
Since $i_{l*}$ preserves codegree, the above equation simplifies to
$$
ch(i_{l*}\mathcal F_l) = i_{l*}(ch(\mathcal F)td(C_l)),
$$
and we have $ch_0(i_{l*}\mathcal F_l) = ch_1(i_{l*}\mathcal F_l)=0$. Moreover, we have
$$
ch_2(i_{l*}\mathcal F_l) = i_{l*}(1),\quad ch_3(i_{l*}\mathcal F_l) = i_{l*}\left(c_1(\mathcal F)+\frac{c_1(C_l)}{2}\right).
$$
By definition, we have $i_{l*}(1) = [C_l]$. Now, by Hirzebruch–Riemann–Roch formula (\ref{releuler}), we have
$$
\chi(\mathcal O_{C_l}) = \chi(\mathcal O_{C_l},\mathcal O_{C_l}) = \int_{C_l}  td(C_l) = \int_{C_l} \frac{c_1(C_l)}{2},
$$
which is an integer by definition. Therefore, we have $\frac{c_1(C_l)}{2}\in H^2(C_l,\bZ)$. Next, since $h^{0,2}(C_l) = 0$, similar to the calculation in  (\ref{esles}), we have $c_1:H^1(C_l,\mathcal O^*_{C_l})\rightarrow H^2(C_l,\bZ)$ is surjective. Therefore, we can find a line bundle $L_l$ on $C_l$, such that $c_1(L_l) = -\frac{c_1(C_l)}{2}$. Now, if we take $\mathcal F_l = L_l$, then we have
$$
ch(i_{l*}\mathcal F_l) = [C_l],
$$
which concludes our proof.
\end{proof}

\begin{remark}
The above lemma shows that $ch$ induces surjections on $H^{k}(\hat X_0,\bZ)$, for $k=0,4,6$. 
\end{remark}

\begin{proof}[Proof of Proposition \ref{lemhosin}]
First, recall we have $ch(\mathcal F^*) = \sum_{k}(-1)^kch_k(\mathcal F)$, then
the basis (\ref{hosobaisis2}) is apparently symplectic with respect to $Q^{st}$ in (\ref{strdsympl}).

Now by Corollary \ref{corshc}, there exist smooth hypersurfaces $\{D_l\}_{l=1,\cdots, r}$, such that $[D_l]$ generates $H^2(\hat X_0,\bZ)$. Denote the closed embedding as $i_l:D_l\xhookrightarrow{} \hat X_0$. Similar to the discussion in Lemma \ref{lemaexie}, we consider the rank 1 sheaves $\mathcal F_l$ on $D_l$, and calculate the Chern character of the torsion sheaf $i_{*l}\mathcal F_l$. Again, by GRR formula (\ref{GRR}), we have
$$
ch(i_{l*}\mathcal F_l) = i_{l*}(ch(\mathcal F_l)td(D_l))td(\hat X_0)^{-1}.
$$
Now, denote $td = 1+td_1+td_2+\cdots$, with $td_n$ has homogeneous degree $2n$. Expanding by degree, we have
\begin{equation}
    \begin{split}
    &ch_0(i_{l*}\mathcal F_l)=0,\quad ch_1(i_{l*}\mathcal F_l) = [D_l],\quad ch_2(i_{l*}\mathcal F_l) = i_{l*}(c_1(\mathcal F_l)+\frac12c_1(D_l)),\\
    &
    ch_3(i_{l*}\mathcal F_l) = i_{l*}(ch_2(\mathcal F_l) + ch_1(\mathcal F_l)td_1(D_l)+td_2(D_l))-[D_l]\cdot td_2(\hat X_0).
    \end{split}
\end{equation}
Furthermore, by the Hirzebruch–Riemann–Roch formula (\ref{releuler}), we have
$$
\chi(\mathcal F_{l}) = \int_{D_l}ch(\mathcal F_{l})td(D_l) = \int_{D_l}ch_2(\mathcal F_l) + ch_1(\mathcal F_l)td_1(D_l)+td_2(D_l).
$$
This implies $ch_2(\mathcal F_l) + ch_1(\mathcal F_l)td_1(D_l)+td_2(D_l)$ lives in $H^4(D_l,\bZ)$, which in turn, tells us $i_{l*}(ch_2(\mathcal F_l) + ch_1(\mathcal F_l)td_1(D_l)+td_2(D_l))$ lives in $H^6(\hat X_0,\bZ)$. Now, by Lemma \ref{lemaexie}, there exists $ n\in\bZ$, such that $i_{l*}(ch_2(\mathcal F_l) + ch_1(\mathcal F_l)td_1(D_l)+td_2(D_l)) = n\cdot ch(\mathcal O_p)$. Therefore, if we define $K_0(\hat X_0)$ classes $E_l:=[i_{l*}\mathcal F_l]-n[\mathcal O_p]$, then we have
$$
ch(E_l) = ([D_l] + R_l)td^{-1}(\hat X_0), 
$$
where $R_l:=i_{l*}(c_1(\mathcal F_l)+\frac12c_1(D_l)) \in \frac 12 H^4(\hat X_0,\bZ)$.

Now, by Lemma \ref{lemaexie}, we can take $F_l\in K_0(\hat X_0),l=1,\cdots,r$ such that $ch(F_l) = K_l$. Since $\{[D_l]\}_{l=1,\cdots,r}$ generates $H^2(\hat X_0,\bZ)$ for each $l=1,\cdots,r$ there exist integers $a_{lm}, m=1,\cdots,r$, such that 
$$
ch(\sum_{m=1}^ra_{lm}E_m) = (T_l+R_l')td^{-1}(\hat X_0)=T_l+R_l'- \frac{1}{12} T_l\cdot c_2(\hat X_0),
$$
with $ R'_l:= \sum_{m=1}^ra_{lm}R_m\in \frac{1}{2}H^4(\hat X_0,\bZ)$. Denote $E_l':= \sum_{m=1}^ra_{lm}E_m$. Furthermore, expand $2R_l'$ in the basis $\{K_j\}_{j=1,\cdots,r}$, as $2R'_l = \sum_{j=1}^rb_{lj}K_j$, we have $b_{lj}\in \bZ$. Therefore, by adding to each K-theory class $2E_l'$ a suitable (non-unique) integral linear combination of $F_j$, we can make $(b_{ij})_{i,j=1,\cdots,r}$ into an integral symmetric matrix. For example, we can make the change $2E_l'\rightarrow 2E_l'+\sum_{l}b_{jl}K_j$, $\forall l=1,\cdots,r$. We denote the result K-theory class as $E''_l$, then we have
$$
ch(2E''_l) = (2T_l+\sum_{j=1}^rC_{lj}K_j)td^{-1}(\hat X_0),\quad C_{ij} = C_{ji}\in\bZ,\forall i,j =1\cdots,r.
$$
Now consider the following set of generators of $K_0(\hat X_0,\bZ)$, 
\begin{equation}\label{basisefina}
\{2[\mathcal O_{\hat X_0}],2E''_l,F_m,[\mathcal O_p]\}_{l,m=1,\cdots,r}.
\end{equation}
Again by the Hirzebruch-Riemann-Roch formula, we easily show the intersection pairing matrix of relative Euler characteristic $\chi$ is $2Q^{st}$. For example, the intersection pairing between $E_l''$ and $E_m''$ is
\[
\begin{split}
\chi(2E_l'',2E_m'') & = \int_{\hat X_0} (-2T_l+\sum_{j=1}^rC_{jl}K_j)(2T_m+\sum_{j=1}^rC_{mj}K_j)td^{-1}(\hat X_0) \\
& = -2C_{ml}+2C_{lm} = 0.
\end{split}
\]
Therefore, taking the Chern character of (\ref{basisefina}), we arrive at the desired basis.

Now, we assume homological mirror symmetry or Conjecture \ref{conjihms} holds for $\hat X_0$. Denote the smooth mirror Calabi-Yau 3-fold as $X_b$. Then we have $(K_0(X),\chi)\cong (H^3(\hat X_p,\bZ),Q)$, where $Q$ is the intersection pairing. Since the intersection pairing is a perfect pairing, we have $\chi$ as a perfect pairing. 

Again take the K-theory classes $F_m$ to be those with $ch(F_{m}) = K_m$. Then we can extend the set $\{F_m,[\mathcal O_{\hat X_0}],[\mathcal O_p]\}$ into an symplectic basis of $K_0(X)$. By the Hirzebruch-Riemann-Roch formula, the Chern character of such a basis has to take the form
\[
\{1,(T_k-\sum_{l}\frac 12 C_{kl}K_l)td^{-1}_{\hat X_0},K_j,V\},
\]
with  $(C_{ij})_{r\times r}$ rational symmetric. Now since $\{K_i\}$ form a basis of $H^4(\hat X_0,\bZ)$, and we have shown when $C$ is integral symmetric, this gives a half-integral basis, then by the uniqueness up to a modification by matrix $B$ as discussed in Remark \ref{rmkaftprop}, $C$ has to be integral, which completes the proof.

\end{proof}

\section{Proof of Proposition \ref{prop2.2sec2}}\label{appb}
In this section, we give a detailed proof of Proposition \ref{prop2.2sec2}. We recall the statement for convenience.
\begin{proposition22}
On the affine chart $\mathrm{Spec}\,\bC[z_1,z_2]$ of the moduli space of the mirror octic. Fix a base point $b$ near the origin (the MUM point). An integral basis of $H^3(X_b,\bC)$ at $b$ is given as the mirror of the basis (\ref{hosbaourca}) with $C_{12} = 4$, $C_{22}=0$.\footnote{This choice is made in order to match those in \cite{candelas1994mirror1}.}  Denote the monodromy matrices around the divisors $D_{(1,0)}$ and $D_{(0,1)}$ as $T_{m_1}$ and $T_{m_2}$, then we have
$$
T_{m_1} = 
\left[\begin{array}{rrrrrr}1 & -1 & 0 & 6 & 4-\frac{C_{11}}{2} & 0 \\0 & 1 & 0 & \frac{-C_{11}}{2}-4 & -8 & -4 \\0 & 0 & 1 & -4 & -4 & 0 \\0 & 0 & 0 & 1 & 0 & 0\\0 & 0 & 0 & 1 & 1 & 0\\ 0 & 0 & 0 & 0 & 0 & 1\end{array}\right],\,
T_{m_2} = 
\left[\begin{array}{rrrrrr}1 & 0 & -1 & 2 & -2 & 0 \\0 & 1 & 0 & -2 & -4 & 0 \\0 & 0 & 1 & 0 & 0 & 0 \\0 & 0 & 0 & 1 & 0 & 0\\0 & 0 & 0 & 0 & 1 & 0\\ 0 & 0 & 0 & 1 & 0 & 1\end{array}\right].
$$
\end{proposition22}
\begin{proof}
Since the origin is the MUM point, mirror symmetry holds at $b$. From the second part of Lemma \ref{lembaourc}, we can take the basis (\ref{newsympbasis}) to be 
\begin{equation}\label{basisapendis}
\{1,(H-\frac12C_{11}h-2l)td^{-1}(\hat X_0),(L-\frac12C_{12}h-\frac12C_{22}l)td^{-1}(\hat X_0),-\frac18 H^3,h,l\}.
\end{equation}
Then expanding  Hosono's $w$-function (\ref{wformourc}) in terms of the above basis, we have
\begin{equation}\label{wbasexpan}
\begin{split}
    w & = w_0\cdot 1+ w_{1}^{(1)}\cdot(H-\frac12 C_{11}h-2l)td^{-1}_{\hat X_0}+w_2^{(1)}\cdot(L-2h)td^{-1}(\hat X_0) \\
    & + w_1^{(2)}\cdot h+ w_2^{(2)}\cdot l+w^{(3)}\cdot(-\frac18 H^3),\quad td^{-1}(\hat X_0) = 1-2l-\frac{14}{3}h.
\end{split}
\end{equation}
where recall $td^{-1}_{\hat X_0} = 1-2l-\frac{14}{3}h$. 
On the other hand, the Taylor expansion of $w$-function gives
\begin{equation*}
\begin{split}
    w & = w|_{H,L=0}\cdot 1+ \partial_H w|_{H,L=0}\cdot H+\partial_L w|_{H,L=0}\cdot L + \frac{1}{2}\partial^2_H w|_{H,L=0}\cdot H^2 \\
    & + \frac{1}{2}\partial^2_L w|_{H,L=0}\cdot L^2
    +\partial_H\partial_L w|_{H,L=0}\cdot HL +\frac{1}{6}\partial^3_H w|_{H,L=0}\cdot H^3\\
    & + \frac{1}{2}\partial^2_H\partial_L w|_{H,L=0}\cdot H^2L +\frac{1}{2}\partial_H\partial^2_L w|_{H,L=0}\cdot HL^2,
    \end{split}
\end{equation*}
Comparing the two expressions of $w$ and using the intersection numbers in \eqref{intnum}, we have the $w_0=w|_{H,L=0} $, and
\begin{equation}\label{reltwows}
\begin{split}
   & w_1^{(1)} = \partial_Hw|_{H,L=0},  \quad  w_2^{(1)} = \partial_Lw|_{H,L=0},\\
  & w_1^{(2)} = \frac{1}{2}C_{11}\partial_Hw|_{H,L=0}+2\partial_Lw|_{H,L=0}+ 4\partial_H^2w|_{H,L=0}+2\partial_H\partial_Lw|_{H,L=0},\\
  & w_2^{(2)} = 2\partial_Hw|_{H,L=0} +2\partial_H^2w|_{H,L=0},\\
  & w^{(3)} = -\frac{14}{3}\partial_Hw|_{H,L=0}-2\partial_Lw|_{H,L=0}-\frac 43 \partial_H^3w|_{H,L=0}-2\partial_H^2\partial_Lw|_{H,L=0}.
\end{split}
\end{equation}

The monodromy of the partial differentials of $w$ can be calculated as follows. First, we notice the terms that are ratios of Gamma functions in (\ref{wformourc}) are holomorphic. So we have
$$
\partial_H w|_{H,L=0} = \frac{1}{2\pi i}\log z_1\cdot w_0+\widetilde{w_{1}},\quad \partial_L w|_{H,L=0} = \frac{1}{2\pi i}\log z_2\cdot w_0+\widetilde{w_{2}},
$$
where $\widetilde{w_{1}}$ and $\widetilde{w_{2}}$ are holomorphic functions that are defined by replacing the Gamma ratios in $w$ by their first order partial derivative with respect to $H$ and $L$, respectively. Then the monodromy around $D_{(1,0)}$ ($z_1 =0$) is
$$
\partial_H w|_{H,L=0}\rightarrow \partial_H w|_{H,L=0} + w_0,\quad \partial_L w|_{H,L=0}\rightarrow\partial_L w|_{H,L=0}.
$$
Similarly about $z_2= 0$, we have
$$
\partial_H w|_{H,L=0}\rightarrow \partial_H w|_{H,L=0},\quad \partial_L w|_{H,L=0}\rightarrow\partial_L w|_{H,L=0}+w_0.
$$
Now, for $\partial_H^2w|_{H,L=0}$, using the same argument, we can rewrite it as
\[
\partial_H^2w|_{H,L=0}  = \left(\frac{\log(z_1)}{2\pi i}\right)^2w_0+2\cdot \frac{\log(z_1)}{2\pi i} \widetilde{w_1} +\widetilde{w_{11}}.
\]
Here $\widetilde{w_{11}}$ is another holomorphic function, which is defined by replacing the Gamma ratios in $w$ by their second-order partial derivative with respect to $H$. Then the monodromy about $z_1=0$, $z_2=0$ are
\[
\begin{split}
&\partial_H^2w|_{H,L=0} \rightarrow \partial_H^2w|_{H,L=0} + 2\partial_Hw|_{H,L=0} +w_0,\\
&\partial_H^2w|_{H,L=0} \rightarrow \partial_H^2w|_{H,L=0}, 
\end{split}
\]
respectively. The monodromies for the rest of the partial differentials are calculated in the same manner. First, we have
\begin{align*}
    \partial_H\partial_L w|_{H,L=0}  = & \frac{\log(z_1)\log(z_2)}{(2\pi)^2}w_0+\frac{\log(z_2)}{2\pi i}\widetilde{w_1}+\frac{\log(z_1)}{2\pi i}\widetilde{w_2}+\widetilde{w_{12}},\\
    \partial_H^3w|_{H,L=0}  = & \left(\frac{\log(z_1)}{2\pi i}\right)^3w_0+ 3\cdot\left(\frac{\log(z_1)}{2\pi i}\right)^2\widetilde{w_1} +3\cdot \frac{\log(z_1)}{2\pi i}\widetilde{w_{11}} +\widetilde{w_{111}},\\
    \partial_H^2\partial_Lw|_{H,L=0}  = & \frac{(\log(z_1))^2\log(z_2)}{(2\pi i)^3} w_0 + 2\cdot \frac{\log(z_1)\log(z_2)}{(2\pi i)^2}\widetilde{w_1}+ 
    \frac{\log(z_1)^2}{(2\pi i)^2}\widetilde{w_2},\\
    & +\frac{\log(z_2)}{2\pi i}\widetilde{w_{11}} + 2\cdot \frac{\log(z_1)}{2\pi i}\widetilde{w_{12}}+ \widetilde{w_{112}},
\end{align*}
where the definition for the holomorphic $\widetilde{w}$ functions are parallel to those above. Then the monodromies can be read off from the logarithmic terms as follows
$$
\left[
\begin{array}{l}
     w_0\\
     \partial_Hw\\
     \partial_Lw\\
     \partial_H^2w\\
     \partial_H\partial_Lw\\
     \partial_H^3w\\
     \partial_H^2\partial_Lw
\end{array}
\right] \xrightarrow{z_1\rightarrow e^{2\pi i}z_1} \left[
\begin{array}{ccccccc}
     1 & 0 & 0 & 0 & 0 & 0 & 0  \\
     1 & 1 & 0 & 0 & 0 & 0 & 0 \\
     0 & 0 & 1 & 0 & 0 & 0 & 0 \\
     1 & 2 & 0 & 1 & 0 & 0 & 0 \\
     0 & 0 & 1 & 0 & 1 & 0 & 0 \\
     1 & 3 & 0 & 3 & 0 & 1 & 0 \\
     0 & 0 & 1 & 0 & 2 & 0 & 1 \\
\end{array}
\right]\cdot
\left[
\begin{array}{l}
     w_0\\
     \partial_Hw\\
     \partial_Lw\\
     \partial_H^2w\\
     \partial_H\partial_Lw\\
     \partial_H^3w\\
     \partial_H^2\partial_Lw
\end{array}
\right],
$$
$$
\left[
\begin{array}{l}
     w_0\\
     \partial_Hw\\
     \partial_Lw\\
     \partial_H^2w\\
     \partial_H\partial_Lw\\
     \partial_H^3w\\
     \partial_H^2\partial_Lw
\end{array}
\right] \xrightarrow{z_2\rightarrow e^{2\pi i}z_2} \left[
\begin{array}{ccccccc}
     1 & 0 & 0 & 0 & 0 & 0 & 0  \\
     0 & 1 & 0 & 0 & 0 & 0 & 0 \\
     1 & 0 & 1 & 0 & 0 & 0 & 0 \\
     0 & 0 & 0 & 1 & 0 & 0 & 0 \\
     0 & 1 & 0 & 0 & 1 & 0 & 0 \\
     0 & 0 & 0 & 0 & 0 & 1 & 0 \\
     0 & 0 & 0 & 1 & 0 & 0 & 1 \\
\end{array}
\right]\cdot
\left[
\begin{array}{l}
     w_0\\
     \partial_Hw\\
     \partial_Lw\\
     \partial_H^2w\\
     \partial_H\partial_Lw\\
     \partial_H^3w\\
     \partial_H^2\partial_Lw
\end{array}
\right],
$$
where all partial differentials are understood as evaluating at $H=L=0$. We note the monodromy of the partial differentiations are independent of our specific question and are essentially binomial expansions.

Then the monodromy of the coefficients in (\ref{wbasexpan}) can be derived from those of the partial differentials by the relation \eqref{reltwows}. The results are
$$
\left[
\begin{array}{l}
     w_0\\
     w_1^{(1)}\\
     w_2^{(1)}\\
     w^{(3)}\\
     w_1^{(2)}\\
     w_2^{(2)}\\
\end{array}
\right] \xrightarrow{z_1\rightarrow e^{2\pi i}z_1} \left[
\begin{array}{ccccccc}
     1 & 0 & 0 & 0 & 0 & 0 \\
     1 & 1 & 0 & 0 & 0 & 0 \\
     0 & 0 & 1 & 0 & 0 & 0 \\
     -6 & \frac{C_{11}}{2}-4 & 0 & 1 & -1 & 0 \\
     \frac{C_{11}}{2}+4 & 8 & 4 & 0 & 1 & 0 \\
     4 & 4 & 0 & 0 & 0 & 1 \\
\end{array}
\right]\cdot
\left[
\begin{array}{l}
     w_0\\
     w_1^{(1)}\\
     w_2^{(1)}\\
     w^{(3)}\\
     w_1^{(2)}\\
     w_2^{(2)}\\
\end{array}
\right],
$$
and
$$
\left[
\begin{array}{l}
     w_0\\
     w_1^{(1)}\\
     w_2^{(1)}\\
     w^{(3)}\\
     w_1^{(2)}\\
     w_2^{(2)}\\
\end{array}
\right] \xrightarrow{z_1\rightarrow e^{2\pi i}z_1} \left[
\begin{array}{ccccccc}
     1 & 0 & 0 & 0 & 0 & 0 \\
     0 & 1 & 0 & 0 & 0 & 0 \\
     1 & 0 & 1 & 0 & 0 & 0 \\
     -2 & 2 & 0 & 1 & 0 & -1 \\
     2 & 4 & 0 & 0 & 1 & 0 \\
     0 & 0 & 0 & 0 & 0 & 1 \\
\end{array}
\right]\cdot
\left[
\begin{array}{l}
     w_0\\
     w_1^{(1)}\\
     w_2^{(1)}\\
     w^{(3)}\\
     w_1^{(2)}\\
     w_2^{(2)}\\
\end{array}
\right],
$$
where we have ordered the basis according to \eqref{basisapendis}.
Finally, by the monodromy invariance of the $w$-function (Corollary \ref{corinv}), the monodromy $T_{m_1}$ and $T_{m_2}$ for the integral symplectic basis \eqref{basisapendis} is just the transpose inverse of the above matrices.
\end{proof}

\input{Revision4.0.bbl}

\end{document}

%% file: Revision4.0.bbl
\newcommand{\etalchar}[1]{$^{#1}$}